\documentclass[a4paper,reqno]{amsart}

\usepackage{amsmath,amsthm,amssymb}
\usepackage{mathrsfs}
\usepackage[shortlabels]{enumitem}
\usepackage{graphicx}
\usepackage[font=small]{caption}
\usepackage{xcolor}
\usepackage{url}

\newcommand{\N}{\mathbb{N}}
\newcommand{\R}{{\mathbb{R}}}
\newcommand{\C}{{\mathbb{C}}}

\newcommand{\dd}{{{\rm d}}}
\newcommand{\ii}{{\rm i}}
\newcommand{\e}{{\rm e}}

\newcommand{\cf}{\emph{cf.}}
\newcommand{\ie}{{\emph{i.e.}}}
\newcommand{\eg}{{\emph{e.g.}}}

\newcommand{\ov}{\overline}

\newcommand{\la}{\lambda}

\newcommand{\eps}{\varepsilon}

\newcommand{\spd}{\sigma_{\rm disc}}

\newcommand{\spp}{\sigma_{\rm p}}

\newcommand{\essinf}{\operatorname*{ess \,inf}}

\newcommand{\Dom}{{\operatorname{Dom}}}

\newcommand{\Ran}{{\operatorname{Ran}}}

\renewcommand{\Re}{\operatorname{Re}}
\renewcommand{\Im}{\operatorname{Im}}

\newcommand{\sgn}{\operatorname{sgn}}
\newcommand{\supp}{\operatorname{supp}}
\newcommand{\Tr}{\operatorname{Tr}}

\newcommand{\BigO}{\mathcal{O}}

\newcommand{\Rd}{\mathbb{R}^d}

\newcommand{\Dti}{-\partial_x^2}

\newcommand{\ls}{\lesssim}
\newcommand{\gs}{\gtrsim}

\theoremstyle{plain}

\newtheorem{theorem}{Theorem}[section]
\newtheorem{lemma}[theorem]{Lemma}

\newtheorem{proposition}[theorem]{Proposition}
\newtheorem{corollary}[theorem]{Corollary}

\theoremstyle{definition}
\newtheorem{example}[theorem]{Example}
\newtheorem{remark}[theorem]{Remark}
\newtheorem{asm-sec}[theorem]{Assumption}

\newcommand\cB{\mathcal B}
\newcommand\cC{\mathcal C}

\newcommand\cP{\mathcal P}

\newcommand\cS{\mathcal S}
\newcommand\cT{\mathcal T}

\newcommand{\Ai}{\operatorname{Ai}}

\usepackage{mathtools}
\mathtoolsset{showonlyrefs}

\usepackage{chngcntr}
\counterwithin{figure}{section}

\newcommand\Sa{S_{\rm A}}

\newcommand\ec{\eps_{\rm crit}}

\begin{document}

\numberwithin{equation}{section}

\graphicspath{{Figures/}}
	
	\title[Diverging eigenvalues in domain truncations]{Diverging eigenvalues in domain truncations of Schr\"odinger operators with complex potentials}
	
	\author{Iveta Semor\'adov\'a}
	\address[Iveta Semor\'adov\'a]{
		Faculty of Nuclear Science and Physical Engineering, Czech Technical University in Prague,
		Czech Republic \& Nuclear Physics Institute, The Czech Academy of Science, Hlavn\'i 130, 25068 \v Re\v z, Czech Republic }
	\email{Iveta.Semoradova@fjfi.cvut.cz}
	
\author{Petr Siegl}
\address[Petr Siegl]{School of Mathematics and Physics, Queen's University Belfast, University Road, BT7 1NN Belfast, UK}

\email{p.siegl@qub.ac.uk}

\thanks{The work of I.S.~was supported by the Grant Agency of the Czech Technical University in Prague, grant No.~SGS19/183/OHK4/3T/14. It was initiated at the International Centre for Theoretical Sciences (ICTS), Bengaluru, India, during a visit of authors in the program - Non-Hermitian Physics - PHHQP XVIII in 2018 (Code: ICTS/nhp2018/06). Further, I.S.~would like to thank Queen's University Belfast for the hospitality during her Erasmus+ stay there is September 2019 -- July 2020.}


	
\date{July 20, 2021}
	
\begin{abstract}
Diverging eigenvalues in domain truncations of Schr\"odinger operators with complex potentials are analyzed and their asymptotic formulas are obtained. Our approach also yields asymptotic formulas for diverging eigenvalues in the strong coupling regime for the imaginary part of the potential.
\end{abstract}

\maketitle

\section{Introduction}

Approximating spectra of non-self-adjoint partial differential operators is a major challenge in spectral analysis even in the case of purely discrete spectra. In this paper, we focus on the spectral convergence of domain truncations for multidimensional Schr\"odinger operators $- \Delta + Q$ in $L^2(\Omega)$ with $\Omega \subset \Rd$ and a \emph{complex} potential $Q :\Omega \to \C$, the study of which was initiated in \cite{Brown-2004-24}. As an example, which we use here to indicated the questions addressed in this paper and our new results, consider the imaginary oscillator 
\begin{equation}\label{T.im.intro}
T_\infty = -\Delta_{\rm D} + \ii |x|^2	
\end{equation}
in $L^2(\Omega)$ with $\Omega = \R^d \setminus \ov{B_1(0)}$ and the Dirichlet boundary condition imposed at~$\partial \Omega$. A possible sequence of truncations are  $T_n = - \Delta_{\rm D} + Q$ in $L^2(\Omega_n)$ with $\Omega_n:=B_{s_n}(0) \cap \Omega$, $s_n \nearrow + \infty$, and Dirichlet boundary conditions at $\partial \Omega_n$,~$n \in \N$. 
The general goal is to determine the relation of spectra of $\{T_n\}$ and $T_\infty$.

It was established in \cite{Boegli-2017-42} that for potentials $Q$ with $\Re Q \geq 0$, $|Q(x)| \to + \infty$ as $|x| \to \infty$ and satisfying suitable regularity conditions (hence in particular for \eqref{T.im.intro}), the truncations $T_n$ converge to $T_\infty$ as $n \to \infty$ in the norm resolvent sense. Consequently the approximation is \emph{spectrally exact}, \ie~all eigenvalues of $T_\infty$ are approximated by eigenvalues of $T_n$ and there is no pollution (there are no finite accumulation points of eigenvalues of $\{T_n\}$ which are not eigenvalues of $T_\infty$), see~\eg~\cite{Brown-2004-24}. The results in \cite{Boegli-2017-42} include more general cases with different boundary conditions and with potentials having negative real part controlled by $\Im Q$, moreover, the convergence rates of eigenvalues were related to the decay of eigenfunctions of $T_\infty$. (For  further works on spectral approximations using limiting essential spectra and essential numerical ranges see \cite{Boegli-2018-8,Boegli-2019-40,Boegli-2020-279}.)

It is however crucial to notice that the spectral exactness does not exclude eigenvalues of the truncations  escaping to infinity as $n \to \infty$, which is in our example~\eqref{T.im.intro} illustrated in Figure~\ref{fig:ix2.intro}. In fact, many other examples, in particular the one-dimensional imaginary cubic oscillator ($Q(x) = \ii x^3$) examined in \cite{Boegli-2017-42,Guenther-2019-arxiv}, suggest that diverging eigenvalues are rather typical and exhibit quite regular patterns.
\begin{figure}
	\includegraphics[width=0.6 \textwidth]{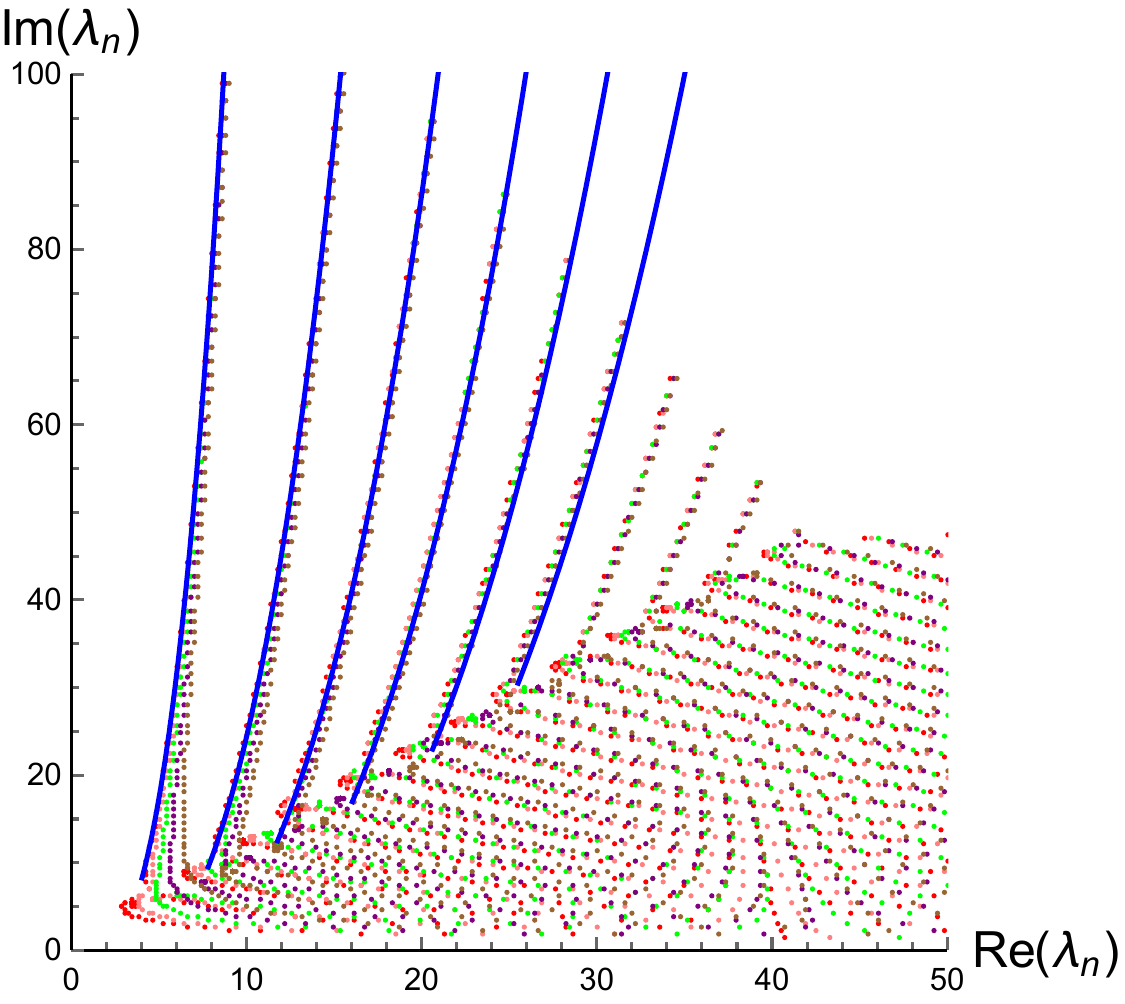}
	\caption{Trajectory in $\C$ of eigenvalues of truncations $T_n$ of $T_\infty$ from \eqref{T.im.intro} for $d=3$, $s_n = 0.1n$, $n=15,16,\dots,115$; eigenvalues for $l=1$ (red), $l=2$ (pink), $l=3$ (green), $l=4$ (purple), $l=5$ (brown) only are plotted; see \eqref{la.knl.rad.intro} and Section~\ref{ssec:Q.radial} for details. The numerics illustrates the spectral exactness (the clusters of eigenvalues at the ray $\e^{\ii \pi/4} \R_+$) as well as the eigenvalues escaping to infinity along the blue curves, see~\eqref{la.knl.rad.intro} for asymptotic formulas.}
	\label{fig:ix2.intro}
\end{figure}
The extreme case are the truncations of the imaginary Airy operator $T_\infty = \Dti + \ii x$ in $L^2(\R)$ to $T_n$ in $L^2((-s_n,s_n))$ where, due to the established spectral exactness and the fact the spectrum of $T_\infty$ is empty, \emph{all} eigenvalues of $T_n$ escape to infinity (see Example~\ref{ex:Airy} for details, more results can be found in \cite[Thm.~3.1]{Beauchard-2015-21}). 

The main goal of this paper is to analyze the diverging eigenvalues employing an operator convergence and a localization strategy inspired by \cite[Thm.~3.1]{Beauchard-2015-21}. 
This approach yields also improvements of the spectral exactness results in \cite{Boegli-2017-42},
moreover, it is applicable in some problems with strongly coupled $\Im Q$ 
\begin{equation}\label{Tg.intro}
	T_g = -\Delta + \Re Q + \ii g \Im Q, \qquad g \to + \infty,
\end{equation}
arising in various contexts like enhanced dissipation, \cf~\cite{Gallagher-2009-2009,Schenker-2011-18}, or $\cP\cT$-symmetric phase transitions, \cf~\cite{Caliceti-2014-19,Baker-2020-61} (see Section~\ref{sec:strong.c} for details). 

In particular in example~\eqref{T.im.intro}, our results show that the truncations $T_n$ contain asymptotically the diverging eigenvalues
\begin{equation}\label{la.knl.rad.intro}
	\la_{k,n,l}=(2s_n)^\frac23 \left(\overline{\nu_k}+ \BigO_{k,l}\left(s_n^{-\frac{4}{3}}\right) \right)+\ii s_n^2,\quad n\to\infty,
\end{equation}
where $\{\nu_k\}$ are eigenvalues of the imaginary Airy operator $\Dti + \ii x$ in $L^2(\R_+)$ with Dirichlet boundary condition at $0$ (see Section~\ref{ssec:Q.radial}, Example~\ref{ex:Airy.cone}, Figures~\ref{fig:ix2.intro} and \ref{fig:ix2.rad}). To be more precise, by writing that spectra of operators $\{A_n\}$ contain asymptotically the eigenvalues $\{\la_{k,n}\}_k$ we mean that 
\begin{equation}
	\forall k\in\N, \quad \exists n_k \in \N,  \quad \forall n>n_k ,\quad \lambda_{k,n}\in \sigma (A_n).  
\end{equation}
Figures~\ref{fig:ix2.intro} and \ref{fig:ix2.rad}, as well as other examples in Section~\ref{sec:ex.div}, exhibit a good correspondence of numerics and obtained asymptotics. Moreover, they suggest that all diverging eigenvalues are described (in these examples); however, this remains open. 

The improvements in the spectral exactness lie in finding the convergence rate for the resolvent norm in terms of the decay of $|Q|^{-1}$, establishing convergence in Schatten norms and estimating the constants in the convergence rates. Moreover, we can also treat truncations of operators with non-empty essential spectrum like $\Dti + x \e^x$ in $L^2(\R)$ where one truncates the part of the domain where the potential is unbounded, \eg~to $(-\infty, s_n)$ with $s_n \nearrow +\infty$, see Example~\ref{ex:tr.ex}.

An example of our results for the operators \eqref{Tg.intro} with a strongly coupled $\Im Q$ are the eigenvalues of operators $T_{g}$ in $L^2(\R)$ 
\begin{equation}\label{Tg.Schenker.intro}
	T_g=\Dti+x^2+\ii g (1+|x|^\kappa)^{-1},
\end{equation}
with $\kappa,g >0$, which are known to satisfy $\Re \sigma (T_g) \geq C_\kappa \, g^{\frac{2}{\kappa+2}}$ for $g >0$, see \cite{Schenker-2011-18}. Our results show that this bound is exhausted as $g \to + \infty$ since the spectra of $T_g$ contain asymptotically the eigenvalues
\begin{equation}
\lambda_{k,g}=g^{\frac{2}{\kappa+2}}(\overline \nu_{k,\kappa}+o_k(1))+\ii g, \quad g \to + \infty,
\end{equation}
where $\{\nu_{k,\kappa}\}$ are eigenvalues of $ \Dti+\ii|x|^\kappa$ in $L^2(\R)$ (see  Example~\ref{ex:enh.dis} for more details and remainder estimates).   

On a more technical side, in this paper we focus on accretive case ($\Re Q \geq 0$) and Dirichlet boundary conditions, nonetheless, several extensions are possible and straightforward, see remarks after Assumption~\ref{asm:Q}, Theorem~\ref{thm:T.basic} and Remark~\ref{rem:sect}.  In Section~\ref{sec:prelim} we collect relevant known facts about Schr\"odinger operators with complex potentials, in particular, the results on the domains, graph norm, compactness of resolvent and eigenfunctions decay; justifications for some slight extensions are given in Appendix. We also include several examples used throughout the paper.

In Section~\ref{sec:pert}, we estimate of the resolvent difference for Schr\"odinger operators with perturbed potential as well as underlying domain, see Theorem~\ref{thm:res}. The latter can be seen as a generalization of the estimates in~\cite{Davies-1983-9,Davies-1984-12} to an accretive case with a variable underlying domain and represents the key technical step in our analysis. The essential ingredient is the graph norm separation
\begin{equation}
	\|(-\Delta +Q)f\| + \|f\| \gs \|\Delta f\| + \|Qf\| + \|f\|,
\end{equation}
see Theorem~\ref{thm:T.basic}, which is known to be valid (for $Q$ unbounded at infinity) if
\begin{equation}\label{Q.nab.as}
|\nabla Q(x)| = o(|Q(x)|^\frac 32), \quad |x| \to \infty,
\end{equation}
see \cite{Krejcirik-2017-221}, Assumption~\ref{asm:Q} and remarks below for more details and Remark~\ref{rem:sect} for possible extensions in less regular cases. Theorem~\ref{thm:res} as well as the related estimates on the eigenvalues and eigenfunctions in Theorem~\ref{thm:rates} are reformulated for a sequence of operators in Corollary~\ref{cor:n} which constitutes our main technical tool. 

In Section \ref{sec:dom.tr} we revisit domain truncations in Theorem~\ref{thm:dom.tr} and improve the previous results in \cite{Boegli-2017-42} which are based on collective compactness (and a slightly stronger assumption than \eqref{Q.nab.as}). In Section~\ref{sec:div.ev} we implement the localization strategy of \cite{Beauchard-2015-21} and reformulate conditions of Corollary~\ref{cor:n}, see Theorem~\ref{thm:div}. The somewhat implicit conditions on the potential $Q$ can be further substantially simplified in one dimensional case with purely imaginary potentials, see Theorem~\ref{thm:1D.im}; interestingly the main condition in Theorem~\ref{thm:1D.im} is Assumption~\ref{asm:U} \ref{asm:U*} is closely related to \eqref{Q.nab.as}. 
A range of examples is covered in Section~\ref{sec:ex.div}, including a multidimensional one in Example~\ref{eq:Q.2d} for which the assumptions of Theorem~\ref{thm:div} are verified directly. 
Finally, in Theorem~\ref{thm:g} we employ Corollary~\ref{cor:n} to analyze operators \eqref{Tg.intro}. %

All plots are produced using build in commands in Mathematica, namely NDEigenvalues using FiniteElement PDEDiscretization method and improving its precision by refinement of the mesh with setting MaxCellMeasure to 0.01.

\subsection{Notation}
We use conventions $\N = \{1,2,\dots\}$, $\N_0=\N \cup \{0\}$, $\N^* = \N \cup \{\infty\}$, a subscript in the notation $\BigO_a$ indicates that the constant depends on the parameter $a$, and $\langle x \rangle := (1+|x|^2)^\frac 12$. We write $a \lesssim b$ to denote that, given $a,b \ge 0$, there exists a constant $C>0$, independent of any relevant variable or parameter, such that $a \le Cb$; $a \gtrsim b$ is analogous and $a \approx b$ means that $a \lesssim b$ and $a \gtrsim b$. 

For an open $\Omega \subset \R^d$ and a measurable function $m : \Omega \to \C$, we define the corresponding multiplication operator in $L^2(\Omega)$ on the maximal domain 
\begin{equation}
	\Dom(m):=\{ f \in L^2(\Omega) \, : \, m f \in L^2(\Omega)\}.
\end{equation}
The Dirichlet Laplacian $-\Delta_{\rm D}$ is defined via its quadratic form, \ie~
\begin{equation}
	\Dom(\Delta_{\rm D}):= \{ f \in W^{1,2}_0(\Omega) \, : \, \Delta f \in L^2(\Omega)\}.
\end{equation}
The characteristic function of a set $\Sigma$ is denoted by $\chi_\Sigma$ and 
$\widetilde \chi_\Sigma := 1 - \chi_\Sigma.$

\section{Preliminaries}
\label{sec:prelim}

We collect several known results on Schr\"odinger operators with complex potentials, mostly following~\cite{Almog-2015-40,Krejcirik-2017-221,Boegli-2017-42}; precise references are given at individual claims. We also work out several examples that are used later to illustrate the results.
\subsection{Schr\"odinger operators with complex potentials}
\label{ssec:Schr}
The main basic assumption on the potential $Q$ reads as follows.

\begin{asm-sec}\label{asm:Q}
	Let $\emptyset \neq \Omega \subset \R^d$ be open and let $Q \in W^{1,\infty}_{\rm loc}(\ov \Omega; \C)$ with $\Re Q \geq 0$ satisfy
	\begin{equation}\label{Q.asm.sep}
     \exists\eps_\nabla \in [0,\ec), \quad  \exists M_{\nabla} \geq 0, \quad  
			|\nabla Q| \leq \eps_\nabla |Q|^\frac 32 + M_{\nabla} \quad \text{a.e.~in~} \Omega;
	\end{equation}
	here $\ec = 2-\sqrt2$. 
	\hfill $\blacksquare$
\end{asm-sec}

The value of $\ec$ in Assumption \ref{asm:Q} is obtained by simple estimates in Lemma~\ref{lem:gn.est.app} in Appendix which imply the graph norm estimate \eqref{T.norm.est}. The optimal value of $\ec$ for the latter in the self-adjoint case is $\ec=2$, see \cite{Evans-1978-80,Everitt-1978-79}. 
In examples of $Q$, \eqref{Q.asm.sep} holds usually with an arbitrarily small $\eps_\nabla>0$, see \eg~Example~\ref{ex:Schr.basic}, and if $Q$ is unbounded it suffices to check~\eqref{Q.nab.as}.

We omit explicit claims, nonetheless, as a consequence \eqref{T.norm.est} or more precisely \eqref{T.norm.est.lem}, the results summarized in this section generalize straightforwardly when a relatively bounded perturbation with a sufficiently small bound is added.

The Dirichlet realization $T$ of $-\Delta +Q$ in $L^2(\Omega)$ can be obtained via the form
\begin{equation}\label{t.def}
	\begin{aligned}
		t[f]:=\|\nabla f\|^2 + \int_\Omega Q(x) |f(x)|^2 \, \dd x,
		\quad
		\Dom(t):= W_0^{1,2}(\Omega) \cap \Dom(|Q|^\frac12)
	\end{aligned}
\end{equation}
invoking the generalization of Lax-Migram theorem due to Almog and Helffer \cite{Almog-2015-40}.
The associated operator is defined in the usual way
\begin{equation}\label{T.def.form}
	\begin{aligned}
		\Dom(T)&:=\{ f \in \Dom(t) \, : \, \exists \eta \in L^2(\Omega), \ \forall g \in \Dom(t), \ t(f,g)=\langle \eta, g \rangle\},
		\\
		Tf&:=\eta = -\Delta f + Qf.
	\end{aligned}
\end{equation}

Under Assumption~\ref{asm:Q} the form $t$ is bounded with respect to a natural norm $\|\cdot\|_{t}^2 := \|\cdot \|_{W^{1,2}}^2 + \| |Q|^\frac12 \cdot \|^2$, %
but not coercive in general. Nevertheless, following \cite{Almog-2015-40}, the form $t$ exhibits a generalized coercivity 
\begin{equation*}
|t_\alpha(f,f)| + |t_\alpha(\Phi f,f)| \gs \|f\|_{t}^2,
\quad 
|t_\alpha(f,f)| + |t_\alpha(f, \Phi f)| \gs\|f\|_{t}^2, \quad f \in \Dom(t),
\end{equation*}
where $t_\alpha:=t+\alpha$ with some $\alpha \geq 0$ and 
\begin{equation}\label{Phi.def}
	\Phi:= \frac{\Im Q}{\sqrt{1+|Q|^2}}.
\end{equation}

Alternatively, one can introduce $T$ by Kato's theorem~\cite[Thm.~VII.2.5]{EE}, see~\cite{Boegli-2017-42}.

The following theorem summarizes known properties of $T$.
\begin{theorem}\label{thm:T.basic}
	Let $Q$ satisfy Assumption~\ref{asm:Q} and let the operator $T$ be defined as in \eqref{T.def.form}. Then $T$ is $m$-accretive, moreover,  
	\begin{enumerate}[\upshape i), wide]
		\item \label{thm:T.basic.gn}
		the graph norm of $T$ separates, \ie~there is a constant $a_{\nabla}>0$, depending only on $\eps_\nabla$, $M_\nabla$, such that for all $f \in \Dom(T)$
		\begin{equation}\label{T.norm.est}
			\|T f\|^2 + \|f\|^2 \geq a_\nabla (\|\Delta f\|^2 + \|Q f\|^2+\|f\|^2),
		\end{equation}
		hence the domain of $T$ separates, \ie~$\Dom(T) = \Dom(-\Delta_{\rm D}) \cap \Dom(Q)$;
		\item $T$ is $\cC$-self-adjoint, \ie~$T^*=\cC T \cC$, where $\cC$ is the complex conjugation operator $\cC f=\overline{f}$, $f \in L^2(\Omega)$, thus
		\begin{equation}\label{T.adjoint}
			\begin{aligned}
				T^* = -\Delta + \ov{Q}, 
				\quad 
				\Dom(T^*) = \Dom(T);
			\end{aligned}
		\end{equation}
		\item if $\Omega$ is bounded or if $\Omega$ in unbounded and 
		\begin{equation}\label{Q.unbd}
			\lim_{R \to \infty} \essinf_{|x|>R, x \in \Omega}|Q(x)| = +\infty, 
		\end{equation}
		then $T$ has compact resolvent, thus the spectrum of $T$ is discrete (consists of isolated eigenvalues of finite algebraic multiplicity);
		\item denote by $S$ the self-adjoint Dirichlet realization of $-\Delta + |Q|$ in $L^2(\Omega)$, \ie~
		$S := -\Delta + |Q|$, $\Dom(S):= \Dom(T)$, then (with $p>0$)
		\begin{equation}\label{Sp.equiv.1}
			(T+1)^{-1} \in \cS_p(L^2(\Omega)) \quad \Longleftrightarrow \quad (S+1)^{-1} \in \cS_p(L^2(\Omega))
		\end{equation}
		and with $k_1, k_2 >0$ depending only on $\eps_\nabla, M_\nabla$, 
		\begin{equation}\label{Sp.equiv.2}
			k_1 \|(S+1)^{-1}\|_{\cS_p} \leq		\|(T+1)^{-1}\|_{\cS_p} \leq k_2		\|(S+1)^{-1}\|_{\cS_p}.
		\end{equation}
		Moreover, let $\partial \Omega \in C^{2,\alpha}$ for some $\alpha>0$. If for $p>0$
		\begin{equation}\label{Sp.crit}
			\int_{\Omega \times \R^d} (|\xi|^2+|Q(x)|+1)^{-p} \, \dd x \dd \xi < \infty,
		\end{equation}
	then $(S+1)^{-1} \in \cS_p(L^2(\Omega))$.
	\end{enumerate}
\end{theorem}

The proofs can be found in \cite{Almog-2015-40,Krejcirik-2017-221} where more general forms of $T$ are analyzed (\eg~with a real magnetic field, complex rotated coefficients, a relatively bounded negative real part of the potentials or singular perturbations controlled by the Laplacian). Estimates on the constant $a_\nabla$ are in Lemma~\ref{lem:gn.est.app}. The equivalences \eqref{Sp.equiv.1}, \eqref{Sp.equiv.2} are a consequence of the characterization of $\Dom(T)$, in particular \eqref{T.norm.est}. Assuming a suitable regularity of $\Omega$, one can also include different boundary conditions (Neumann, Robin), see \cite{Boegli-2017-42} for some details.

%
%\begin{remark}\label{rem:C.eps}
%One can relax the condition~\eqref{Q.asm.sep} in Assumption~\ref{asm:Q} to the existence of a sufficiently small $\eps_0>0$ for which \eqref{Q.asm.sep} holds with some constant $M_{\eps_0,\nabla} \geq 0$. The estimate on the possible value of $\eps_0$ can be read off the proof of Lemma~\ref{lem:gn.est.app} in Appendix; see also \cite{Brown-1999-455} for the self-adjoint case.
%
%
%For further considerations, it is crucial to observe that $C_\eps$ appearing in the graph norm estimate \eqref{T.norm.est} depends on $\eps$ and the constant $M_{\nabla,\eps}$ from \eqref{Q.asm.sep} only; for details see Lemma~\ref{lem:gn.est.app}
%%
%\hfill $\blacksquare$
%\end{remark}

\begin{example}\label{ex:Schr.basic}
	Let $Q$ with $\Re Q \geq 0$ be such that 
	\begin{equation}\label{ex.Q.beh}
 |Q(x)| \approx \langle x \rangle^\gamma, \qquad x \in \R^d.
	\end{equation}
	with some $\gamma>0$ and let \eqref{Q.nab.as} and hence \eqref{Q.asm.sep} be satisfied; notice that \eg~in the special case of $Q(x) = \ii \langle x \rangle^\gamma$, we have 
	\begin{equation}
		|\nabla Q(x)| = \BigO(|Q(x)||x|^{-1}), \qquad |x| \to \infty.
	\end{equation}
	From Theorem~\ref{thm:T.basic}, the corresponding Schr\"odinger operator $T$ is m-accretive and has a compact resolvent, moreover, for any $\eps>0$,
	\begin{equation}\label{Qg.Sp}
		(T+1)^{-1} \in \cS_{p_{\gamma,d}+\eps}(L^2(\R^d)), \qquad p_{\gamma,d}= \frac{2+\gamma}{2 \gamma} d;
	\end{equation}
	the latter follows from \eqref{Sp.crit} by Young's inequality.
	\hfill $\blacksquare$
\end{example}

\subsection{Decay of eigenfunctions}
\label{subsec:EF.dec}

Result on the eigenfunctions decay for accretive Schr\"odinger operators can be found in \cite{Krejcirik-2017-221}. A slight adaptation of \cite[Prop.~4.1]{Krejcirik-2017-221} yields the following, (see Appendix for details).

\begin{theorem}\label{thm:EF.dec}
Let Assumption~\ref{asm:Q} be satisfied, $T$ be the Dirichlet realizations of $-\Delta+Q$ in $L^2(\Omega)$, $\Phi$ be as in \eqref{Phi.def} and define
\begin{equation}\label{Qt.def}
\widetilde Q := \Re Q + \Phi \Im Q - \frac1{2}|\nabla\Phi|^2.
\end{equation}
Let $\la \in \spp(T)$ and $\psi,\psi_0 \in \Dom(T)$ satisfy $T \psi = \la \psi + \psi_0$.
Suppose that for this $\la$, there exist an open $\Omega_1 \subset \Omega$, a constant $\delta>0$ and a weight $W \in W^{1,\infty}_{\rm loc}(\Omega;\R)$ such that $e^W \psi_0 \in L^2(\Omega)$, $e^W + |\widetilde Q| + |\nabla W| \in L^\infty(\Omega \setminus \Omega_1)$ and
\begin{equation}\label{Q.W.ineq}
\widetilde Q -3 |\nabla W|^2 - \Re \la - |\Im \la| \geq 2 \delta >0, \quad \text{a.e. in } \Omega_1. 
\end{equation}	
Then there exists $C = C(\la,\delta,\Omega_1,W,\widetilde Q)>0$ such that
\begin{equation}\label{eW.psi}
\|e^W \psi\| \leq C \|\psi\|+ \delta^{-2} \|\e^W \psi_0\|.
\end{equation}	
(For an estimates of $C$ see the proof Theorem~\ref{thm:EF.dec} in Appendix.)	
\end{theorem}
%
%Notice that under Assumption~\ref{asm:Q}, we get by straightforward estimates that
%%
%\begin{equation}\label{nabla.Phi.es}
%|\nabla \Phi|^2 \leq \frac{4}{1+|Q|^2} \left(\eps_\nabla |Q|^\frac32 + C_{\eps_\nabla}\right)^2,
%\end{equation}
%%
%which provides an estimate on $\widetilde Q$ in \eqref{Qt.def}. This in particular simplifies if $Q$ satisfies \eqref{Q.unbd} and \eqref{Q.nab.as} since then $|\nabla \Phi|^2$ is negligible comparing to $Q$ for large $|x|$.
%
% A suitable weight $W$ can be constructed using the Agmon distance, see \eg~\cite{Krejcirik-2017-221} for details.
%%

\begin{example}\label{ex:EF.dec} 
Let $Q$ and $T$ be as in Example~\ref{ex:Schr.basic}. It follows from  \eqref{ex.Q.beh}, \eqref{Q.nab.as} and a straightforward estimate of $|\nabla \Phi|$ that there exist $\beta, R>0$ such that
\begin{equation}
\widetilde Q(x) \geq  \beta |x|^\gamma, \quad |x| > R.
\end{equation}
Suppose that we can find a weight $W \in W^{1,\infty}_{\rm loc}(\Rd;\R)$ such that
\begin{equation}\label{nabW.ex}
|\nabla W(x)|^2 \leq \frac{1-\kappa}{3} \widetilde Q(x), \quad |x|>R,
\end{equation}
where $\kappa \in (0,1)$. Then for all $|x|>R$ 
\begin{equation}
\widetilde Q(x)-3|\nabla W(x)|^2 - \Re \la -|\Im \la| \geq 
\kappa \beta |x|^\gamma - \Re \la -|\Im \la|.
\end{equation}
Thus for each $\la \in \C$, $\Re \la \geq 0$, we can find $r_\la \geq R$ such that \eqref{Q.W.ineq} is satisfied with $\Omega_1=\Rd \setminus \ov{B_{r_\la}(0)}$ (and some $\delta>0$).  

A simple weight $W$ satisfying \eqref{nabW.ex} can be selected as a radial function
\begin{equation}\label{W.ex}
W(x):=\sqrt{\beta \frac{1-\kappa}{3}} \int_0^{|x|}  s^\frac \gamma2 \, \dd s. 
\end{equation}
Thus eigenfunctions $\psi$ of $T$ satisfy $e^W \psi \in L^2(\R^d)$ and the same follows for generalized eigenfunctions by a repeated application of Theorem~\ref{thm:EF.dec}. 
	
	Finally, we remark that the result on the eigenfunctions decay will remain valid also in other underlying domains than $\Rd$, \eg~for cones in Example~\ref{ex:Airy.cone} below. 
	\hfill $\blacksquare$	
\end{example}
\begin{example}\label{ex:EF.dec.ex} 
Let $T$ be the Schr\"odinger operator in $L^2(\R)$ with the potential $Q :\R \to \C$ that satisfies with some $\gamma >0$ and $R_0 \geq 0$,
\begin{equation}
|Q(x)| \approx |x|^\gamma \e^x, \quad |x| >R_0	
\end{equation}
and $|Q'(x)| = o(|Q(x)|^{3/2})$ as $x \to + \infty$ (hence Assumption~\ref{asm:Q} is satisfied with arbitrarily small $\eps_\nabla>0$.) Theorem~\ref{thm:EF.dec} applies with $\Omega_1 = (r_\la,+\infty)$ where $r_\la>0$ is sufficiently large so that \eqref{Q.W.ineq} holds for the weight (with a sufficiently small $\kappa>0$)
\begin{equation}
W(x) = \kappa \, \chi_{\R_+}(x) \int_0^x s^{\frac \gamma 2 } \e^\frac s2 \, \dd s, \quad x \in \R. 
\end{equation}
Thus the (generalized) eigenfunctions of $T$ decay fast as $x \to + \infty$.
\hfill $\blacksquare$	
\end{example}

\section{Perturbations in domain and potential}
\label{sec:pert}

For $j=1,2$, we consider the Dirichlet realizations $T_j = -\Delta + Q_j$ in $L^2(\Omega_j)$ with open $\Omega_j \subset \R^d$  as in Section \ref{ssec:Schr} and study the distance of resolvents and spectra.

\begin{asm-sec}\label{asm:pert}
For $j =1,2$, let $\Omega_j \subset \R^d$ be open and let $Q_j$ satisfy Assumption~\ref{asm:Q} with constants $\eps_{\nabla,j}$, $M_{\nabla,j}$. Let $T_j =-\Delta+Q_j$ be the  Dirichlet realizations in $L^2(\Omega_j)$ defined via the forms $t_j$ given in Theorem~\ref{thm:T.basic}.
Let $\Omega_0:=\Omega_1 \cup \Omega_2$ and suppose that $\Omega_1$ and $\Omega_2$ are such that there exists a cut-off $\xi: \Omega_0 \to [0,1]$ satisfying that $\chi_{\Omega_1 \cap \Omega_2}\xi = \xi$ on $\Omega_0$, $|\nabla \xi| + \Delta \xi \in L^\infty(\Omega_0)$ and
	\begin{equation}\label{xi.tt0}
		\begin{aligned}
			&\forall f \in \Dom(T_1), \quad \xi f   \in \Dom(t_2),
			\\
			&\forall g \in \Dom(T_2), \quad \xi g  \in \Dom(t_1);
		\end{aligned}
	\end{equation} 
	here we understand $\xi f$ as 
\begin{equation}\label{domain.cond}
	\begin{cases}
		\xi(x) f(x), & x \in \Omega_2 \cap \Omega_1, 
		\\
		0, & x \in \Omega_2 \setminus \Omega_1,
	\end{cases}
\end{equation}
and analogously for $\xi g$. 
%Finally, let $\widetilde \xi:=1-\xi$ and let $\widetilde \zeta : \Omega_1 \cup \Omega_2 \to [0,1]$ be such that $\zeta \equiv 1$ on $\supp \widetilde\xi$. 
%	
\hfill $\blacksquare$	
\end{asm-sec}

An illustration of a choice of suitable cut-off $\xi$ is in Figure~\ref{fig:cut-off}.  
\begin{figure}[htb!]
	\includegraphics[width=0.6\textwidth]{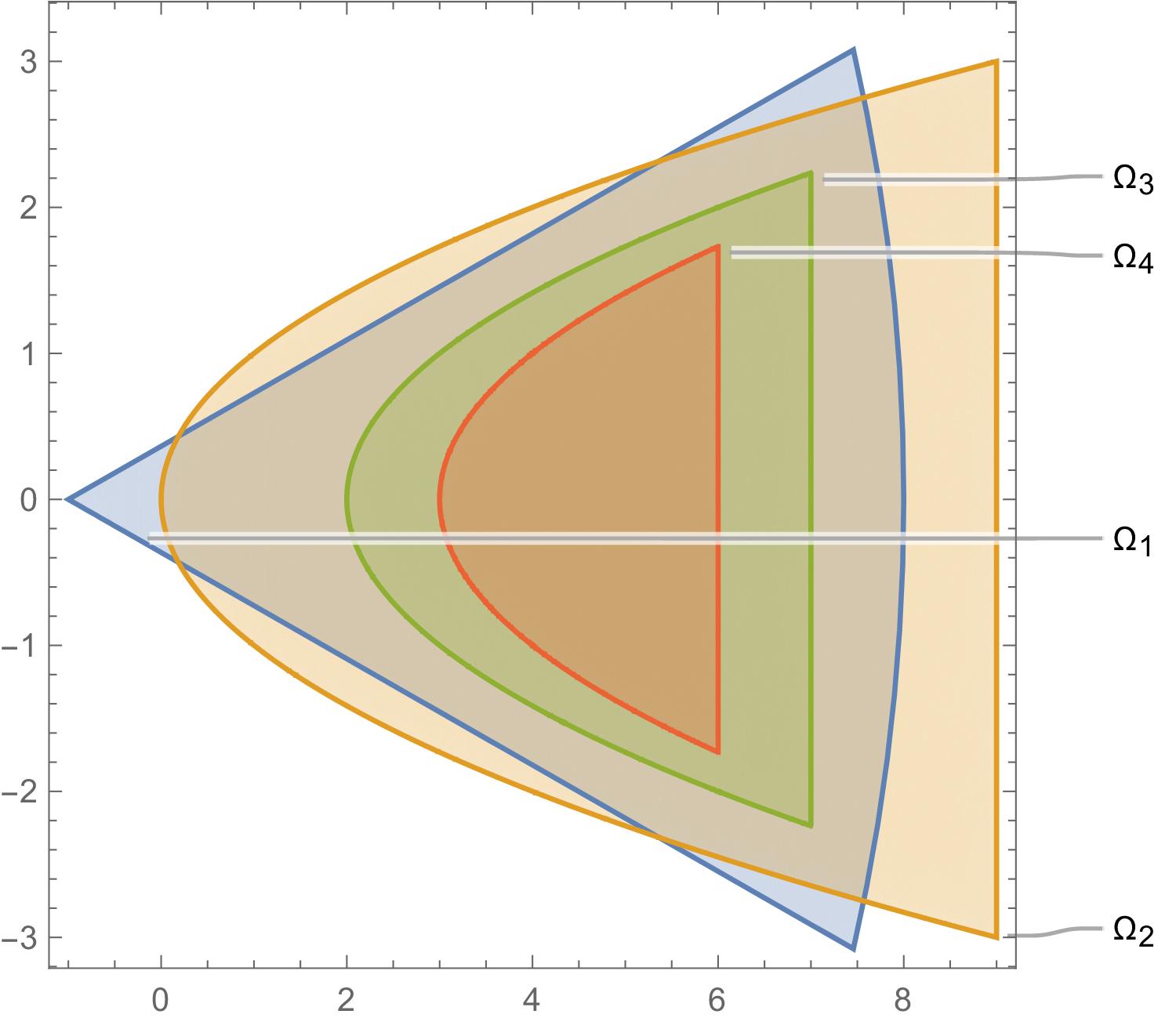}
	%	{mnoziny.pdf}
	\caption{The domains $\Omega_1$ (blue) and $\Omega_2$ (yellow) are taken as a part of sector and parabola, respectively. One can construct $\xi \in C^{\infty}(\Omega_0)$ with $\Omega_0:=\Omega_1 \cup \Omega_2$ such that $\xi=1$ on $\Omega_4 \subset \Omega_1 \cap \Omega_2$ (orange)  and $\xi=0$ on the complement of $\Omega_3$ (green) in $\Omega_0$. Since $\supp \xi$ is bounded, the conditions \eqref{xi.tt0} are satisfied for any admissible $Q_1$, $Q_2$ (which is not the case in general for unbounded $\Omega_1, \Omega_2$ and unbounded $\supp \xi$).}
	\label{fig:cut-off}
\end{figure}

\subsection{Resolvent difference estimate}
For $j=1,2$ and $z \in \rho(T_j)$, we write  $R_j(z):=(T_j-z)^{-1}$, $j=1,2$. We introduce 
\begin{equation}\label{zeta.def}
\widetilde \xi := 1 -\xi,	\qquad \zeta := \chi_{\supp \widetilde \xi}
\end{equation}
where $\xi$ is as in Assumption~\ref{asm:pert}; see also Figure~\ref{fig:cut-off}. Notice that
\begin{equation}\label{zeta.xi}
 \zeta \widetilde \xi = \widetilde \xi, \quad \zeta \nabla \xi = \nabla \xi, \quad \zeta \Delta \xi = \Delta \xi.
\end{equation}
In $L^2(\Omega_0)$, let $P_j$, $P$ and $\widetilde P$ be the following orthogonal projections 
\begin{equation}\label{Pj.def}
P_j f = \chi_{\Omega_j} f, \quad P f = \chi_{\Omega_1 \cap \Omega_2} f,  \quad \widetilde P:=I-P,\qquad f \in L^2(\Omega_0).
\end{equation}
\begin{theorem}\label{thm:res}
For $j=1,2$, let $\Omega_j$, $T_j$ and $\xi$ be as in Assumption \ref{asm:pert}, let $\zeta$ be as in \eqref{zeta.def} and let $P_j$ be as in \eqref{Pj.def}. Then there exists a constant $K \geq 0$, depending only on  $\|\nabla \xi\|_{L^\infty}$ and $\eps_{\nabla,j}$, $M_{\nabla,j}$, such that
\begin{equation}\label{TT0.res.est}
\begin{aligned}				
& \|R_1(-1)P_1 - R_2(-1) P_2 \|_{\cB(L^2(\Omega_0))}
\\ & \qquad 
\leq K 
\left( 
\left\| \frac{\xi (Q_1-Q_2)}{(Q_1+1)(Q_2+1)} \right\|_{L^\infty(\Omega_1 \cap \Omega_2)} + \sum_{j=1}^2\left\| \frac{\zeta }{Q_j+1}
 \right\|_{L^\infty(\Omega_j)} 
\right).
\end{aligned}
\end{equation}

If in addition $R_j(-1)\in \cS_p(L^2(\Omega_j))$ with some $p>0$, then for every $q \in (p,\infty)$ there exists a constant $K_q > 0$ such that
\begin{equation}\label{T.Sp}
%	\begin{aligned}
		\|R_1(-1)P_1 - R_2(-1) P_2\|_{\cS_q} 
%		\\
%		&  \qquad 
		\leq K_q \|R_1(-1)P_1 - R_2(-1) P_2 \|_{\cB(L^2(\Omega_0))}^{1-\frac pq}.
%	\end{aligned}
\end{equation}
\end{theorem}
\begin{remark}\label{rem:sect}	
The proof of \eqref{TT0.res.est} is based on the ``maximal estimate'' \eqref{T.norm.est} for the graph norms which results in terms with $|Q_j+1|^{-1}$. An analogue of \eqref{TT0.res.est} holds if weaker lower estimates of the graph norms are available, \eg~in the sectorial case with $Q_j \in L^1_{\rm loc}(\Omega_j)$, the denominators in \eqref{TT0.res.est} would contain $|Q_j+1|^{-1/2}$ instead. 

The claim \eqref{TT0.res.est} can be extended to all  $z \in \rho(T_1) \cap \rho(T_2)$ 
\begin{equation}\label{TT0.res.est.z}
	\begin{aligned}
	 \|R_1(z) P_1 - R_2(z) P_2\| & \leq \| (I - (z+1) R_1(z) P_1)^{-1} \| \| I + (z+1) R_2(z) P_2 \|  \times
		\\
		 & \qquad \|R_1(-1) P_1 - R_2(-1) P_2\|;
	\end{aligned}
\end{equation}
the proof employs resolvent identities in a straightforward way.
\hfill $\blacksquare$	
\end{remark}

\begin{proof}[Proof of Theorem~\ref{thm:res}]
We establish a resolvent-type identity (for all $f,g \in L^2(\Omega_0)$)
\begin{equation}\label{res.weak}
	\begin{aligned}
		& \langle (R_1(z)P_1-R_2(z)P_2)f,g \rangle  =
		\\&  \quad \langle R_2(z) \xi (Q_2-Q_1) R_1(z) P f,  P   g \rangle
		\\&  \quad + \langle (\nabla \xi) R_1(z) P f, \nabla R_2(z)^* P g \rangle - \langle \nabla R_1(z) P f, (\nabla \xi)  R_2(z)^* P g \rangle
		\\&  \quad + \langle  \widetilde \xi R_1(z) P f,  g \rangle - \langle f, \widetilde \xi  R_2(z)^* P g \rangle 
		%\\ & \qquad 
		+ \langle R_1(z) P_1 \widetilde P f, g \rangle   - \langle  f, R_2(z)^* P_2  \widetilde P g \rangle.
	\end{aligned}
\end{equation}
To this end, we analyze the first term after the following splitting
\begin{equation}\label{res.1}
	\begin{aligned}
		\langle (R_1(z)P_1-R_2(z)P_2)f,g \rangle & = \langle (R_1(z)-R_2(z))Pf,Pg \rangle 
		\\ & \quad +
		\langle (R_1(z)P_1-R_2(z)P_2) \widetilde P f,g \rangle 
		\\ & \quad + \langle (R_1(z)-R_2(z)) P f, \widetilde P g \rangle.
	\end{aligned}
\end{equation}
Let $F:=R_1(z) P f$ and $G:=R_2(z)^* P g$, then, inserting $1=\xi+\widetilde \xi$, we obtain
\begin{equation}\label{res.2}
	\begin{aligned}
		\langle (R_1(z)-R_2(z))Pf,Pg \rangle &= 
		\langle \xi F, T_2^* G \rangle - \langle  T_1 F, \xi G \rangle
		\\
		& \quad + \langle \widetilde \xi R_1(z) Pf, Pg \rangle - \langle Pf , \widetilde \xi R_2(z)^* P g \rangle.
	\end{aligned}
\end{equation}
Using the assumption \eqref{xi.tt0}, we get
\begin{equation}\label{res.3}
	\begin{aligned}
	&	\langle \xi F, T_2^* G \rangle - \langle  T_1 F, \xi G \rangle
		= t_2(\xi F,G) - t_1(F, \xi G)
		\\& \quad 
		= \langle R_2(z) \xi (Q_2-Q_1) R_1(z) P f,  P   g \rangle 
%		\\ & \quad 
		+  \langle (\nabla \xi) R_1(z) P f, \nabla R_2(z)^* P g \rangle  
		\\ & \quad \quad - \langle \nabla R_1(z) P f, (\nabla \xi) R_2(z)^* P g  \rangle 
	\end{aligned}
\end{equation}
Putting together \eqref{res.1}, \eqref{res.2}, \eqref{res.3} and using $\widetilde \xi \widetilde P =\widetilde P$ we arrive at \eqref{res.weak}.

Next, we employ \eqref{res.weak} with $z=-1$ and 
	\begin{equation}\label{norm.weak}
		\|h\| = \sup_{g \neq 0} \frac{|\langle h,g \rangle|}{\|g\|}.
	\end{equation}
	Denoting $R_1:=R_1(-1)$, $R_2:=R_2(-1)$ and considering $\zeta$ in \eqref{zeta.def}, we get
	\begin{equation*}
		\begin{aligned}
			\|R_1P_1-R_2P_2\| & \leq   \| R_2 \xi (Q_2-Q_1) R_1 P\|  + 
			 \|\nabla \xi\|_{L^\infty}(\|\zeta R_1\| \|\nabla R_2^*\| + \|\zeta R_2^*\| \|\nabla R_1\|) 
			\\ & \quad +\|\zeta R_1\| + \|\zeta R_2\| + \|\zeta R_1^*\| + \|\zeta R_2^*\|.
		\end{aligned}
	\end{equation*}

	In the sequel $j=1,2$. We show that $(Q_j+1) R_j$, $(Q_j+1) R_j^*$ and $\nabla R_j$, $\nabla R_j^*$ are bounded operators. To this end, observe that 
	it follows from \eqref{T.norm.est} and \eqref{T.adjoint} that there are constants $C_{T_j} >0$ such that for all $f_j \in \Dom(T_j) = \Dom(T_j^*)$ 
	\begin{equation}\label{T.gn.eq}
		\begin{aligned}
			\|(T_j+1)f_j\| + \|f_j\| &\geq C_{T_j} (\|\Delta f_j\| + \|(Q_j+1) f_j\| + \|f_j\|),
			\\
			\|(T_j^*+1)f_j\| + \|f_j\| &\geq C_{T_j} (\|\Delta f_j\| + \|(Q_j+1) f_j\| + \|f_j\|),
		\end{aligned}
	\end{equation}
	and $C_{T_j}$ depend only on $\eps_{\nabla,j}$ and $M_{\nabla,j}$ from \eqref{Q.asm.sep}. 
	
	Using a numerical range argument, we have $\|R_1\| \leq 1$, thus 
	\begin{equation}\label{QR.est}
	\|(Q_1+1)R_1 f\| \leq \frac{1 + \|R_1\| }{C_{T_1}}  \|f\|\leq \frac{2}{C_{T_1}} \|f\|, \quad  f \in L^2(\Omega_1);
	\end{equation}
	the estimate of other similar terms is analogous. 
	Furthermore, as for all $g \in \Dom(T_1)$ we have $\|\nabla g\|^2 \leq \|\Delta g\| \|g\|$, 
	we get from \eqref{T.gn.eq} that
	\begin{equation}
		\|\nabla R_1 f\|^2 \leq \|\Delta R_1 f\| \|R_1 f\| \leq \frac{2}{C_{T_1}} \|f\|^2, \quad f \in L^2(\Omega_1);
	\end{equation}
	the estimate of other similar terms is analogous. 
	
	Finally, by inserting $1 = (Q_1+1)^{-1} (Q_1+1)$, we get
	\begin{equation}\label{T.eta.est}
		\begin{aligned}
			\|\zeta R_1 \| &\leq \| \zeta (Q_1+1)^{-1} \| \|(Q_1+1) R_1 \| 
			%\\&
			\leq
			\frac{2}{C_{T_1}}  \| \zeta (Q_1+1)^{-1} \|
		\end{aligned}
	\end{equation}
	and analogously for the remaining terms. The claim \eqref{TT0.res.est} follows by putting the estimates above together.	
	The estimate with Schatten norms \eqref{T.Sp} follows by H\"older's inequality in $\cS_p$, see~\eg~\cite[Lem.~XI.9]{DS2}.
\end{proof}

\subsection{Eigenvalues and eigenfunctions convergence}

Let $T_j$, $j=,1,2$, be as in Assumption~\ref{asm:pert} and let $\mu \in \sigma(T_1)$ be an isolated eigenvalue of finite algebraic multiplicity $m \in \N$. If the gap distance of $T_1$ and $T_2$ is sufficiently small (or equivalently the norm of the difference of resolvents estimated in Theorem~\ref{thm:res}), then $\sigma(T_2)$ contains exactly $m$ eigenvalues $\{\mu_k\}_{k=1}^m$ in a neighborhood of $\mu$ (counting with multiplicities). This follows by estimating the norm of difference of spectral projections (with a suitable contour $\gamma_\mu$)
\begin{equation}\label{E.def}
	E_1:= \frac1{2\pi \ii} \int_{\gamma_\mu} (z-T_1)^{-1} P_1 \, \dd z, \qquad  E_2:=\frac1{2\pi \ii} \int_{\gamma_\mu} (z-T_2)^{-1} P_2 \, \dd z;
\end{equation}
for details see \eg~\cite[Thm.~5.1]{Boegli-2017-42}, \cite{Osborn-1975-29}, \cite[Chap.~IV]{Kato-1966}. 

Our goal is to estimate the distance of $\mu$ and the average of $\mu_k$
\begin{equation}\label{hla.def}
	\hat \mu := \frac1m \sum_{k=1}^m \mu_k
\end{equation}
and the distance of eigenfunctions. Notice that the estimate \eqref{TT0.res.est} relates the resolvent difference with a decay of~$|Q|^{-1}$. Nevertheless, the convergence rate of eigenvalues and eigenfunctions is typically much faster as these are related to the decay of eigenfunctions, for an illustration see~Corollary~\ref{cor:n} and Theorem~\ref{thm:dom.tr} below.

\begin{theorem}\label{thm:rates}
For $j=1,2$, let $\Omega_j$, $T_j$ and $\xi$ be as in Assumption \ref{asm:pert}, let $\zeta$ be as in \eqref{zeta.def}, let $P_j$, $P$ and $\widetilde P$ be as in \eqref{Pj.def}. Let $\mu \in \sigma(T_1)$ be an isolated eigenvalue of finite algebraic multiplicity $m \in \N$. Suppose further that $\Omega_j$ and $Q_j$, $j=1,2$, are such that the spectral projections $E_j$, $j=1,2$, in \eqref{E.def} satisfy $\|E_1-E_2\| <1$. Then the following hold. 

\begin{enumerate}[\upshape i), wide]
\item  Let $\hat \mu$ be as in \eqref{hla.def}, then
	\begin{equation}\label{la.conv.est}
		|\mu-\hat \mu| \leq C_{1,\mu} \max_{\substack{\phi \in \Ran (E_1) \\ \|\phi\|=1}} \left\| \frac{\xi (Q_1-Q_2)}{(Q_1+1)(Q_2+1)} \phi \right\| + C_{2,\mu} \max_{\substack{\phi \in \Ran (E_1) \\ \|\phi\|=1}} \| \zeta \phi\|,
	\end{equation}
\item For all $ \psi \in \Ran(E_1)$, we have
	\begin{equation}\label{EF.conv}
		\begin{aligned}			
		\|\psi - E_2 \psi\| & \leq D_{1,\mu}\max_{\substack{\phi \in \Ran (E_1) \\ \|\phi\|=1}} \left\| \frac{\xi (Q_1-Q_2)}{(Q_1+1)(Q_2+1)} \phi \right\| +  D_{2,\mu} \max_{\substack{\phi \in \Ran (E_1) \\ \|\phi\|=1}} \| \zeta \phi\| 
		\\
		& \quad+ \|E_2\| \max_{\substack{\phi_2 \in \Ran (E_2) \\ \|\phi_2\|=1}} \| \widetilde P \phi_2\| , 
		\end{aligned}
	\end{equation}
(For estimates of the constants $C_{j,\mu}$ and $D_{j,\mu}$, $j=1,2$, see \eqref{CD.mu} below.)
\end{enumerate}
\end{theorem}
\begin{proof}
We follow standard regular perturbation theory arguments, see~\eg~\cite[Thm.~2]{Osborn-1975-29} or \cite[Chap.~XII.2]{Reed4} for details.
\begin{enumerate}[\upshape i), wide]
\item Since $E_1$ and $E_2$ are sufficiently close,  $E_2 \restriction \Ran(E_1) : \Ran (E_1) \to \Ran (E_2)$ is bijective and for
$F_2:=(E_2 \restriction \Ran(E_1))^{-1}: \, \Ran(E_2) \to \Ran (E_1)$
we get from $\|E_1f\| - \|E_2f\| \leq \|(E_1-E_2)f\|$ that 	$\|F_2\| \leq (1-\|E_1-E_2\|)^{-1}$.
Moreover, $F_2 E_2\restriction \Ran(E_1) = I \restriction \Ran(E_1)$ and $E_2 F_2 \restriction \Ran(E_2) = I \restriction \Ran(E_2)$. We define $\hat T_1:=T_1 E_1$ and $\hat T_2 := F_2 T_2 E_2 \restriction \Ran(E_1)$ to obtain 
\begin{equation}\label{la.est.1}
\mu-\hat \mu = \frac1m \Tr (\hat T_1 - \hat T_2) = \frac1m \sum_{k=1}^m \langle  (\hat T_1 - \hat T_2) f_k, f_k \rangle,
\end{equation}
where $\{f_k\}_{k=1}^m$ is an orthonormal basis of $\Ran(E_1)$. Using the assumption \eqref{xi.tt0}, one can check that $\xi f_k \in \Dom(T_j)$, $j=1,2$, $k=1,\dots,m$, thus (with $\widetilde \xi:=1-\xi$)
\begin{equation}\label{TT0fj}
\begin{aligned}
\hat T_1 f_k - \hat T_2 f_k &= F_2 E_2 T_1 \xi f_k - F_2 E_2 T_2 \xi f_k + F_2 E_2 T_1 \widetilde \xi f_k - F_2 T_2 E_2 \widetilde \xi f_k 
\\ &
= F_2  E_2 (Q_1-Q_2) \xi f_k - F_2 T_2 E_2 \widetilde \xi f_k 
\\
& \qquad
 + 
F_2  E_2 
(
 \widetilde \xi T_1 f_k - 2 \nabla \xi. \nabla f_k  - (\Delta \xi) f_k 
).
\end{aligned}
\end{equation}
Since $f_k \in \Ran (E_1)$, we can estimate 
\begin{equation}
\|\widetilde \xi T_1 f_k\| \leq \frac {\|\widetilde \xi T_1 f_k\|}{\|T_1 f_k\|} \|T_1 E_1 f_k\| 
\leq \max_{\substack{\phi \in \Ran (E_1) \\ \|\phi\|=1}} \| \zeta \phi\| \|T_1 E_1\|.
\end{equation}
The remaining terms are estimated in a straightforward way; we use that $\nabla \xi. \nabla f_k = \nabla \xi. \nabla (\zeta f_k)$ and $E_2 \nabla \xi. \nabla$ has a bounded extension.
%, see \eqref{CD.mu} below for the estimate of the norm. 

\item Consider $\psi \in \Ran (E_1)$. Using \eqref{E.def} and the resolvent identity \eqref{res.weak}, we get
\begin{equation}
\langle \psi-E_2 \psi,g \rangle  = -\frac{1}{2 \pi \ii} \int_{\gamma_\mu} \langle (R_1(z)P_1-R_2(z)P_2) \psi,g \rangle \, \dd z
\end{equation}
for all $g \in L^2(\Omega)$. The claim \eqref{EF.conv} follows by \eqref{norm.weak}, \eqref{zeta.xi} and the following manipulations.
The term with $Q_1-Q_2$ is estimated in a straightforward way. In terms with $\nabla \xi$, we integrate by parts and get
\begin{equation}
\begin{aligned}
&\langle (\nabla \xi) R_1(z) P f, \nabla R_2(z)^* P g \rangle - \langle \nabla R_1(z) P f, (\nabla \xi) R_2(z)^* P g \rangle
\\
& \quad 
= 2 \langle (\nabla \xi) R_1(z) P f, \nabla R_2(z)^* P g \rangle + \langle (\Delta \xi) R_1(z) P f, R_2(z)^* P g \rangle.
\end{aligned}
\end{equation}
We further rewrite $P\psi = \psi - \widetilde P \psi$, use that $R_1(z) \psi \in \Ran(E_1)$ and obtain
\begin{equation}
\|\zeta R_1(z) \psi\| \leq \|R_1(z) \psi\| \max_{\substack{\phi \in \Ran (E_1) \\ \|\phi\|=1}} \| \zeta \phi\|.
\end{equation}
In the other terms, the integration leads to formulas with spectral projections. 	
\qedhere
\end{enumerate}
\end{proof}

From the proof we can deduce the following estimates on the appearing constants

\begin{equation}\label{CD.mu}
	\begin{aligned}
		C_{1,\mu} & = \|(E_2 \restriction \Ran(E_1))^{-1}\| \|E_2(Q_2+1)\| \|(Q_1+1)E_1\|, 
		\\
		C_{2,\mu} & = \|(E_2 \restriction \Ran(E_1))^{-1}\|  (\|E_2\| \| T_1 E_1\| + \|T_2 E_2\| 
		\\ & \qquad \qquad  \qquad  \qquad \qquad \qquad  + 2  \| \nabla \xi \|_{L^\infty} \|\nabla E_2^*\| + \|\Delta \xi\|_{L^\infty} \|E_2^*\| ), 
		\\
		D_{1,\mu} & = \frac{|\gamma_\mu|}{2\pi} \max_{z \in \gamma_\mu} \left( \|(Q_1+1)R_1(z)\| \|(Q_2+1) R_2(z)^* \| \right), 
		\\
		D_{2,\mu} & = \frac{|\gamma_\mu|}{\pi} \max_{z \in \gamma_\mu} ( 2\|\nabla \xi\|_{L^\infty} \|R_1(z)\| \|\nabla R_2(z)^* \| 
%		\\ & \qquad \qquad \qquad  \qquad
		+ \|\Delta \xi\|_{L^\infty} \|R_1(z)\| \|R_2(z)\|) 
		\\ & \quad
		+1 + \|E_1\|+\|E_2\|. 
	\end{aligned}
\end{equation}

\subsection{Sequence of operators}

In next sections, we use Theorems~\ref{thm:res} and \ref{thm:rates} for a sequence of operators $\{T_n\}$ converging to $T_\infty$, in the setting summarized as follows.

\begin{asm-sec}\label{asm:n}
Suppose that  
\begin{enumerate}[\upshape i), wide]
	\item \label{asm:n.dom} domains $\{\Omega_n\}_{n \in \N^*} \subset \Rd$ are open (non-empty)
%	, $\Omega_\infty$ is \petrn{unbounded (why?)} 
	and $\Omega_n \subset \Omega_\infty$, $n \in \N$;
	\item \label{asm:n.Qn} potentials $Q_n \in W^{1,\infty}_{\rm loc}(\ov{\Omega_n})$ with $\Re Q_n \geq 0$, $n\in \N^*$, satisfy \eqref{Q.asm.sep} uniformly, 
	\begin{equation}\label{Qn.sep}
	 \exists\eps_\nabla \in [0,\ec), \  \exists M_{\nabla} \geq 0, \ \forall n \in \N^*, \ 
		|\nabla Q_n| \leq \eps_\nabla |Q_n|^\frac 32 + M_{\nabla} \quad \text{a.e.~in~} \Omega_n;
	\end{equation}

	\item \label{asm:n.xin} operators $T_n = -\Delta + Q_n$ in $L^2(\Omega_n)$ (introduced via forms $t_n$, $n \in \N^*$, as in Section~\ref{ssec:Schr}) and  
	cut-offs $\{\xi_n\}_{n \in \N}$ are such that 
	\begin{equation}\label{xin.nabla.unif}
		\sup_{n \in \N} \left (\||\nabla \xi_n| \|_{L^\infty} + \|\Delta \xi_n \|_{L^\infty} \right) < \infty
	\end{equation}
	and that the conditions of Assumption~\ref{asm:pert} are satisfied for $\Omega_1$, $\Omega_2$, $\xi$, $T_1$, $T_2$, $t_1$, $t_2$ replaced by $\Omega_n$, $\Omega_\infty$, $\xi_n$, $T_n$, $T_\infty$, $t_n$, $t_\infty$, respectively, for every $n \in \N$;
	\item \label{asm:n.Q.conv} potentials $\{Q_n\}$ converge in the following sense
	\begin{equation}\label{tau.n.def}
	\begin{aligned}
	\tau_n&:=\left\| \frac{\xi_n (Q_n-Q_\infty)}{(Q_n+1)(Q_\infty+1)} \right\|_{L^\infty(\Omega_n)} \! \! \! \! + \left\| \frac{\zeta_n}{Q_n+1} \right\|_{L^\infty(\Omega_n)} \! \! \! \! +\left\| \frac{\zeta_n}{Q_\infty+1} \right\|_{L^\infty(\Omega_\infty)}
	\\ & = o(1), \quad n \to \infty,
	\end{aligned}
	\end{equation}
	where 
$\widetilde \xi_n := 1 -\xi_n$, $\zeta_n := \chi_{\supp \widetilde \xi_n}$, $n \in \N$.	
\hfill $\blacksquare$	
\end{enumerate}
\end{asm-sec}

A typical situation we analyze is with $\Omega_\infty$ being $\R^d$, an exterior domain in $\R^d$ or a cone in $\Rd$ and $\Omega_n$ are expanding subsets eventually exhausting $\Omega_\infty$, \eg~expanding balls. Moreover, note that in the setting of Assumption~\ref{asm:n}, the domain $\Omega_0$ in Assumption~\ref{asm:pert} corresponds to $\Omega_\infty$; the projections $P_1$, $P_2$ in \eqref{Pj.def} correspond to 
\begin{equation}\label{Pn.def}
P_n:= \chi_{\Omega_n} \cdot, \quad P_\infty = I_{L^2(\Omega_\infty)}, 
\end{equation}
respectively and the projection $P$ in \eqref{Pj.def} corresponds to $P_n$. 

To formulate the result we introduce notation $\{\nu_k\}$ for the isolated eigenvalues of $T_\infty$ with finite algebraic multiplicities $\{m_a(\nu_k)\}$, 
\begin{equation}\label{nu.inf.def}
	\spd(T_\infty) =  \{\nu_k\}, \quad m_a(\nu_k) < \infty,
\end{equation}
and for spectral projections of $T_\infty$ and $\{T_n\}$
\begin{equation}\label{Ekn.def}
E_k:= \frac1{2\pi \ii} \int_{ \gamma_k} (z-T_\infty)^{-1} \, \dd z, \qquad  E_{k,n}:=\frac1{2\pi \ii} \int_{ \gamma_k} (z-T_n)^{-1} P_n \, \dd z,
\end{equation}
where $\{\gamma_k\}$ are suitable contours around $\{\nu_k\}$. Moreover, we use notation
\begin{equation}\label{kap.n.def}
\kappa_n:=\max_{\substack{\phi \in \Ran E_k \\ \|\phi\|=1}} 
\left(
\left\| \frac{\xi_n (Q_n-Q_\infty)}{(Q_n+1)(Q_\infty+1)} \phi \right\| + \| \zeta_n \phi\| 
\right).
\end{equation}

The statements of Corollary~\ref{cor:n} follow in a straightforward way from Theorems~\ref{thm:res} and \ref{thm:rates}. The details on the claim on no pollution can be found in \cite[Chap.~IV.6]{Kato-1966} or \cite[Rem.~4.2]{Boegli-2017-42}. Notice that the last term in \eqref{EF.conv} disappears since $\widetilde P$ corresponds to $I-P_n$. Finally,  if all $T_n$, $n \in \N^*$, have compact resolvents, Corollary~\ref{cor:n} yields spectral exactness.

\begin{corollary}\label{cor:n}
Let Assumption~\ref{asm:n} be satisfied and let $P_n$, $\nu_k$, $E_k$, $E_{k,n}$ and $\kappa_n$ be as in \eqref{Pn.def}, \eqref{nu.inf.def}, \eqref{Ekn.def} and \eqref{kap.n.def}. Then the following hold as $n \to \infty$.
\begin{enumerate}[\upshape i), wide]
\item $\{T_n\}$ converge to $T_\infty$ in the norm resolvent sense (hence there is no spectral pollution): for every $z \in \rho(T_\infty)$, there is $n_z>0$ such that $z \in \rho(T_n)$, $n>n_z$, and
\begin{equation}\label{Tn.gnr}
\|(T_n-z)^{-1}P_n - (T_\infty-z)^{-1} \|_{\cB(L^2(\Omega_\infty))} = \BigO_z(\tau_n);
\end{equation}
\item spectral projections converge in norm:
\begin{equation}\label{En.conv.cor}
\|E_{k,n} - E_k\|_{\cB(L^2(\Omega_\infty))} = \BigO_k(\tau_n);
\end{equation}
\item there is spectral inclusion for isolated eigenvalues with finite algebraic multiplicities: for every $\nu_k \in \spd(T_\infty)$, as $n\to \infty$, there are exactly $m_a(\nu_k)$ eigenvalues $\mu_{k,n}^{(j)}$, $j=1, \dots, m_a(\nu_k)$, of $T_n$ in a neighborhood of $\nu_k$ (repeated according to their algebraic multiplicities)  and 
\begin{equation}\label{la.conv.n}
|\nu_k-\hat \mu_{k,n}| = \BigO_k \left(\kappa_n\right), \quad 	\hat \mu_{k,n} := \frac1{m_a(\nu_k)} \sum_{j=1}^{m_a(\nu_k)} \mu_{k,n}^{(j)};
\end{equation}
\item (generalized) eigenfunctions converge in norm: for very $\psi \in \Ran(E_k)$
\begin{equation}\label{EF.conv.n}
	\|\psi - E_{k,n} \psi\| = \BigO_k \left(\kappa_n\right), \quad n \to \infty.
\end{equation}
\end{enumerate}
\end{corollary}

\begin{remark}\label{rem:res.z}
Based on \eqref{TT0.res.est.z} and an additional Neumann series argument, one can give an estimate on $n_z$, the constant in \eqref{Tn.gnr} and hence also in the other estimates. Namely, using notation $\varrho_n(z):= \|(T_n-z)^{-1}P_n - (T_\infty-z)^{-1} \|_{\cB(L^2(\Omega_\infty))}$, we obtain
\begin{equation}\label{res.z.n}
\varrho_n(z) \leq \frac{\|(T_\infty-z)(T_\infty+1)^{-1}\|\|(T_\infty+1)(T_\infty-z)^{-1}\|}{1-	|z+1| \|(T_\infty+1)(T_\infty-z)^{-1}\| \varrho_n(-1)} \varrho_n(-1).
\end{equation}
if the denominator is positive.
%
%\begin{equation}
%	|z+1| \|(T_\infty+1)(T_\infty-z)^{-1}\| \varrho_n(-1) <1,
%\end{equation}
\hfill $\blacksquare$
\end{remark}

\section{Domain truncations}
\label{sec:dom.tr}

We apply first Corollary~\ref{cor:n} to domain truncations of a given Schr\"odinger operator with the underlying initial domain $\Rd$ to bounded expanding domains $\{\Omega_n\}$. It is easy to verify that the results can be reformulated for other initial domains like exterior domains in $\Rd$, cones in $\Rd$ etc. 

\begin{asm-sec}\label{asm:dom.tr}
Let $\Omega_\infty:=\Rd$, let $\{\Omega_n\} \subset \Rd$ be bounded and open sets and let there exist a sequence $\{r_n\} \subset (0, \infty)$ such that $r_n \nearrow + \infty$ and 
\begin{equation}\label{Omn.rn}
	B_{r_n+2}(0) \subset \Omega_n, \qquad n \in \N.
\end{equation} 
Let $Q_\infty$ satisfy Assumption~\ref{asm:Q} (with $Q:=Q_\infty$) and suppose in addition that
\begin{equation}\label{Q.dom.unbd}
\lim_{R \to \infty} \essinf_{|x|>R, x \in \Omega}|Q_\infty(x)| = +\infty.
\end{equation}
\hfill $\blacksquare$
\end{asm-sec}
Note that \eqref{Q.dom.unbd} implies that all operators in Theorem~\ref{thm:T.basic} have compact resolvent, thus we obtain spectral exactness. Under a stronger version of Assumption~\ref{asm:Q}, the spectral exactness and estimate on the convergence rate of eigenvalues and eigenfunctions were proved in \cite{Boegli-2017-42}; the norm resolvent convergence was based on a collective compactness which did not provide the convergence rate. 
%In Example \eqref{ex:dom.tr.Q.pol} below precise rates are discussed. 
%
\begin{theorem}\label{thm:dom.tr}
Let Assumption~\ref{asm:dom.tr} be satisfied and let $T_n$ be the Dirichlet realizations of $-\Delta +Q$ in $L^2(\Omega_n)$, $n \in \N^*$, respectively. Then the statements of Corollary~\ref{cor:n} hold with $\{\nu_k\} := \sigma(T_\infty)$ and
\begin{equation}\label{tau.kap.dom.tr}
\tau_n = \left\| \widetilde \chi_{B_{r_n}(0)} |Q|^{-1} \right \|_{L^\infty}, 
\quad \kappa_n =  \max_{\substack{\phi \in \Ran E_k \\ \|\phi\|=1}} \| \widetilde \chi_{B_{r_n}(0)} \phi\|_{L^2}, \quad n \in \N.
\end{equation}

Suppose in addition that  $(T_\infty+1)^{-1} \in \cS_p(L^2(\Rd))$ for some $p>0$. 
Then, for any fixed $q > p$ and for any $m \in \N$ with $m> p$
\begin{equation}\label{Tn.Sp.trace}
\begin{aligned}
\|(T_n+1)^{-1}P_n - (T_\infty+1)^{-1} \|_{\cS_q} &= \BigO \left(\tau_n^{1-\frac pq} \right), 
\\
\left| \Tr \Big( (T_n+1)^{-m}P_n  - (T_\infty+1)^{-m} \Big) \right| &= \BigO \left(\tau_n^{m-p} \right), \quad n \to \infty.
\end{aligned}
\end{equation}
\end{theorem}
\begin{proof}
The conditions \ref{asm:n.dom} and \ref{asm:n.Qn} in Assumption~\ref{asm:n} are clearly satisfied, in particular since $Q_n = Q_\infty \restriction \Omega_n$. 
The cut-offs $\xi_n$ can be constructed by a suitable mollification of $\chi_{B_{r_n}(0)}$ so that $\supp(\xi_n) \subset B_{r_n+1}(0)$ and the condition \eqref{xin.nabla.unif} hold. Notice that then we have also $\chi_{\Omega_n}\xi_n=\xi_n$ in $\Rd$. Moreover, since $\Omega_n$ are bounded, we have $\Dom(t_n) = W^{1,2}_0(\Omega_n)$, $n \in \N$, thus every $f \in \Dom(T_n) \subset W^{1,2}_0(\Omega_n)$ satisfies $\xi_n f \in \Dom(t_\infty) =  W^{1,2}(\Rd)\cap \Dom(|Q_\infty|^\frac12)$. On the other hand, every $g \in \Dom(T_\infty) = W^{2,2}(\R^d) \cap \Dom(Q_\infty)$ satisfies $\xi_n g \in \Dom(t_n)$ using the properties of $\xi_n$. Thus the conditions of Assumption~\ref{asm:n}.\ref{asm:n.xin} are satisfied as well. The functions $\{\zeta_n\}$ from \eqref{tau.n.def} satisfy $\zeta_n \leq \chi_{B_{r_n}(0)}$, $n \in \N$. Obviously $\xi_n(Q_n-Q_\infty) = 0$ in our case, thus we get from \eqref{Q.dom.unbd} that the condition Assumption~\ref{asm:n}.\ref{asm:n.Q.conv} is satisfied with $\tau_n$ and $\kappa_n$ as in \eqref{tau.kap.dom.tr}.
In summary the assumptions of Corollary~\ref{cor:n} are satisfied.

By the equivalence \eqref{Sp.equiv.2} and a min-max argument, see \eg~\cite[Prop.~4 in Sec.~XIII.15]{Reed4}, applied to the self-adjoint Dirichlet realizations $S_n:=-\Delta + |Q|$ in $L^2(\Omega_n)$ and $S_\infty:= -\Delta + |Q|$ in $L^2(\Rd)$, we obtain that $\|(T_n+1)^{-1}\|_{\cS_p} \ls \|(T_\infty+1)^{-1}\|_{\cS_p}$ and so the convergence in $\cS_q$ follows from \eqref{T.Sp} and the already established norm resolvent convergence.

Next we use H\"older inequality for Schatten norms and the telescope sum formula with $r \in \N$ (valid for any two bounded operators $A$, $B$, see \eg~\cite{BelHadjAli-2011})
\begin{equation}
	A^r - B^r = \sum_{j=0}^{r-1} A^{r-1-j} (A-B)B^j.
\end{equation}
Denoting $c_p:=\sup_{n \in \N_*} \|(T_n+1)^{-1}\|_{\cS_p}$, we arrive at
\begin{equation}
\| (T_n+1)^{-r}P_n  - (T_\infty+1)^{-r} \|_{\cS_1} \leq r c_p^{r-1} \|(T_n+1)^{-1}P_n  - (T_\infty+1)^{-1} \|_{\cS_q}
\end{equation}
where $(r-1)/p+1/q =1$. Finally, the second claim in \eqref{Tn.Sp.trace} follows by $|\Tr A| =\|A\|_{\cS_1}$ for $A \in \cS_1$, see \eg~\cite[Chap.~3]{Simon-2005}.
\end{proof}
	
\begin{example}\label{ex:dom.tr.Q.pol}
For $Q$ from Example~\ref{ex:Schr.basic}, consider $T_\infty = -\Delta + Q$ in $L^2(\R^d)$ and its truncations $T_n$, $n \in \N$, for which \eqref{Omn.rn} holds with $r_n=n$. Then Theorem~\ref{thm:dom.tr} and Example~\ref{ex:EF.dec} yield that (with some $c_k>0$)
\begin{equation}
	\tau_n = \BigO(n^{-\gamma}), \quad \kappa_n = \BigO(\exp(-c_k n^{\frac \gamma2 +1})). 
\end{equation}
Moreover, for any $r \in \N$ such that $r > p$ for $p>p_{\gamma,d}$, see \eqref{Qg.Sp},
\begin{equation}\label{Q.g.trace}
\left| \sum_{j=1}^\infty \frac{1}{(\la_{j,n}+1)^r}   - \sum_{j=1}^\infty \frac{1}{(\la_{j,\infty}+1)^r} \right| = \BigO \left(n^{-\gamma(r-p)} \right), \quad n \to \infty,
\end{equation}
where $\sigma(T_n) =\{\la_{j,n}\}_j$, $n \in \N_*$. 
For real $Q$, Hoffman-Wielandt inequality yields
\begin{equation}
	\sum_{j =1}^\infty \left| \frac1{\widetilde \la_{j,n}+1} - \frac1{\la_{j,\infty}+1}\right|^q 
	= \BigO \left(n^{-\gamma(q- p)} \right), \quad n \to \infty,
\end{equation}
for any $q > p> p_{\gamma,d}$, and suitable enumerations $\{\widetilde \la_{k,n}\}_k$ of $\{\sigma(T_n)\}$, see \cite{Kato-1987-111}.
\hfill $\blacksquare$
\end{example}

Our results are applicable for truncations of operators $T$ without compact resolvent to suitable unbounded domains. Roughly speaking, one can truncate the parts of domain where the potential $Q$ is unbounded as $x \to \infty$. We illustrate this on simple self-adjoint one dimensional example; generalizations to complex potentials and more dimensions are straightforward.

\begin{example}\label{ex:tr.ex}
Consider
\begin{equation}
T_\infty = - \frac{\dd^2 }{\dd x^2} + x \e^{ x},
\quad 
\Dom(T_\infty) :=W^{2,2}(\R) \cap \Dom(x \e^{x}),
\end{equation}
and (Dirichlet) truncations $T_n:= \Dti + x \e^x$ to $\Omega_n=(-\infty,n+2)$, $n \in \N$.
Notice that standard arguments show that the essential spectrum of $T_n$ is $[0,\infty)$ for all $n \in \N^*$ and that $T_\infty$ has at least one negative eigenvalue. 

Regarding Assumption~\ref{asm:n}, we construct cut-offs $\xi_n$ as mollifications of $\chi_{(-\infty,n]}$ such that $\xi_n \restriction  (-\infty,n] = 1$, $\supp \xi_n \subset (-\infty,n+1)$ and 
$\sup_{n} ( \|\nabla \xi_n\|_{L^\infty} + \|\Delta \xi_n\|_{L^\infty}) < \infty$.
The conditions in Assumption~\ref{asm:n} are then easy to verify. Hence  Corollary~\ref{cor:n} yields the norm resolvent convergence of $T_n$ to $T_\infty$ with $\tau_n = n^{-1} \e^{-n}$. Thereby we obtain also no spectral pollution and spectral inclusion of discrete eigenvalues of $T_\infty$ with the eigenvalue convergence rate determined by $\kappa_n = \exp(-c_\la n \e^{\frac n2})$ with some $c_\la>0$, see Example \ref{ex:EF.dec.ex}.
\hfill $\blacksquare$
\end{example}

At the first sight, the spectral exactness of the domain truncations established in Theorem~\ref{thm:dom.tr}, does not seem to be troublesome for complex potentials. Nevertheless, it is crucial to observe that while the obtained rates depend on the decay of $|Q|^{-1}$ and of (generalized) eigenfunctions at infinity, the estimate of $z$-dependent constants, indicated in the $\BigO$-notation in Theorem~\ref{thm:dom.tr}, is very different for real and complex potentials. The occurring quantities are related to the norm of resolvent of $T$ and $T_n$ in $z$ as well as to the norms of the spectral projections corresponding to a given eigenvalue $\la$, see Theorem~\ref{rem:sect}, Remarks~\ref{rem:sect}, \ref{rem:res.z}, Theorem \ref{thm:rates} and \eqref{CD.mu}. The behavior of these can be very far from the self-adjoint case already for the simplest one dimensional cases with $Q(x)=\ii x^n$, $n \in \N$, where we have exponential growth when $z \to \infty$ along relevant rays in $\C$, see \eg~\cite{Davies-1999-200,Dencker-2004-57,Henry-2014-4,Henry-2014-15,Krejcirik-2015-56}.

The second observation and the main motivation of this paper is the presence of diverging eigenvalues of $T_n$ as $n \to \infty$, see Figure~\ref{fig:ix2.intro} and Section~\ref{sec:ex.div} for examples.

\section{Diverging eigenvalues in domain truncations}
\label{sec:div.ev}

In this section we analyze diverging eigenvalues in truncations $T_n$ of $T = - \Delta + Q$ in $L^2(\Omega)$ to a certain type of open domains $\{\Omega_n\} \subset \R^d$ with corners. More precisely, we assume that after suitable shifts and rescalings the domains $\{\Omega_n\}$ lie in a cone $\Gamma$ in $\R^d$ symmetric around positive $x_1$-axis and they exhaust $\Gamma$ eventually. If domains $\{\Omega_n\}$ and $Q$ satisfy Assumption~\ref{asm:dom.tr} in addition, then the approximation is spectrally exact, however, we do not make such assumptions here.

To identify the diverging eigenvalues, a combination of suitable unitary transforms (translation and scaling) is performed, following the ideas in \cite[Thm.~3.1]{Beauchard-2015-21}.
This procedure explained in the model case in Example~\ref{ex:Airy} reveals a suitable limiting operator and hence asymptotic formulas for diverging eigenvalues.
% In this model case, the eigenvalue convergence is justified immediately by Corollary~\ref{cor:n}.
%
%
% There it was in particular showed that for $T_n$ in \eqref{ix.Tn.def}, 
%%
%%
%\begin{equation}\label{ev.ix.lim}
%	\lim_{n \to \infty} (\inf \Re \sigma(T_n)) = \frac{|\mu_1|}{2}
%\end{equation}
%%
%where $\mu_1 \approx -2.338$ is the first zero of the Airy function $\Ai$. 

\begin{example}[Imaginary Airy operator]
	\label{ex:Airy}
Consider $\Omega_n := (-s_n,s_n)$ with some $\{s_n\} \subset \R_+$ with $s_n \nearrow + \infty$ and
	\begin{equation}\label{ix.Tn.def}
	T_n=\Dti+\ii x,\quad \Dom(S_n) = W^{2,2}(\Omega_n) \cap W_0^{1,2}(\Omega_n). 
	\end{equation}
	The translation $x \mapsto x-s_n$ leads to unitarily equivalent operators	%
	\begin{equation}\label{transn}
		\begin{aligned}
			\Dti+\ii x-\ii s_n=:S_n-\ii s_n, \quad \Dom(S_n) = W^{2,2}(\Sigma_n) \cap W_0^{1,2}(\Sigma_n),
		\end{aligned}
	\end{equation}
	where $\Sigma_n = (0,2 s_n)$. 
	Theorem~\ref{thm:dom.tr} implies that $S_n$ converges to $\Sa = \Dti+\ii x$ in $L^2(\R_+)$ in the norm resolvent convergence sense, hence the approximation is spectrally exact and so the spectra of $S_n$ contain asymptotically the eigenvalues $\{\nu_k+\rho_{k,n}\}_k$ where $\sigma(\Sa) = \{\nu_k\} \neq \emptyset$, see Example~\ref{ex:Airy.cone} for more details, and with some $c_k>0$ we have $\rho_{k,n} =\BigO_k( \exp(-c_k s_n^{3/2}))$ as $n\to\infty$.
	Thus, by spectral mapping and \eqref{transn}, we obtain that spectra of $T_n$ contain asymptotically the eigenvalues
	\begin{equation}
		\la_{k,n}=(\nu_k+\rho_{k,n})-\ii s_n, \quad n \to \infty.
	\end{equation}

	A second set of diverging eigenvalues of $T_n$ can be obtained by two transformations $x \mapsto x+s_n$ and $x \mapsto -x$. Isospectral partner of \eqref{ix.Tn.def} in $L^2(\Sigma_n)$ then takes the form $\Dti-\ii x+\ii s_n$, which is the adjoint operator of \eqref{transn}. Hence its spectrum contains asymptotically the conjugate eigenvalues $\{\overline{\la_{k,n}}\}$.
	
	In summary, the spectrum of $T_n$ contains asymptotically the complex conjugated eigenvalues $\{\la_{k,n}, \overline{\la_{k,n}}\}$, \cf~Figure~\ref{fig:ix}.
	\begin{figure}[!htb]
		\includegraphics[width=0.47\textwidth]{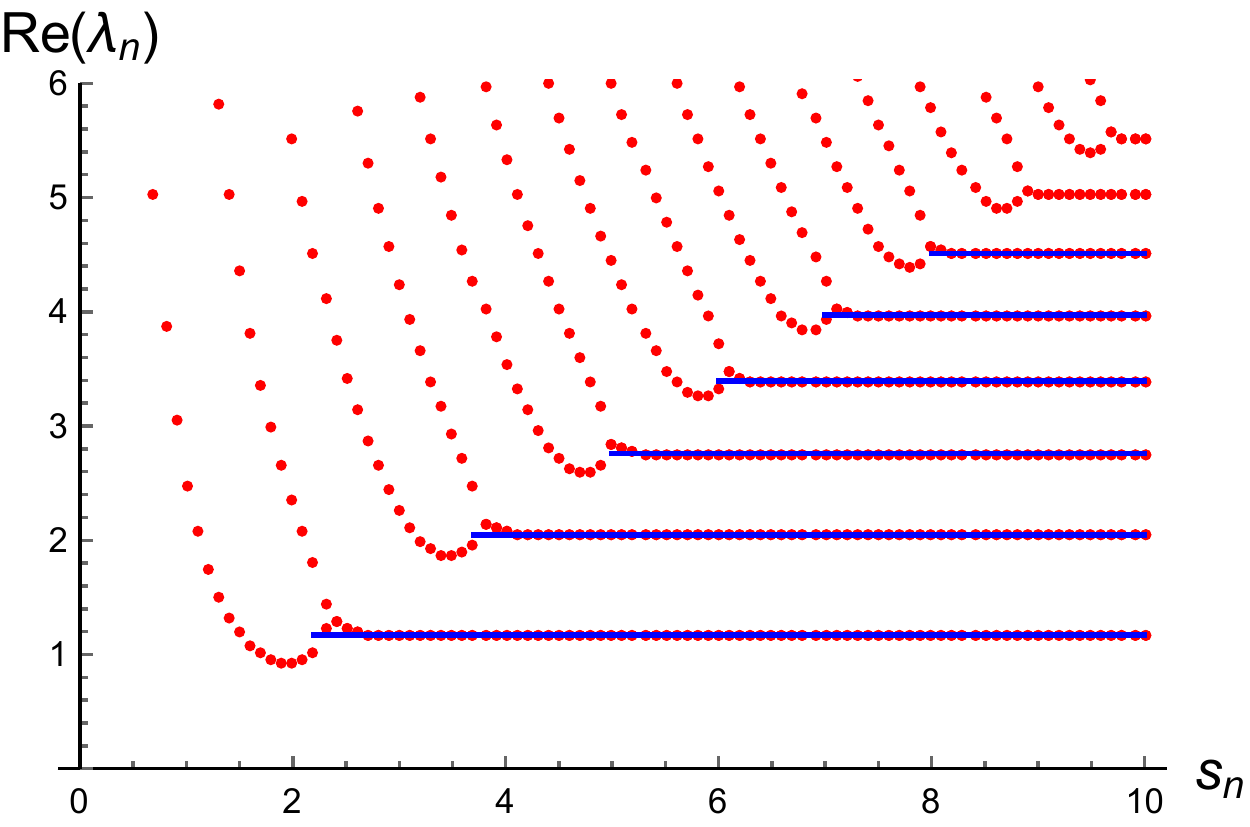}
		\hfill
		\includegraphics[width=0.47\textwidth]{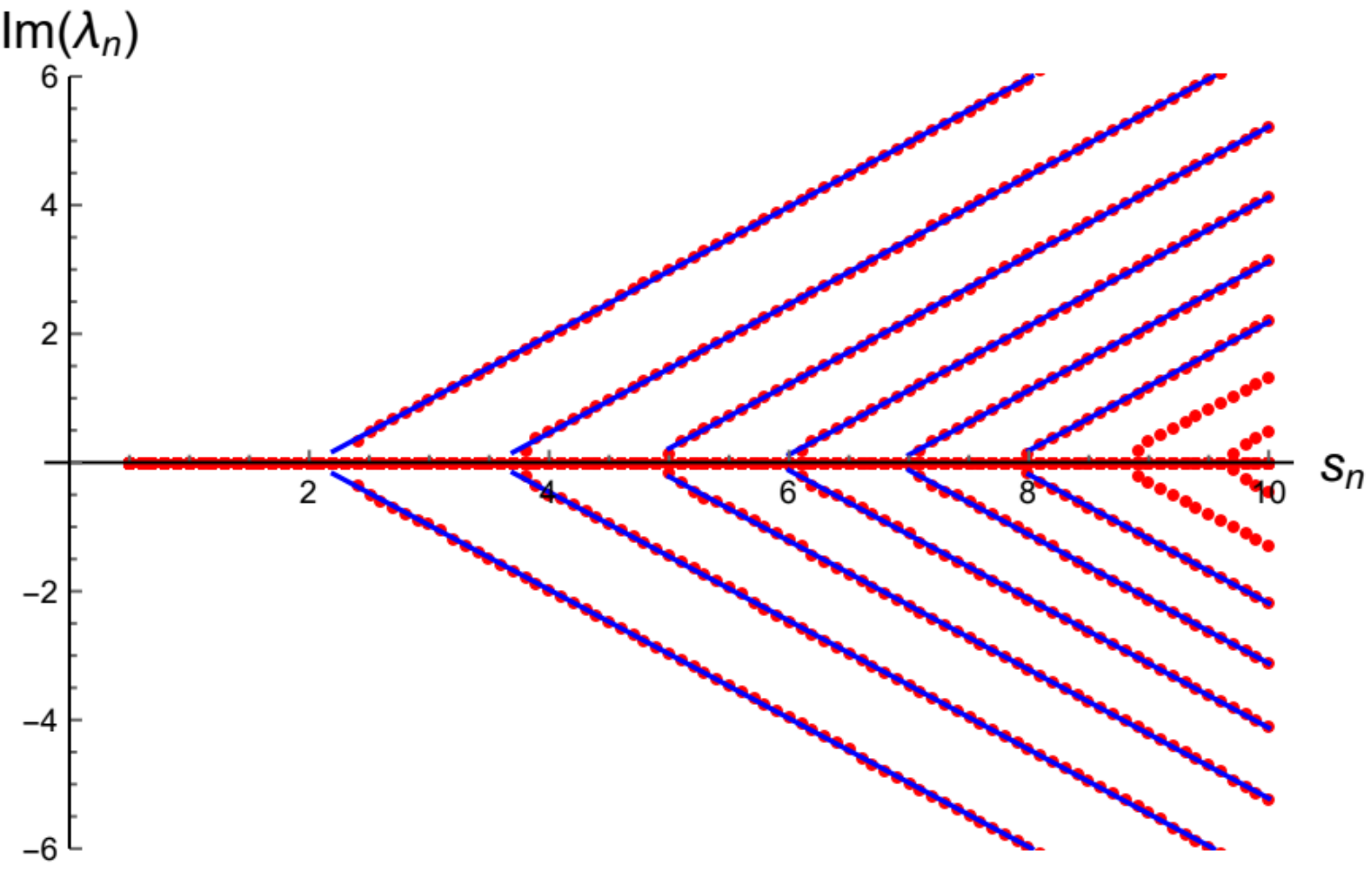}
		\caption{Real (left) and imaginary (right) parts of eigenvalues of domain truncations of imaginary Airy operator $\Dti + \ii x$ in $L^2(\R)$ to $L^2((-s_n,s_n))$, $s_n=0.1n$, $n=5,6,\dots,100$; subject to Dirichlet boundary conditions. Six asymptotic curves (blue) to which the eigenvalues converge; see Example~\ref{ex:Airy} and Section~\ref{subsec:ex.1D}.}
		\label{fig:ix}
	\end{figure}
\hfill $\blacksquare$
\end{example}

\subsection{General convergence result}
\label{subsec:div.gen}

As Example~\ref{ex:Airy} suggests, an essential role is played by complex Airy operators on cones, the eigenvalues of which enter the main asymptotic term of the diverging eigenvalues. 

For $\theta \in [0,\pi/2)$, we introduce notation for cones symmetric with respect to the positive $x_1$-axis,
\begin{equation}\label{Gamma.def}
	\begin{aligned}
		\Gamma_\theta(x_0) & := \{  (x_1, \widetilde x) \in \R^{1+(d-1)} \, : \, x_1 > x_0 \text{ and } |\widetilde x| < \tan (\theta) |x_1-x_0|\},
		\\
		\Gamma_\theta & := \Gamma_\theta(0), \qquad \Gamma_\theta(x_0) = (x_0,\infty) \text{ if } d=1.
	\end{aligned}
\end{equation}

\begin{example}[Complex Airy operator in a cone]\label{ex:Airy.cone}
	Let $\Gamma:=\Gamma_{\theta_\infty} \subset \R^d$ with some $\theta_\infty \in (0,\pi/2)$ if $d>1$. Consider the complex Airy operator in $L^2(\Gamma)$, \ie~
	\begin{equation}\label{Airy.def}
		\begin{aligned}
			\Sa:= -\Delta + \e^{\ii \omega} \, \vartheta \cdot x, 
			\quad
			\Dom(\Sa) := \Dom(\Delta_{\rm D}) \cap \Dom(|x|),
		\end{aligned}
	\end{equation}
	with $\Re e^{\ii \omega} \geq 0$ and $|\vartheta| =1$ such that 
	\begin{equation}\label{vt.cond}
		\forall \theta \in \ov \Gamma, \ |\theta| =1, \ \theta \cdot \vartheta >0.	
	\end{equation}
	Notice that \eqref{vt.cond} implies that there is a constant $c>0$ such that
	\begin{equation}\label{vtx.le}
		|\vartheta \cdot x| > c |x|, \quad x \in \Gamma.
	\end{equation}
	
	The property \eqref{vtx.le} guarantee that the potential in $\Sa$ satisfies \eqref{Q.unbd} in $\Gamma$, hence the resolvent of $\Sa$ is compact. Moreover, the eigenvalues of $\Sa$ can be related to the eigenvalues of the corresponding self-adjoint Airy operator in $L^2(\Gamma)$ by complex scaling, namely
	\begin{equation}\label{nuk.def}
		\sigma(\Sa) = \e^{\frac23 \ii \omega}\sigma (-\Delta_{\rm D} + \vartheta \cdot x) =: \{\nu_k\}_{k \in \N}.
	\end{equation}
	In particular for $d=1$, the eigenvalues of the (Dirichlet) Airy operator $\Dti + \e^{\ii \omega}x$ in $L^2(\R_+)$ are explicit in terms of the zeros $\{\mu_k\}$ of Airy function $\Ai$
	\begin{equation}
		\nu_k = \e^{\ii (\frac{2\omega}{3}-\pi)} \mu_k, \quad k \in \N; 
	\end{equation}
	 $\{\mu_k\}$ are ordered decreasingly and they diverge to $-\infty$, see~\eg~\cite{Almog-2008-40} for more details. 

	The eigenfunctions $\{\psi_k\}$ of $\Sa$ related to $\{\nu_k\}$ exhibit a super-exponential decay 
	\begin{equation}\label{ck.Airy}
		\exp( c_k |\cdot|^\frac 32 ) \psi_k \in L^2(\Gamma)
	\end{equation}
	with some $\{c_k\} \subset (0,+\infty)$, see Theorem~\ref{thm:EF.dec}, Example~\ref{ex:EF.dec} and \eqref{vtx.le}.
	\hfill $\blacksquare$
\end{example}

Under the following assumptions, we follow the steps in Example~\ref{ex:Airy} with an additional scaling to reveal a suitable limiting Airy operator in a cone. 

\begin{asm-sec}\label{asm:div}
Let $\Gamma = \Gamma_{\theta_\infty}$ be a cone as in \eqref{Gamma.def} with $0 < \theta_\infty < \pi/2$ if $d>1$. Let $\{\Omega_n\} \subset \R^d$ be open, let $\Omega:= \cup_{n \in \N} \Omega_n$ and let $Q \in C^1(\ov \Omega)$ satisfy Assumption~\ref{asm:Q} on every $\Omega_n$, $n\in \N$. 
Suppose further that
\begin{enumerate}[\upshape i),wide]
\item \label{asm:div.sn} there exists $\{s_n\} \subset \R_+$ with $s_n \nearrow +\infty$ such that for $p_n:=(s_n,0) \in \R^{1+(d-1)}$ 
\begin{equation}
|\nabla Q(-p_n)| \neq 0, \quad n \in \N; 
\end{equation}
\item \label{asm:div.Sign} shifted and scaled domains exhaust the cone $\Gamma$ eventually:

\noindent
there is $\{r_n\} \subset \R_+$ with $r_n \nearrow +\infty$ such that for 
\begin{equation}\label{t.Om.sigma.def}
\Sigma_n :=	\sigma_n^{-1} \left( p_n + \Omega_n \right), \quad \sigma_n := |\nabla Q(-p_n)|^{-\frac13}, \quad n \in \N,
\end{equation}
we have
\begin{equation}\label{Om.G.exhaust}
\Sigma_n \cap B_{r_n+2}(0) =  \Gamma \cap B_{r_n+2}(0), \quad n \in \N;
\end{equation}
\item \label{asm:div.Qn.sep} the potentials
\begin{equation}\label{Qn.div}
Q_n(x)  := \sigma_n^2 \left( Q(\sigma_nx-p_n)-Q(-p_n) \right), \quad x \in \Sigma_n,	
\end{equation}
satisfy Assumption~\ref{asm:n}.\ref{asm:n.Qn};
\item \label{asm:div.Qn.conv} linear approximation of $\{Q_n\}$ converge: 
	there is $\omega \in [-\pi/2,\pi/2]$ and $\vartheta$ is such that for all $\theta \in \ov \Gamma$ with $|\theta| =1$, we have $\theta \cdot \vartheta >0$ and such that
	\begin{equation}\label{omega.lim}
		\lim_{n \to \infty}	\frac{\nabla Q(-p_n)}{|\nabla Q(-p_n)|} = \e^{\ii \omega} \vartheta
	\end{equation}
	and
	\begin{equation}\label{iotan.def}
	 \iota_n := \left\|
	\frac{Q_n(x)-\e^{\ii \omega} \vartheta \cdot x}{(Q_n(x)+1)(\e^{\ii \omega} \vartheta \cdot x+1)}
	\right\|_{L^\infty(\Sigma_n)} = o(1), \quad n \to \infty.
	\end{equation}
	\hfill $\blacksquare$	
\end{enumerate}
\end{asm-sec}
To check the condition \eqref{Qn.sep}, it is convenient to use $\sigma_n x=:y  \in p_n+\Omega_n$ as
\begin{equation}\label{Qn.form1}
|\nabla Q_n(x)| = \frac{|\nabla Q(y-p_n)|}{|\nabla Q(-p_n)|}, 
\qquad 
\frac{|\nabla Q_n(x))|}{|Q_n(x)|^\frac32} = \frac{|\nabla Q(y-p_n)|}{|Q(y-p_n)-Q(-p_n)|^\frac32}.
\end{equation}
Moreover, assuming that \eqref{omega.lim} is satisfied, it suffices to estimate
\begin{equation}\label{Qn.form2}
\frac{|Q(y-p_n)-Q(-p_n)- \nabla Q(-p_n) \cdot y|}{|Q(y-p_n)-Q(-p_n)|(|\nabla Q(-p_n)|^\frac13|y|+1)}, \quad y\in p_n+\Omega_n
\end{equation}
to verify \eqref{iotan.def}.

\begin{theorem}\label{thm:div}
Let Assumption \ref{asm:div} be satisfied, let $T_n$ be the Dirichlet realizations of  $-\Delta + Q$  in $L^2(\Omega_n)$, $n \in \N$, and let $\{\nu_k\}$ be the eigenvalues of the complex Airy operator $\Sa$ with $\omega$ and $\vartheta$ as in \eqref{omega.lim}. Then the spectra of $T_n$ contain asymptotically the eigenvalues (with $k \in \N$ and $j \in \{1,\dots,m_a(\nu_k)\}$)
\begin{equation}\label{la.form}
\la_{k,n}^{(j)}=|\nabla Q(-p_n)|^\frac23\left(\nu_k+\rho_{k,n}^{(j)}\right)+ Q(-p_n),\quad n\to\infty,
\end{equation}
where $\rho_{k,n}^{(j)} = o_{j,k}(1)$ as $n \to \infty$ and (with $c_k>0$ as in \eqref{ck.Airy})
\begin{equation}\label{rkn.est}
\frac1{m_a(\nu_k)} \left| \sum_{j=1}^{m_a(\nu_k)} \rho_{k,n}^{(j)}\right| = \BigO_{k}(\iota_n + \exp(-c_k r_n^\frac 32)), \quad n \to \infty.
\end{equation}
\end{theorem}
\begin{proof}
We consider $T_n$ and perform transformations analogous to those in Example~\ref{ex:Airy} to obtain suitable $S_n$. Then we use Assumption~\ref{asm:div} to apply  Corollary~\ref{cor:n} and finally obtain \eqref{la.form} by spectral mapping.

We shift the point $-p_n$ to the origin, \ie~ we translate $x \mapsto x-p_n$, and transform $T_n$ to the unitarily equivalent operator $-\Delta+ Q(x-p_n)$ in $L^2(p_n+\Omega_n)$. Next we subtract the ``absolute term'' $Q(-p_n)$ and employ the scaling $x\mapsto\sigma_n x$ leading to the operator in $L^2(\Sigma_n)$ 
\begin{equation}
\sigma_n^{-2} S_n:=\sigma_n^{-2} \left[-\Delta + \sigma_n^2 \Big(Q(\sigma_nx-p_n)-Q(-p_n)) \Big) \right].
\end{equation}

The Taylor expansion of the potential leads further to 
\begin{equation}\label{V.T}
	S_n = -\Delta +Q_n(x)  = -\Delta +\frac{\nabla Q(-p_n) \cdot x}{|\nabla Q(-p_n)|}  +\sigma_n R(x), \quad n \in \N.
\end{equation}

In the next step, we apply Corollary~\ref{cor:n} where we replace $T_n$, $n \in \N^*$, by $S_n$, $n \in \N$, $S_\infty:=\Sa$ and $\Omega_n$ by $\Sigma_n$, $n \in \N^*$. In detail, the condition \ref{asm:n.dom} of Assumption~\ref{asm:n} is satisfied with $\Sigma_\infty := \Gamma$ as in Assumption~\ref{asm:div}. The condition \ref{asm:n.Qn} holds by Assumption~\ref{asm:div}.\ref{asm:div.Qn.sep}. 

Regarding the condition \ref{asm:n.xin} of Assumption~\ref{asm:n}, let $P_n$ being the orthogonal projection in $L^2(\Gamma)$ to $L^2(\Sigma_n \cap B_{r_n+2})$. Since \eqref{Om.G.exhaust} holds, we can construct cut-offs $\xi_n \in C^2(\ov \Gamma)$ with $0\leq \xi_n \leq 1$ be such that $\xi_n \equiv 1$ on $\ov{B_{r_n}(0) \cap \Gamma}$, $\xi_n \equiv 0$ on $\ov\Gamma \setminus \ov{B_{r_n+1}(0)}$, $n \in \N$, and such that \eqref{xin.nabla.unif} is satisfied. Moreover, using the properties of $\{\xi_n\}$, it is straightforward to verify that every $f \in \Dom(S_n) \subset W^{1,2}_0(\Sigma_n) \cap \Dom(|Q_n|^\frac12)$ satisfies $\xi_n f \in \Dom(s_\infty) =  W_0^{1,2}(\Gamma)\cap \Dom(|x|^\frac12)$ and also every $g \in \Dom(S_\infty)$ satisfies $\xi_n g \in \Dom(s_n)$.

Using \eqref{iotan.def}, we can estimate
\begin{equation}
	\left\|
	\frac{\zeta_n(x)}{Q_n(x) +1} 
	\right\|_{L^\infty(\Sigma_n)}
	\leq 
	\iota_n + 
	\left\|
	\frac{\zeta_n(x)}{\e^{\ii \omega} \vartheta \cdot x +1} 
	\right\|_{L^\infty(\Sigma_n)},
\end{equation}
thus the condition~\ref{asm:n.Q.conv} of Assumption~\ref{asm:n} is satisfied with $\tau_n:= \iota_n + r_n^{-1}$.

Corollary~\ref{cor:n} yields the convergence of $S_n$ to $\Sa$ and the eigenvalue convergence with the rate given by $\kappa_n= \iota_n + \exp(-c_k r_n^{3/2})$ where the second term originates in the decay of eigenfunctions of $\Sa$, see Example~\ref{ex:Airy.cone}.

Finally, \eqref{la.form} follows by spectral mapping since 
\begin{equation}\label{Tn.Sn}
T_n = \sigma_n^{-2} U_n S_n U_n^{-1} + Q(-p_n)	
\end{equation}
where $U_n$ are unitary transformations (the shift and scaling) described above.
\end{proof}

Theorem~\ref{thm:div}, based on Corollary~\ref{cor:n}.iii), provides a result for diverging eigenvalues, obtained by the spectral mapping and convergence of operators $S_n$ to an Airy operator, \cf~\eqref{V.T} and \eqref{Tn.Sn}. By Corollary~\ref{cor:n}.iv), the corresponding eigenfunctions (after appropriate transformations) converge to the eigenfunctions of the Airy operator, thus they localize to the corner $-p_n$.

\begin{remark}\label{rem:Sn.pert}
The claim of Theorem~\ref{thm:div} remains valid (with an additional term $\iota_n'$ as below in \eqref{rkn.est}) for perturbations of $T_n$. In detail, assume that the transformed operator $S_n$, see the proof of Theorem~\ref{thm:div}, is perturbed by $W_n$ satisfying 
\begin{equation}\label{Wn.rb}
|W_n(x)| \leq a |Q_n(x)| + b, \quad x \in \Sigma_n, \quad n \to \infty,
\end{equation}
with a sufficiently small $a>0$ (depending on $\eps_\nabla$) and some $b>0$, and	
\begin{equation}\label{iotan'.def}
\iota_n' := \left\|
	\frac{W_n(x)}{(Q_n(x)+1)(\e^{\ii \omega} \vartheta \cdot x+1)}
	\right\|_{L^\infty(\Sigma_n)} = o(1), \quad n \to \infty.
\end{equation}
The condition \eqref{Wn.rb} guarantees that the graph norms of $S_n$ and $S_n+W_n$ are equivalent (uniformly in $n$) and we obtain (for a sufficiently large $z_0>0$) that 
\begin{equation}
\|(S_n+W_n+z_0)^{-1}P_n - (\Sa+z_0)^{-1}\| = \BigO \left(\iota_n + \iota_n' + r_n^{-1} \right) = o(1)
\end{equation}
as $n \to +\infty$, see \eqref{TT0.res.est}. 
\hfill $\blacksquare$	
\end{remark}

\subsection{One dimensional imaginary potentials}
\label{subsec:im.Q}

The conditions on potential $Q$ in Assumption~\ref{asm:div} are expressed implicitly in terms of $Q_n$. In one dimensional case and when $Q$ is imaginary, we give explicit conditions on $Q$ which guarantee that the Assumption~\ref{asm:div} is satisfied. To avoid working with conditions at $-\infty$  we express $Q$ in a specific way in terms of a new function $U$, namely as 
\begin{equation}\label{Q.U.def}
Q(x)= - \ii U(-x), \quad x \in \R.
\end{equation}
\begin{asm-sec}\label{asm:U}
Let $U \in C^1(\R;\R) \cap C^2((x_0,\infty))$ with a sufficiently large $x_0>0$ as below satisfy the condition \eqref{Q.asm.sep} on $\R$ with an abitrarily small $\eps_\nabla>0$ (where we replace $Q$ by $U$). Let $\Omega_n=(-s_n,t_n)$ with $s_n \nearrow +\infty$ and 
\begin{equation}\label{sn.1d}
\lim_{n \to \infty}(s_n+t_n) \sigma_n^{-1} = +\infty,	
\end{equation}
where $\sigma_n:=|U(s_n)|^{- \frac 13}$. Suppose further that
\begin{enumerate}[\upshape i), wide]
\item \label{asm:U.U'.inf}  $U$ is eventually increasing and unbounded at $+\infty$: 
\begin{equation}\label{asm:U'}
U'(x) > 0, \quad x > x_0, \qquad \lim_{x \to +\infty }U(x)= +\infty;
\end{equation}

\item $U$ has controlled derivatives: there is $\nu\geq-1$ such that
\begin{equation}\label{asm:U.nu}
U'(x) \ls U(x) x^\nu, \qquad |U''(x)| \ls U'(x) x^\nu, \qquad x>x_0,
\end{equation}
\item\label{asm:U*}
 $U$ grows sufficiently fast at $+ \infty$:
\begin{equation}\label{Ups.def}
	\Upsilon(x):= \frac{x^{\nu}}{U'(x)^{\frac13}} \to 0, \quad x \to + \infty
\end{equation}
\item \label{asm:U.neg} $U$ is relatively smaller on $(-\infty,x_0)$:  there exists $\delta_0 \in (0,1)$ such that 
\begin{equation}\label{U.left}
	\sup_{y \in (s_n-x_0,s_n+t_n)} U(s_n-y) \leq (1-\delta_0) U(s_n), \quad n \to \infty.
\end{equation}
\hfill $\blacksquare$
\end{enumerate}
\end{asm-sec}
By Gronwalls inequality, \eqref{asm:U.nu} implies that for all sufficiently large $x>0$
\begin{equation}\label{V.beh}
U(x) \ls \begin{cases}
x^{\gamma},&  \nu = -1,
\\	
\exp(\gamma x^{\nu+1}), & \nu > -1.
\end{cases}
\end{equation}
with some $\gamma>0$. Moreover, there exist constants $c_1,c_2>0$ such that for all sufficiently large  $x>0$ and for all $|\delta|\leq \tfrac 14 |x|^{-\nu}$,  we have
\begin{equation} \label{eq:asm}
c_1 U^{(j)}(x) \leq U^{(j)}(x+\delta) \leq c_2 U^{(j)}(x), \qquad j=0,1;
\end{equation}
for details see~\cite[Sec.~3.1]{Krejcirik-2019-276} and \cite[Sec.~2]{MiSiVi-2020}. 
We also remark that Assumption~\ref{asm:U}.\ref{asm:U*} is related with the condition \eqref{Q.nab.as} since, by \eqref{asm:U.nu}, 
\begin{equation}
\frac{U'(x)}{U(x)^\frac32} \ls \frac{U'(x)}{(U'(x) x^{-\nu})^\frac32} = \Upsilon^\frac32(x) \to  0, \quad x \to +\infty.
\end{equation}

\begin{theorem}	\label{thm:1D.im}
Let Assumption~\ref{asm:U} be satisfied, let $Q$, $U$ be as in \eqref{Q.U.def} and let
\begin{equation}\label{sigma.Sigma.def,1d}
\sigma_n= U'(s_n)^{-\frac13}, \qquad  \Sigma_n:=(0,(s_n+t_n) \sigma_n^{-1}).
\end{equation}
Then the potentials (with $x \in \Sigma_n$)
\begin{equation}\label{Vn.U1.def}
Q_n(x) =  \sigma_n^2 \big( Q(\sigma_n x-s_n)-Q(-s_n) \big) 
= \ii \sigma_n^2 \big(U(s_n)-U(s_n-\sigma_n x) \big), 
\end{equation}
satisfy Assumption~\ref{asm:div} with $\omega=\pi/2$, $\vartheta=1$, $\Gamma=\R_+$ and $\iota_n \ls \Upsilon(s_n)$.

Hence the spectra of Dirichlet realizations $T_n = \Dti + Q$ in $L^2(\Omega_n)$, $n \in \N$, contain asymptotically as $n \to \infty$ the eigenvalues 
\begin{equation}\label{la.1d}
\la_{k,n}=U'(s_n)^\frac23\left(\nu_k+\rho_{k,n}\right) - \ii U(s_n),
\quad 
\rho_{k,n} = \BigO_{k} \big(\Upsilon(s_n) + \exp(-c_k r_n^\frac 32) \big), 
\end{equation}
where $c_k>0$ are as in \eqref{ck.Airy} and $r_n=(s_n+t_n) \sigma_n^{-1}-2$.
\end{theorem}
\begin{proof}
Due to \eqref{sn.1d}, Assumption~\ref{asm:div}.\ref{asm:div.sn} and \ref{asm:div.Sign} are satisfied and we have $\omega = \pi/2$ and $\vartheta=1$, see \eqref{omega.lim}.
Next we verify Assumption~\ref{asm:div}.\ref{asm:div.Qn.sep}, in particular the condition \eqref{Qn.sep}. We use the variable variable $y=\sigma_n x \in (0,s_n+t_n)$ and formula \eqref{Qn.form1}.
\begin{equation}
|Q_n'(x)| \ls \frac{|U'(s_n-y)|}{U'(s_n)} \ls 1, \qquad n \to \infty.
\end{equation}
Next let $y \in [ \tfrac 14 s_n^{-\nu},s_n-x_1) \cap (0,s_n+t_n)$ with $x_1 \geq x_0$ sufficiently large so that in particular \eqref{eq:asm} holds for all $x>x_1$; a further restriction on the choice $x_1$ is below. First note that $s_n-y + \tfrac14 (s_n-y)^{-\nu} \leq s_n$, hence using \eqref{asm:U'} and \eqref{eq:asm}, we obtain 
\begin{equation}
U(s_n)-U(s_n-y) \geq \int_{s_n-y}^{s_n-y+\tfrac14(s_n-y)^{-\nu}} U'(t) \, \dd t	
\gs U'(s_n-y) (s_n-y)^{-\nu}.
\end{equation}
Thus
\begin{equation}\label{Qn.1d.1}
\frac{|U'(s_n-y)|}{|U(s_n)-U(s_n-y)|^\frac32} 
\ls 
\Upsilon^\frac32(s_n-y)
\end{equation}
and so by \eqref{Ups.def}, $x_1$ can be selected so large that, for all $y$ in the considered range, the 
right hand side of \eqref{Qn.1d.1} is arbitrarily small.

For  $y \in [s_n-x_1,s_n-x_0) \cap (0,s_n+t_n)$, it suffices to use that $U$ and $U'$ are locally bounded and $U$ is unbounded at $+\infty$, see \eqref{asm:U'}.

In the last case, $y \in [s_n-x_0,s_n+t_n)$, we use that $Q$ and thus $U$ satisfy \eqref{Q.asm.sep} and 
that $U(s_n)-U(s_n-y) \geq \delta_0  U(s_n)$ from Assumption \ref{asm:U}.\ref{asm:U.neg},
\begin{equation*}
\begin{aligned}
\frac{|U'(s_n-y)|}{|U(s_n)-U(s_n-y)|^\frac32} 
&\leq \eps_\nabla \left| \frac{U(s_n-y)}{U(s_n)-U(s_n-y)} \right|^\frac32 
	  + \frac{M_\nabla}{|U(s_n)-U(s_n-y)|^\frac32}
\\ 
& \leq \eps_\nabla \left(1+ \frac{1}{\delta_0}\right)^\frac32 + \frac{M_\nabla}{\delta_0 U(s_n)};
\end{aligned}
\end{equation*}
note that $\eps_\nabla$ can be taken arbitrarily small by assumption and the second term decays. Hence, 
putting estimates above together we obtain that the condition \eqref{Qn.sep} is indeed satisfied.

Finally, we show that Assumption~\ref{asm:div}.\ref{asm:div.Qn.conv} is satisfied using \eqref{Qn.form2}.  In the first case,  $y \in (0, \tfrac 14 s_n^{-\nu}) \cap (0,s_n+t_n)$, Taylor's theorem,   \eqref{asm:U.nu} and \eqref{eq:asm} yield
\begin{equation}
\left|\frac{U(s_n)-U(s_n-y)-U'(s_n) y}{U'(s_n)^\frac13 (U(s_n)-U(s_n-y))y}\right|
 \ls
 \frac{U'(s_n) s_n^\nu y^2}{U'(s_n)^\frac43 y^2 } \ls \Upsilon(s_n). 
\end{equation}
For the remaining steps, we use the estimate
\begin{equation}
\left|\frac{U(s_n)-U(s_n-y)-U'(s_n) y}{U'(s_n)^\frac13 (U(s_n)-U(s_n-y))y}\right|
\leq
\frac{1}{U'(s_n)^\frac13 y} + \frac{U'(s_n)^\frac23}{U(s_n)-U(s_n-y)}.
\end{equation}
In the case  $y \in [\tfrac14 s_n^{-\nu},s_n-x_0) \cap (0,s_n+t_n)$, using $U'(x)>0$ for $x>x_0$, mean value theorem and \eqref{eq:asm}, we get
\begin{equation}
\frac{1}{U'(s_n)^\frac13 y} + \frac{U'(s_n)^\frac23}{U(s_n)-U(s_n-y)}
\ls
\Upsilon(s_n) + \frac{U'(s_n)^\frac23}{U'(s_n)s_n^{-\nu}} 
\ls 
\Upsilon(s_n). 
\end{equation}
Finally, if $y \in [s_n-x_0,s_n+t_n)$, we obtain from \eqref{U.left} and \eqref{asm:U.nu}  that
\begin{equation*}
\frac{1}{U'(s_n)^\frac13 y} + \frac{U'(s_n)^\frac23}{U(s_n)-U(s_n-y)}
\ls
\left(\frac{s_n^{3\nu}}{U'(s_n)} \right)^\frac13 \frac{1}{s_n^{1+\nu}} + \frac{U'(s_n)^\frac23}{\delta_0 U(s_n)} 
\ls 
\Upsilon(s_n). 
\end{equation*}
Thus in summary, we obtain the estimate on $\iota_n$ in the claim.
\end{proof}			
			
In the next step, following Remark~\ref{rem:Sn.pert}, we determine a class of admissible perturbations of $U$ as in Assumption~\ref{asm:U}.

\begin{proposition}\label{prop:pert}
Let Assumption~\ref{asm:U} be satisfied. Suppose that $U_1 \in L_{\rm loc}^\infty(\R;\C)$, $U_1'\in L_{\rm loc}^\infty((x_1,\infty);\C)$ for some $x_1>0$ and (using notation of Assumption~\ref{asm:U}),
\begin{equation}\label{U1.asm}
U_1'(x) = o(U'(x)), \quad x \to +\infty, 
\quad
\|U_1\|_{L^\infty((-t_n,s_n))} = o(U(s_n)),  
\end{equation}
Then, with $\sigma_n$ as in \eqref{Vn.U1.def},
\begin{equation}
W_n(x) := \sigma_n^2 (U_1(s_n)-U_1(s_n-\sigma_n x)), \qquad x \in (0, (s_n+t_n) \sigma_n^{-1})
\end{equation} 	
satisfies the conditions \eqref{Wn.rb} and \eqref{iotan'.def} with respect to $Q_n$ as in \eqref{Vn.U1.def}. 
Hence the claim of Theorem~\ref{thm:1D.im} remains valid with (see Remark~\ref{rem:Sn.pert})
\begin{equation}\label{la.1d.pert}
\begin{aligned}
\la_{k,n}&=U'(s_n)^\frac23\left(\nu_k+\rho_{k,n}'\right) - \ii U(s_n) - U_1(s_n),
\\
\rho_{k,n}' &= \BigO_{k} \big(\Upsilon(s_n) + \iota_n'+ \exp(-c_k r_n^\frac 32) \big).
\end{aligned}
\end{equation}

In particular if the support of $U_1$ is bounded, then
\begin{equation}\label{iota.U1.supp}
	\iota_n' = \BigO(U(s_n)^{-1} \Upsilon(s_n) s_n^{\nu-1}), \quad n \to \infty.
\end{equation}

\end{proposition}

\begin{proof}
Using the first assumption in \eqref{U1.asm}, for any $\eps_0>0$, there exists a sufficiently large $y_0>0$ such that (with $\sigma_n x = y$)
\begin{equation}
\frac{|W_n(x)|}{|Q_n(x)|}
\leq 
\frac{ \int_{s_n-y}^{s_n} \frac{|U_1'(t)|}{U'(t)} U'(t) \, \dd t }{U(s_n)-U(s_n-y)}
\leq \eps_0, \quad y \in [0, s_n-y_0]. 	
\end{equation}
From \eqref{U.left} we have that (with some $\delta(y_0)>0$)
\begin{equation}
	U(s_n) - U(s_n-y) \geq \delta(y_0) U(s_n), \quad y \in [s_n-y_0,s_n+t_n],
\end{equation}
thus 
\begin{equation}\label{Wn.Qn.2}
\frac{|W_n(x)|}{|Q_n(x)|}
\leq 
\frac{ 2\|U_1\|_{L^\infty((-t_n,s_n))}}{\delta (y_0) U(s_n)}, \quad y \in [s_n-y_0,s_n+t_n]. 	
\end{equation}
Using the second assumption in \eqref{U1.asm}, the right hand side of \eqref{Wn.Qn.2} decays as $n \to \infty$. In summary, the condition \eqref{Wn.rb} is satisfied.

Similarly, for a given $\eps_0>0$, we split the estimate to $[0,s_n-y_0]$ and $[s_n-y_0,s_n+t_n]$ with a suitable $y_0$ so that
\begin{equation}
\left|\frac{U_1(s_n)-U_1(s_n-y)}{(U(s_n)-U(s_n-y))(|U'(s_n)|^\frac13 y+1)} \right|  \leq \eps_0 + 
\frac{2\|U_1\|_{L^\infty((-t_n,s_n))}}{\delta(y_0)U(s_n)(|U'(s_n)|^\frac13 y+1)}.
\end{equation}
The validity of \eqref{iotan'.def} follows from the second assumption in \eqref{U1.asm} and since  $\eps_0$ can be taken arbitrarily small.
In the case of bounded support of $U_1$, we can take fixed $y_0$ and obtain \eqref{iota.U1.supp}.
\end{proof}

\begin{remark}[First eigenvalue correction]
\label{rem:1st.cor}
In Theorem~\ref{thm:1D.im}, the eigenvalues $\{\nu_k\}$ of $\Sa$ are simple. We introduce the truncated imaginary Airy operator in $L^2(\Sigma_n)$
\begin{equation}
\widetilde S_n := \Dti + \ii  x, \qquad \Dom(\widetilde S_n) = W^{2,2}(\Sigma_n) \cap W_0^{1,2}(\Sigma_n)  \cap \Dom(x)
\end{equation}
and write 
\begin{equation}\label{lak.nuk}
\la_{k,n} - \nu_k = \la_{k,n} - \widetilde \la_{k,n} + \widetilde \la_{k,n} - \nu_k,
\end{equation}
where $\widetilde \la_{k,n}$ are eigenvalues of $\widetilde S_n$ which converge to $\nu_k$ as $n \to \infty$. Recall that the rate of $\widetilde \la_{k,n} - \nu_k$ is super-exponential since it corresponds to the eigenvalue convergence of domain truncations for the imaginary Airy operator, see Theorem~\ref{thm:dom.tr}. Thus the main term in $\la_{k,n} - \nu_k$ usually arises from $\la_{k,n} - \widetilde \la_{k,n}$, \ie~from the difference of eigenvalues of the truncated imaginary Airy operator $\widetilde S_n$ and the ``perturbed'' truncated imaginary Airy operator $S_n$. 

The standard perturbation theory can be used to express $\la_{k,n} - \widetilde \la_{k,n}$  as
\begin{equation}\label{1st.cor}
	\la_{k,n} - \widetilde \la_{k,n} = \frac{\langle  (Q_n(x)-\ii x) \psi_{k,n}, \psi_{k,n}^* \rangle}{\langle\psi_{k,n}, \psi_{k,n}^*\rangle }   + \widetilde \rho_{k,n},
\end{equation}
where $\psi_{k,n}$, $\psi_{k,n}^*$ are eigenfunctions of $\widetilde S_n$ and $\widetilde S_n^*$ associated with $\widetilde \la_{k,n}$ and $\widetilde \la_{k,n}^*$, respectively. Finally, in \eqref{1st.cor} one can further replace $\psi_{k,n}$ by $\psi_k$, the eigenfunctions of $\Sa$ associated with $\nu_k$, producing typically only an exponentially small error due to the fast convergence of $\psi_{k,n}$ to $\psi_k$, see Theorem~\ref{thm:dom.tr} and decay estimates on eigenfunctions, see Theorem~\ref{thm:EF.dec} and Examples~\ref{ex:EF.dec}, \ref{ex:Airy.cone}. We implement these observations in Example~\ref{ex:Q.odd} below.
\hfill $\blacksquare$
\end{remark}

\section{Examples}
\label{sec:ex.div}

We illustrate the results on several examples. We start with the one dimensional ones, where Theorem~\ref{thm:1D.im} and Proposition~\ref{prop:pert} are employed. Next we analyze a radially symmetric cases truncated to annuli, still using the one dimensional results. Finally, Theorem~\ref{thm:div} is used in a two dimensional example with truncations to domains with corners.

\subsection{One dimensional examples}
\label{subsec:ex.1D}

\begin{example}[Odd imaginary potentials]\label{ex:Q.odd}
Let $U:\R \to \R$ be odd and satisfy Assumption~\ref{asm:U}; note that \eqref{U.left} holds automatically if the previous conditions are satisfied. We consider Dirichlet realizations $T_n = \Dti + \ii U$ in $L^2((-s_n,s_n))$ with $s_n \nearrow + \infty$. Since $U$ is odd, \eqref{Q.U.def} corresponds to the relation $Q=\ii U$, thus by Theorem~\ref{thm:1D.im}, the spectra of $T_n$ contain asymptotically the eigenvalues $\{\lambda_{k,n}\}_k$ in~\eqref{la.1d}. Due to the antilinear symmetry of $T_n$ ($x \mapsto -x$ together with complex conjugation, the so-called $\cP\cT$-symmetry), the spectra of $T_n$ contain also $\{\overline{\lambda}_{k,n}\}_k$. 

In particular, $U(x)=\sgn(x)|x|^\alpha$ with $\alpha>0$, satisfies Assumption~\ref{asm:U} with $\nu =-1$ and a possible lack of differentiability of $U$ at $0$ can be treated by splitting $U= \eta U + (1-\eta) U$ with $\eta \in C_0^\infty((-2,2))$ and $\eta = 1$ on $(-1,1)$. Notice that $U_1 = \eta U$ satisfies assumptions of Proposition~\ref{prop:pert}. Hence we obtain 
\begin{equation}
\lambda_{k,n}=\alpha^\frac23 s_n^{\frac{2(\alpha-1)}{3}}\left(\nu_k+\BigO_{k}\left(s_n^{-\frac{2+\alpha}{3}} \right)\right)-\ii s_n^\alpha, 
\end{equation}
and their complex conjugates; see Figures~\ref{fig:ix} and \ref{fig:ix3} for illustration in two well-known special cases (the imaginary Airy operator and imaginary cubic oscillator).
\begin{figure}[htb!]
\includegraphics[width=0.47\textwidth]{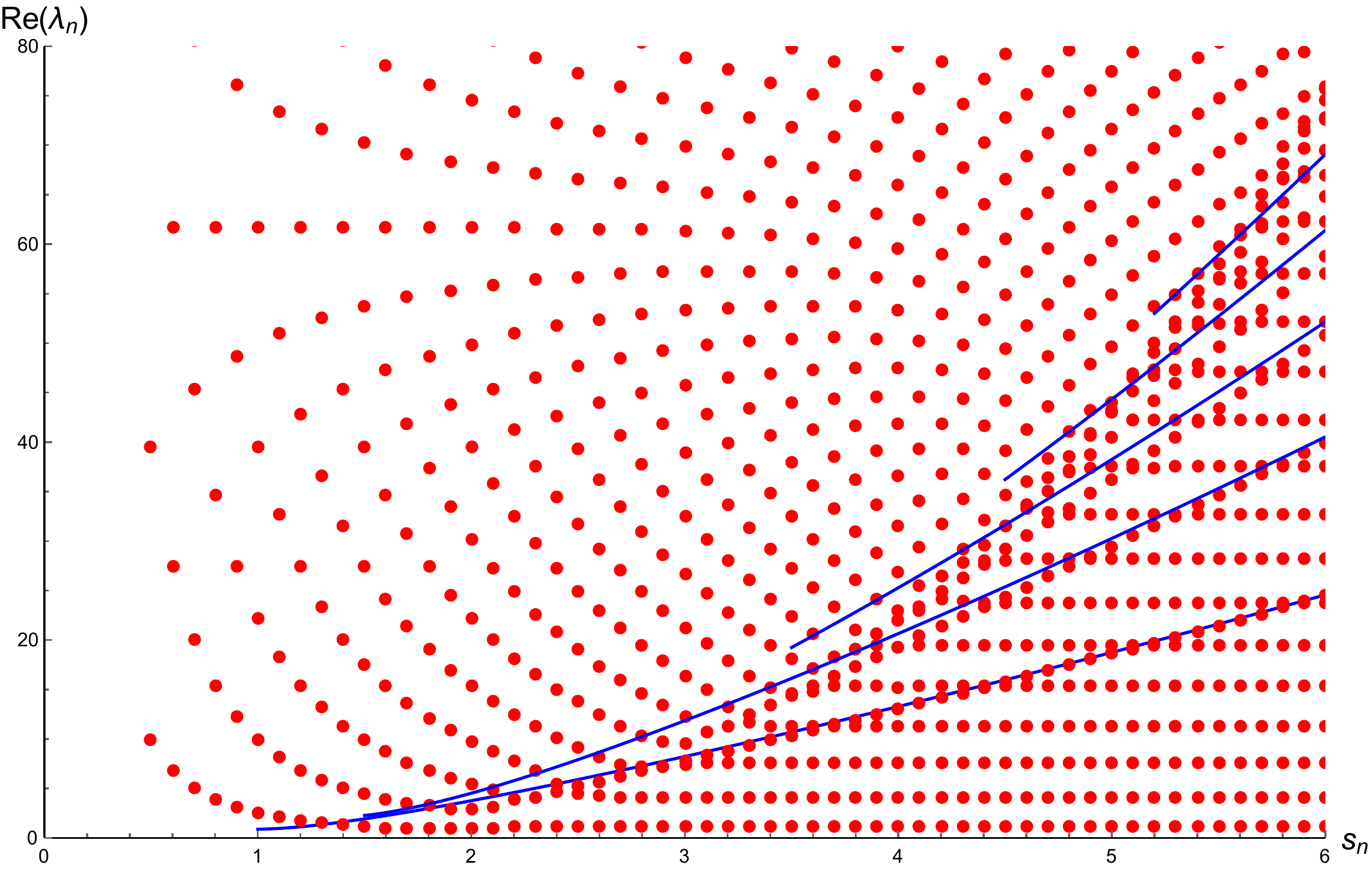}
\hfill
\includegraphics[width=0.47\textwidth]{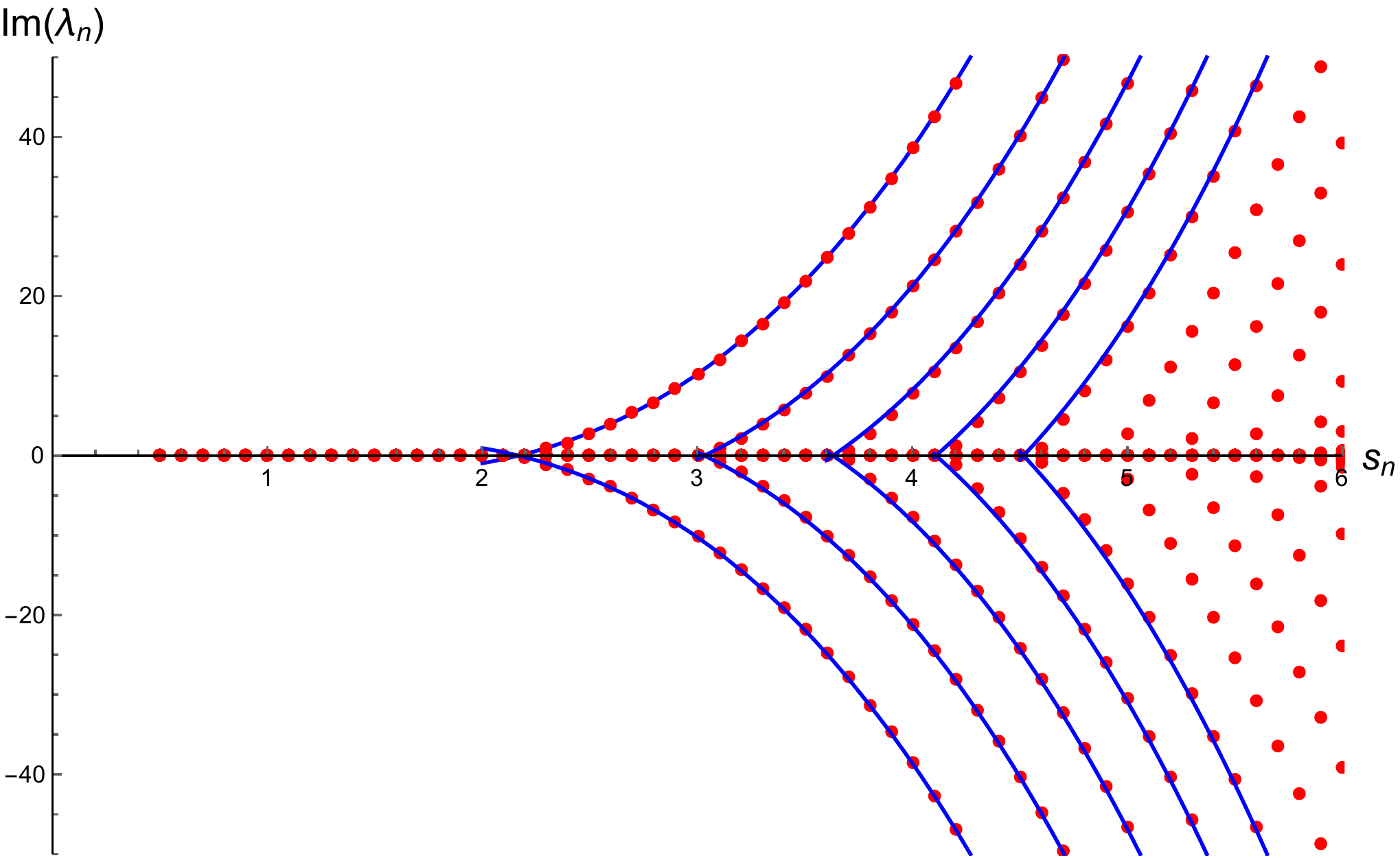}
\caption{$Q(x)=\ii x^3$: Real (left) and imaginary (right) part of the eigenvalues (red) of truncated operators $T_n$, defined on $L^2((-s_n,s_n))$ with $s_n=0.1n$, $n=5,6,\dots,60$. Asymptotic curves (blue) for $\lambda_{k,n}$, $\ov{\lambda_{k,n}}$ with first corrections for $k=1,2,\dots,5$.}
\label{fig:ix3}
\end{figure}
In Figure~\ref{fig:ix3} we plot the asymptotic curves taking into account the first correction 
with $\psi_k(y)=\Ai(\e^\frac{\ii\pi}{6}y+\mu_k)$, see Example~\ref{ex:Airy.cone} with $d=1$ and~Remark~\ref{rem:1st.cor}.
\hfill $\blacksquare$
%  
%%
%\begin{equation}
%	R_{1,n,k}=\frac{\langle(\ii \sigma_n^2(U(s_n)-U(s_n-\sigma_n y)-\ii y) \psi_k,\psi_k^*\rangle_{L^2(\R_+)}}{\langle \psi_k,\psi_k^*\rangle_{L^2(\R_+)}},
%	\label{eq:1st.cor.compute}
%\end{equation}
%%
%where 
%%
%
%%
%Specifically for $\alpha=3$
%%
%\begin{equation}
%	R_{1,n,k}^{(1)}=\frac{(\ii(-3^{-\frac13}s_n^{-\frac53}y^2+3^\frac13 s_n^{-\frac{10}{3}}y^3) \psi_k,\psi_k^*)_{L^2(\R_+)}}{(\psi_k,\psi_k^*)_{L^2(\R_+)}},\quad R_{1,n,k}^{(2)}=\overline{R_{1,n,k}^{(1)}}
%\end{equation}
%%
%The corrected asymptotic formula then holds 
%%
%\begin{equation}
%	\lambda_{k,n}^{(1)}=3^\frac23 s_n^\frac43 \nu_k - \ii s_n^3 +\ii 3^\frac13 s_n^{-\frac13} R_k^a-\ii s_n^{-2}R_k^b+s_n^\frac43 \widetilde{r}_{k,n},\quad \lambda_{k,n}^{(2)}=\overline{\lambda_{k,n}^{(1)}},
%	\label{la.ix3.cor}
%\end{equation}
%%
%where
%%
%\begin{equation}
%	R_k^a=\frac{\int_0^\infty y^2 f_k^2\dd y}{\int_0^\infty f_k^2\dd y},\quad R_k^b=\frac{\int_0^\infty{y^3 f_k}^2\dd y}{\int_0^\infty{f_k^2\dd y}}.
%\end{equation}
%%
\end{example}

\begin{example}[Even imaginary potentials]
	\label{ex:Q.even}
	Let $V: \R \to \R$ be an even and with $V'(x)>0$ for $x>0$ and consider Dirichlet realizations $T_n = \Dti + \ii V$ in $L^2((-s_n,s_n))$ with $s_n \nearrow + \infty$. Theorem~\ref{thm:1D.im} is not directly applicable because of the condition~\eqref{U.left}. Nonetheless, due to the symmetry of $V$, eigenfunctions of $T_n$ satisfy either Dirichlet or Neumann boundary conditions at $0$. Therefore we can split the spectral problem and analyze separately the spectra of
	\begin{equation*}
		\begin{aligned}
			T_n^{\rm DD} &= \Dti + \ii V(x),&&  \Dom(T_n^{\rm DD}) = W^{2,2}((-s_n,0)) \cap W_0^{1,2}((-s_n,0)),
			\\
			T_n^{\rm DN} &= \Dti + \ii V(x),&&  \Dom(T_n^{\rm DN}) = \{f \in W^{2,2}((-s_n,0)) \, : 
%			\\ &&& \qquad \qquad \qquad \qquad \qquad 
			f'(0)=f(-s_n)=0\}.
		\end{aligned}
	\end{equation*}
Introducing $U:= V \chi_{\R_+}$, we obtain that $(T_n^{\rm DD})^* = \Dti + Q$  in $L^2((-s_n,0))$ with $Q(x) = - \ii U(-x)$ as in \eqref{Q.U.def}. 

We assume that this $U$ satisfies Assumption~\ref{asm:U}, possibly with perturbations as in  Example~\ref{ex:Q.odd}, and notice that \eqref{U.left} is satisfied automatically. Then Theorem~\ref{thm:1D.im} yields that the spectra of $T_n^{\rm DD}$ contain asymptotically the eigenvalues
	\begin{equation}\label{la.even.DD}
			\lambda_{k,n}^{\rm DD}=V'(s_n)^\frac23 \Big(\overline{\nu_k}+\rho_{k,n}^{\rm DD} \Big)+\ii V(s_n), \quad n\to\infty.
	\end{equation}

It is not difficult to see that the claim of Theorem~\ref{thm:1D.im} holds also for Neumann boundary conditions at the endpoints as well as for the combinations of Dirichlet and Neumann boundary conditions. Depending on the boundary condition at $0$, the limiting operator is Dirichlet or Neuman imaginary Airy operator in $L^2(\R_+)$, in the Neumann case with eigenvalues $\{\nu_k'\}=\{\e^{\ii (\frac{2\omega}{3}-\pi)} \mu_k'\}$ where $\{\mu_k'\}$ are zeros of $\Ai'$. Thus we obtain that the spectra of $T_n^{\rm ND}$ contain asymptotically the eigenvalues
\begin{equation}\label{la.even.ND}
\lambda_{k,n}^{\rm DN}=V'(s_n)^\frac23 \Big(\overline{\nu_k}+\rho_{k,n}^{\rm ND} \Big)+\ii V(s_n), \quad n\to\infty.
\end{equation}
These two sets of eigenvalues have the same main asymptotic terms, however, the corresponding eigenfunctions of $T_n$ are very different (odd and even).
%
%For particular $V(x)=|x|^\alpha$, $x \in \R$ and $\alpha>0$, we get
%%
%\begin{equation}
%	\begin{aligned}
%		\lambda_{k,n}^{\rm DD}&=\alpha^\frac23 s_n^{\frac{2(\alpha-1)}{3}} \left(\overline{\nu_k} + \BigO_{k}\left(s_n^{-\frac{2+\alpha}{3}} \right) \right)+\ii s_n^{\alpha}, \quad n\to\infty.
%%		\\
%%		\lambda_{k,n}^{\rm DN}&=\alpha^\frac23 s_n^{\frac{2(\alpha-1)}{3}} \left(\overline{\nu_k} + \BigO_{k}\left(s_n^{-\frac{2+\alpha}{3}} \right) \right)+\ii 			s_n^{\alpha}, \quad n\to\infty
%	\end{aligned}
%\end{equation}
%
%Figure \ref{fig:ix2} illustrates the well-known case with $\alpha=2$ (rotated Davies' oscillator). 
%%
%\begin{figure}[htb!]
%\includegraphics[width=0.47\textwidth]{ix2_Re_prolozeni_1korekce.pdf}
%\hfill
%\includegraphics[width=0.47\textwidth]{ix2_Im_prolozeni_1korekce.pdf}
%\caption{$V(x)=\ii x^2$: Real (left) and imaginary (right) part of the eigenvalues (red) of truncated operators $T_n$, defined on $L^2((-s_n,s_n))$ with $s_n=0.1n$,  $n=5,6,\dots,100$. Asymptotic curves (blue) for $\lambda_{k,n}^{\rm DD},\lambda_{k,n}^{\rm DN}$ with first correction for $k=1,2,\dots,5$.}
%	\label{fig:ix2}
%\end{figure}
%
\hfill $\blacksquare$
\end{example}

\begin{example}[Imaginary exponential potential with non-empty essential spectrum]
	\label{ex:exp}
	Consider the operator $T = \Dti + \ii \e^{x}$ and its truncations $T_n$ to $(-\infty,s_n)$ with $s_n \nearrow + \infty$. Defining $U(x):=\e^{x}$ and $Q(x):=-\ii U(-x)$ as in \eqref{Q.U.def}, we obtain that $T_n^*$ is unitarily equivalent via the reflection $x \mapsto -x$ to $\Dti + Q$ in $L^2((-s_n,\infty))$. This $U$ satisfies Assumption~\ref{asm:U} with $t_n=+\infty$ and $\nu=0$, thus by Theorem~\ref{thm:1D.im}, the spectra of $T_n$ contain asymptotically the eigenvalues
	\begin{equation}\label{ev.iex}
		\lambda_{k,n}= \e^{\frac 23 s_n}  \left(\overline{\nu_k}+\BigO_k\left(\e^{-\frac13 s_n} \right) \right)+\ii \e^{s_n}, \quad n\to \infty.
	\end{equation}

	In fact, since Assumption~\ref{asm:U} is satisfied also with $t_n=s_n$, the eigenvalues \eqref{ev.iex}, with possibly different remainders, are asymptotically contained in the spectra of operators $T_n = \Dti + \ii \e^{x}$ subject to Dirichlet boundary conditions in $L^2((-s_n,s_n))$; spectra of these are illustrated in Figure~\ref{fig:iex}.
	\begin{figure}[htb!]
		\includegraphics[width=0.47\textwidth]{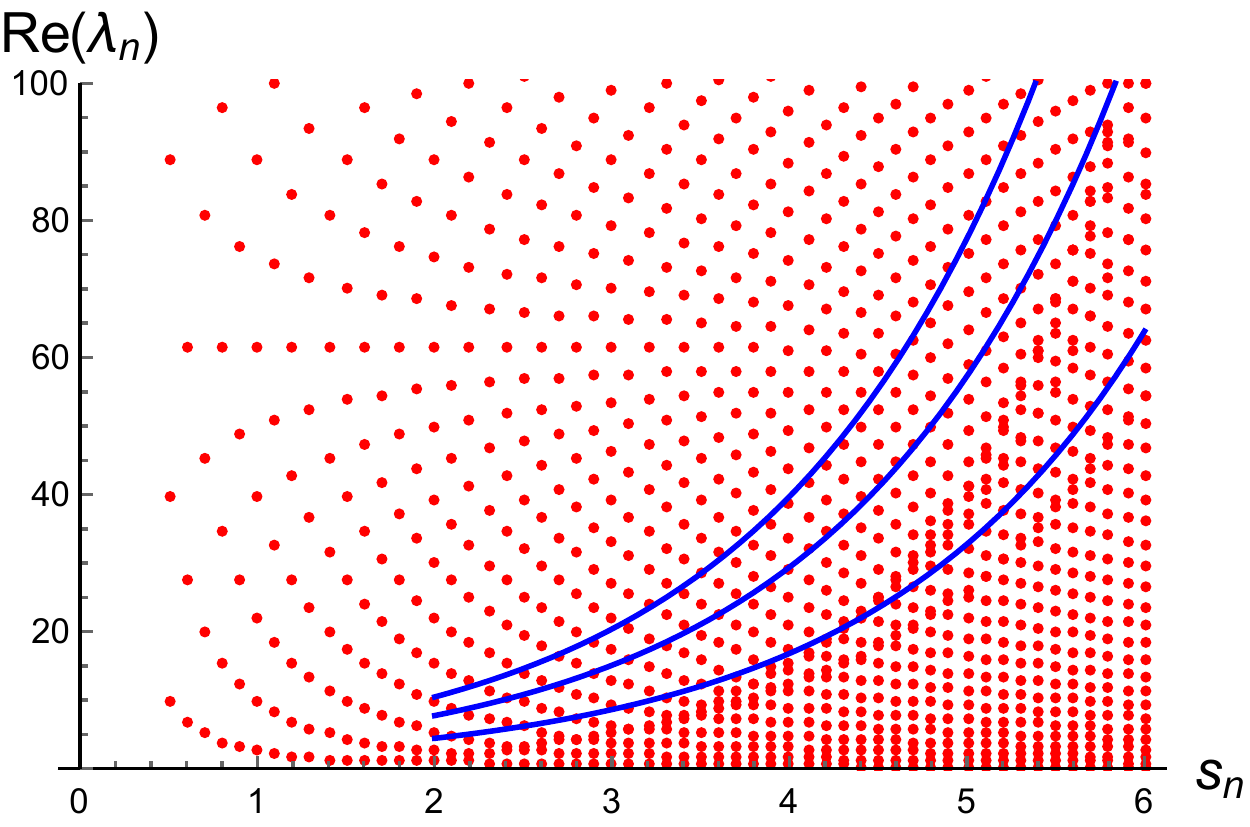}
		\hfill
		\includegraphics[width=0.47 \textwidth]{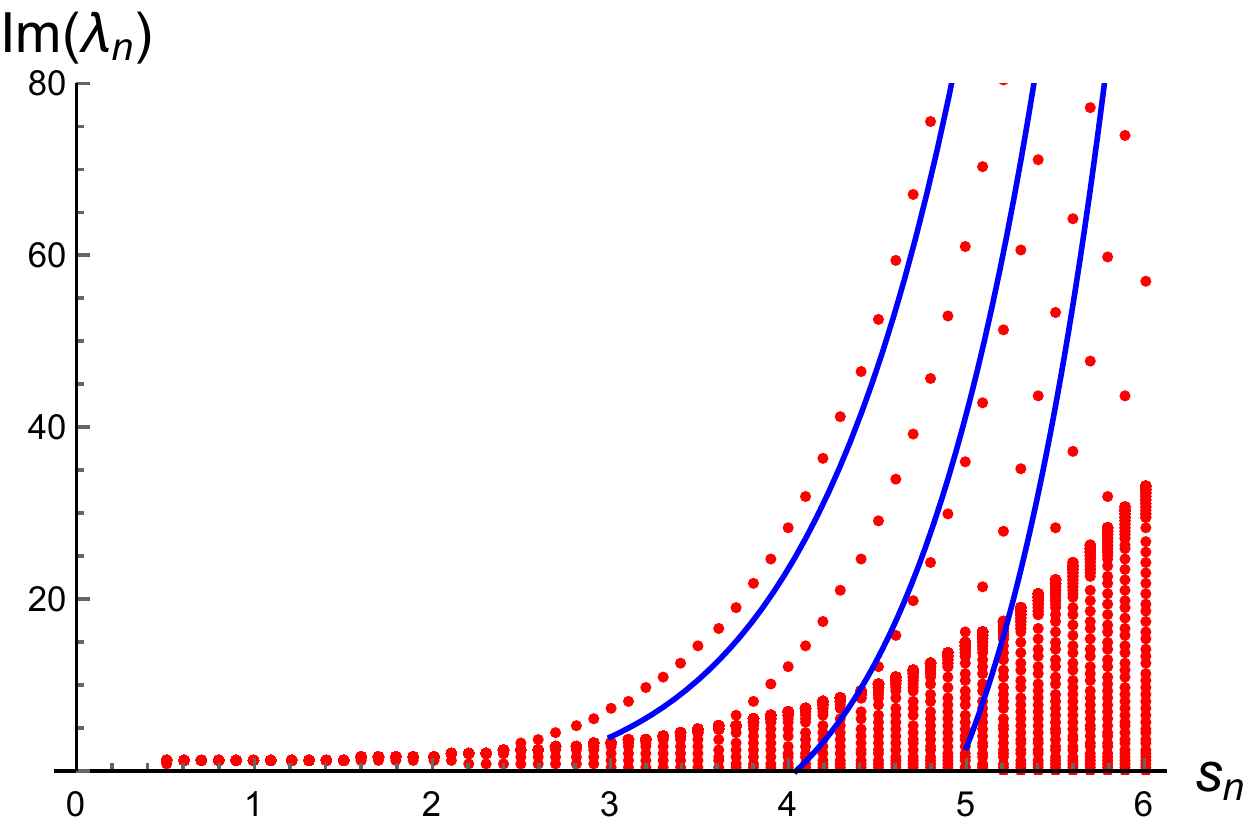}
		\caption{$U(x)=\ii \e^{x}$: Real (left) and imaginary (right) part of the eigenvalues (red) of truncated operators $T_n$, defined on $L^2((-s_n,s_n))$ with $s_n=0.1n$, $n=5,6,\dots,60$. Asymptotic curves (blue) for $\lambda_{k,n}$ with first correction for $k=1,2,3$. } 
		\label{fig:iex}
	\end{figure}
	\hfill $\blacksquare$
\end{example}

\subsection{Radially symmetric potentials on annuli}
\label{ssec:Q.radial}

Consider the exterior domain $\Omega = \R^d \setminus \ov{B_1(0)}$, a radial potential $V : \Omega \to \C$ satisfying Assumption~\ref{asm:Q} (with $Q$ replaced by $V$) and the Dirichlet realization of $T = -\Delta + V$ in $L^2(\Omega)$. Consider also the truncated  operators $T_n = -\Delta +V$ in $L^2(\Omega_n)$ with $\Omega_n = \Omega \cap B_{s_n}(0)$ and $s_n \nearrow +\infty$, subject to Dirichlet boundary conditions both on $\partial B_1(0)$ and $\partial B_{s_n}(0)$. Truncations of a specific problem of this type were originally considered in \cite[Sec.~3.1]{Brown-2004-24} and it was shown in \cite[Sec.~6]{Boegli-2017-42} that such domain truncation is spectrally exact, see also Theorem~\ref{thm:dom.tr}. Our aim here is to investigate the diverging eigenvalues.

We transform $T_n$ in spherical coordinates with $r\in (1,s_n)$, $\Theta \in \cS^{d-1}$, employ the usual unitary transform in the radial part (see \eg~\cite[Chap.~18]{Weidmann-2003})
\begin{equation}
L^2(\R_+;r^{d-1}\dd r) \to L^2(\R_+,\dd r) \, :\, h(r) \mapsto r^{(d-1)/2} h(r), 	
\end{equation}
and use the spherical harmonics $\{Y_{l,j}\}_{j=1}^{N(l,d)}$, $l\in\N_0$, $N(l,d)=\frac{(2l+d-2)(l+d-3)!}{l!(d-2)!}$ in $d-1$ dimensions, which satisfy
$
-\Delta_{\cS^{d-1}}Y_{l,j}(\Theta)=l(l+d-2)Y_{l,j}(\Theta).	
$
%
%denotes the number of linearly independent homogeneous harmonic polynomials of degree l in d variables
Thereby we obtain a decomposition of $T_n$ to one dimensional operators
\begin{equation}	\label{Tnl.def}
		T_{n,l}:=-\partial_r^2+\ii U(r)+U_1(r),
		\quad
		\Dom(T_{n,l}):=W^{2,2}((1,s_n))\cap W^{1,2}_0((1,s_n)),
\end{equation}
where $U(r) = V(x)$ for $|x|=r$ and
\begin{equation}
U_1(r)=\frac{(d-1)(d-3)+ 4l(l+d-2)}{4 r^2}.
\end{equation}
Similarly as in Examples~\ref{ex:Q.even}, \ref{ex:exp}, $T_{n,l}^*$ is unitarily equivalent via the reflection $r \mapsto -r$ to $-\partial_r^2 + Q$ in $L^2((-s_n,-1))$ with $Q(r) = - \ii U(-r) - U_1(-r)$. 

We suppose that $U \chi_{[1,+\infty]}$ satisfies Assumption \ref{asm:U} (with perturbations as in Example~\ref{ex:Q.odd}) and note that $U_1$ satisfies conditions of Proposition~\ref{prop:pert}.
Then Theorem~\ref{thm:1D.im}, Proposition~\ref{prop:pert} and Theorem~\ref{thm:div} yield that the spectra of $T_{n,l}$ in \eqref{Tnl.def} contain asymptotically the eigenvalues
\begin{equation}
\lambda_{k,n,l}=U'(s_n)^\frac23 \left(\overline{\nu_k}+\rho_{k,n,l}\right)+\ii U(s_n) - U_1(s_n),\quad n\to\infty.
\label{la.rad}
\end{equation}

In particular for $V(x)=\ii |x|^2$ with $x \in \R^d \setminus \ov{B_1(0)}$ we obtain from \eqref{la.rad} that the spectral of the one dimensional operators $T_{n,l}$, see~\eqref{Tnl.def}, contain asymptotically the eigenvalues 
	\begin{equation}\label{la.knl.rad}
		\la_{k,n,l}=(2s_n)^\frac23 \left(\overline{\nu_k}+ \BigO_{k,l}\left(s_n^{-\frac{4}{3}}\right) \right)+\ii s_n^2,\quad n\to\infty;
	\end{equation}
Figures~\ref{fig:ix2.rad} and \ref{fig:ix2.intro} illustrate this result. 

	\begin{figure}[htb!]
	\includegraphics[width=0.47\textwidth]{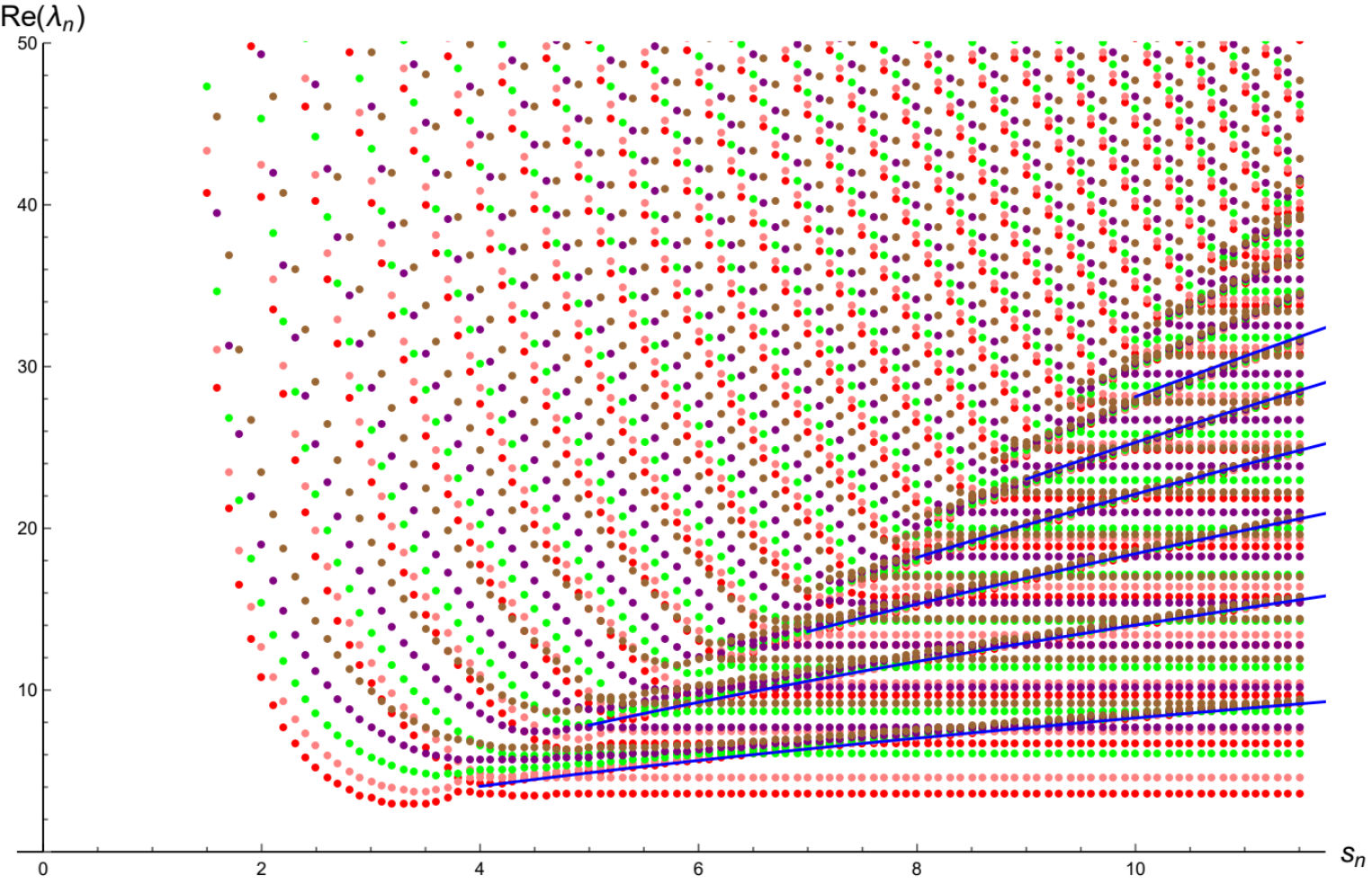}
	\hfill
	\includegraphics[width=0.47\textwidth]{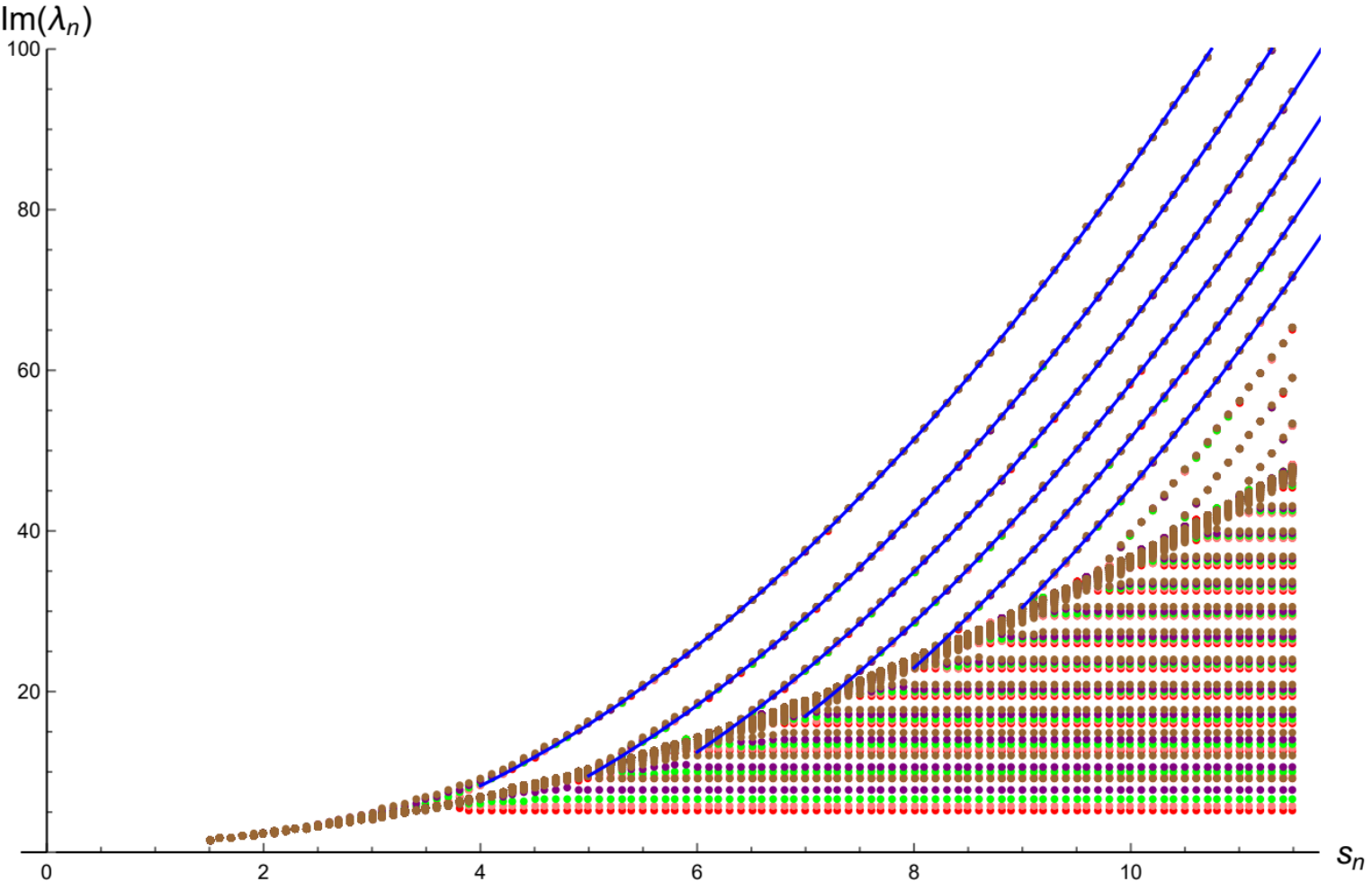}
	\caption{$V(x)=\ii |x|^2$: Real (left) and imaginary (right) part of the eigenvalues of truncated operators $T_{n,l}$ with $d=3$ and $l=1,\dots,5$ (red, pink, green, purple, brown), defined on $L^2((1,s_n))$ with $s_n=0.1n$, $n=15,16,\dots,115$. Asymptotic curves (blue) for $\lambda_{k,n}$ with the first correction for $k=1,2,\dots,6$.}
	\label{fig:ix2.rad}
\end{figure}
	%
%Finally, notice that the transformations of $T_{n,l}$ leading to $S_{n,l}$ are independent of $l$, therefore instead of the convergence of each $S_{n,l}$, we could investigate the whole transformed operator 
%%
%\begin{equation}
%	\begin{aligned}
%		S_n = \diag\{S_{n,l}\}_l 
%		&=  -\partial_r^2
%		-  \frac{\sigma_n^2}{(s_n- \sigma_n r)^2} \Delta_{\cS^{d-1}}
%		\\ & \quad +  \ii \sigma_n^2 \Big( U(s_n)-U(s_n- \sigma_n r) \Big)  + \sigma_n^2 \frac{(d-1)(d-3)}{4(s_n- \sigma_n r)^2}
%	\end{aligned}
%\end{equation}
%%
%However, for every $l \in \N_0$ the operator $S_{n,l}$ converges to the same limit $\Sa$ with $\omega = \pi/2$ and $\Gamma=\R_+$, see Example~\ref{ex:Airy.cone}, \ie~the~Dirichlet imaginary Airy operator in $L^2(\R_+)$. Hence $S_n \to \diag\{ \Sa \}_l$ in the generalized strong resolvent sense, however, the convergence in the norm resolvent sense cannot hold since the limit $\diag\{ \Sa \}_l$, \ie~infinite number of copies of $\Sa$, does not have compact resolvent. 
%\hfill $\blacksquare$	
%	

\subsection{Two dimensional rotated squares and polynomial potential}

Finally we show that Theorem~\ref{thm:div} can be applied directly in more dimensional problems. The verification of the assumptions is analogous to the steps in proof of Theorem~\ref{thm:1D.im} in the one dimensional case.

\begin{example}\label{eq:Q.2d}
We consider the potential 
\begin{equation}\label{Q.ex.2D}
Q(x_1,x_2)=\ii(x_1^3+x_2^4)+x_1^2x_2^2, \quad x:=(x_1,x_2) \in \R^2
\end{equation}
and a sequence of domains $\{\Omega_n\} \subset \R^2$, which are expanding squares rotated by $\pi/4$ with the left-most corner at $(-s_n,0)$, \ie~with some $\{s_n\}$ with $s_n \nearrow +\infty$,
\begin{equation}\label{Omega.n.def}
\begin{aligned}
\Omega_n = \{  x \in \R^2 \, : \, & (x_1 \in (-s_n,0] \text{ and } |x_2|<x_1+s_n) \text{ or }
\\
& 
(x_1 \in (0,s_n) \text{ and } |x_2|<(s_n-x_1))\}. 
\end{aligned}
\end{equation}

Using Theorems~\ref{thm:dom.tr} and \ref{thm:div} we explain below that the Dirichlet truncations of $T = -\Delta + Q$ in $L^2(\R^2)$ to $T_n = -\Delta + Q$ in $L^2(\Omega_n)$, $n \in \N$, are spectrally exact and the spectra of $T_n$ contain asymptotically the eigenvalues 
\begin{equation}\label{Q.ex.la}
	\la_{k,n}^{(j)}= (3s_n^2)^\frac23 \Big(\nu_k + \rho_{k,n}^{(j)} \Big) - \ii s_n^3, \quad n \to \infty, 
\end{equation}
where $\{\nu_k\}$ are eigenvalues of the complex Airy operator $\Sa$ with $\Gamma = \Gamma_{\pi/4}$, $\omega = \pi/2$ and $\vartheta = (1,0)$, see Example~\ref{ex:Airy.cone}, and 
\begin{equation}
\frac{1}{m_a(\nu_k)} \left|\sum_{j=1}^{m_a(\nu_k)} \rho_{k,n}^{(j)}  \right| = \BigO(s_n^{-\frac23}), \quad n \to \infty.
\end{equation}

Clearly $Q \in C^1(\R^d)$ satisfies Assumption~\ref{asm:Q} on each $\Omega_n$, $n\in\N$, and 
\begin{equation}\label{Q.ex.def}
\nabla Q(x_1,x_2) = (2x_1x_2^2+ 3 \ii x_1^2, 2 x_1^2 x_2 + 4 \ii x_2^3), \ \ |\nabla Q(-p_n)| = 3 s_n^2, \  p_n=(s_n,0).
\end{equation}

We first check that $Q$ satisfies Assumption~\ref{asm:Q} and \eqref{Q.unbd} on $\R^2$, so the truncations on $\{\Omega_n\}$ are spectrally exact by Theorem~\ref{thm:dom.tr}. Indeed, \eqref{Q.nab.as} holds since
\begin{equation}
\begin{aligned}
|x_2| \leq \tfrac 12 |x_1|^{\frac 43}, |x_1| \gs 1: &\quad |Q(x)| \gs |x_1|^3,&  |\nabla Q(x)| &\ls x_1^4,
\\
|x_2| \geq \tfrac 32 |x_1|^{\frac 43}, |x_2| \gs 1: & \quad |Q(x)| \gs x_2^4,&  |\nabla Q(x)| &\ls |x_2|^3,
\\
|x_2| \in (\tfrac 12 |x_1|^{\frac 43},  \tfrac 32 |x_1|^{\frac 43}), |x_1| \gs 1:& \quad |Q(x)| \gs |x_1|^{\frac{14}{3}},&  |\nabla Q(x)| &\ls |x_1|^4.
\end{aligned}
\end{equation}

To apply Theorem~\ref{thm:div} we check conditions in Assumption~\ref{asm:div}. The conditions~\ref{asm:div.sn} and \ref{asm:div.Sign} are satisfied with $\Gamma = \Gamma_{\pi/4}$ and $r_n \approx s_n^{5/3}$. Moreover, we have $\omega = \pi/2$ and $\vartheta = (0,1)$, see \eqref{omega.lim}.
To verify the condition \ref{asm:div.Qn.sep}, we employ formulas \eqref{Qn.form1}, \eqref{Qn.form2} and proceed similarly as in the one dimensional case (see the proof of Theorem~\ref{thm:1D.im}). 

Using the variable $\sigma_n x = y$, if $(y_1,y_2) \in p_n + \Omega_n$ with $0 < y_1 <1$, then (since $|y_2|\leq y_1$)
\begin{equation}
|(\nabla Q_n)(x)| = \frac{|\nabla Q(y-p_n)|}{|\nabla Q(-p_n)|}  \ls \frac{s_n^2}{s_n^2} = 1.	
\end{equation}
Next, for $(y_1,y_2) \in p_n + \Omega_n$ with $y_1 \in [1,s_n]$, by the elementary identity for  $s_n^3-(s_n-x)^3$, we obtain
\begin{equation}
|Q(y_1-s_n,y_2)-Q(-s_n,0)| \geq  y_1 (s_n^2 + s_n (s_n-y_1) + (s_n-y_1)^2) \geq y_1 s_n^2. 
\end{equation}
Thus, (recall $|y_2| \leq y_1$)
\begin{equation*}
\frac{|(\nabla Q_n)(x)|}{|Q_n(x)|^\frac32} 
= 
\frac{|(\nabla Q)(y_1-s_n,y_2)|}{|Q(y_1-s_n,y_2)-Q(-s_n,0)|^\frac32}
\ls \frac{s_n y_1^2 + s_n^2 + s_n^2 y_1 + y_1^3}{y_1^\frac 32 s_n^3} \ls \frac{1}{s_n}. 
\end{equation*}
Finally, for $(y_1,y_2) \in p_n + \Omega_n$ with $y_1 \in (s_n,2 s_n)$, we have
\begin{equation}
|Q(y_1-s_n,y_2)-Q(-s_n,0)| \gs s_n^3 	
\end{equation}
and so (again since $|y_2| \leq y_1$)
\begin{equation}
\frac{|(\nabla Q_n)(x)|}{|Q_n(x)|^\frac32} \ls \frac{s_n^3}{s_n^\frac 92} = \frac{1}{s_n^\frac 32}. 
\end{equation}
Finally, using \eqref{Qn.form2}
\begin{equation}
\iota_n \ls \frac{|(s_n-y_1)^2y_2^2 + \ii (y_2^4-3 s_n y_1^2 +y_1^3)|}{s_n^\frac23(s_n^3-(s_n-y_1)^3) y_1}, \quad (y_1,y_2) \in p_n + \Omega_n.
\end{equation}
For $0 < y_1 <s_n$, by the identity for  $s_n^3-(s_n-y_1)^3$ and $|y_2|/y_1 \leq 1$, we obtain
\begin{equation}
\frac{|(s_n-y_1)^2y_2^2 + \ii (y_2^4-3 s_n y_1^2 +y_1^3)|}{s_n^\frac23(s_n^3-(s_n-y_1)^3) y_1}
\ls \frac{s_n^2}{s_n^{\frac23 + 2 }} = \frac{1}{s_n^\frac23}
\end{equation}
and for $s_n \leq  y_1 <2s_n$, we arrive at
\begin{equation}
	\frac{|(s_n-y_1)^2y_2^2 + \ii (y_2^4-3 s_n y_1^2 +y_1^3)|}{s_n^\frac23(s_n^3-(s_n-y_1)^3) y_1}
	\ls \frac{s_n^4}{s_n^{\frac23 + 4 }} = \frac{1}{s_n^\frac23}.
\end{equation}
Thus $\iota_n = \BigO(s_n^{-\frac 23})$, $n \to \infty$.
\hfill $\blacksquare$
\end{example}

\section{Remarks on strong coupling}
\label{sec:strong.c}

We consider a family of Dirichlet realizations of 
\begin{equation}\label{Tg.def}
	T_g = -\Delta + Q_1 + \ii g Q_2, \qquad g>0,
\end{equation}
in $L^2(\Omega)$ where $\Omega$ is open, functions $Q_i$, $i=1,2$, are real valued and $g \to +\infty$. 

Operators with this structure arise in several contexts, in particular, in enhanced dissipation, 
see Example~\ref{ex:enh.dis}, in $\mathcal{PT}$-symmetric phase transitions, see Examples~\ref{ex:BM} and \ref{ex:CG}, or when $Q_1=0$ as semi-classical problems with purely imaginary potentials, see \eg~\cite{Henry-2014,Almog-2016-48,Almog-2018-59}, in particular in the context of Bloch-Torrey equation.
%, and \cite{Boegli-toappear} with an example showing sharpness of Lieb-Thirring bounds, see~Example~\ref{ex:BS}. 

We focus here on the case when $\Omega$ is (typically) unbounded and $|Q_2|$ has a global minimum inside of $\Omega$, see Assumption~\ref{asm:g} for details. As an application of Corollary~\ref{cor:n}, we describe some of the diverging eigenvalues as $g \to +\infty$. In Section~\ref{ssec:ex.g} we show how Theorem~\ref{thm:g} can be implemented and indicate its possible further extensions.

\begin{asm-sec}\label{asm:g}
Let $\Omega \subset \Rd$ be open with $0 \in \Omega$, let $\ov{B_R(0)} \subset \Omega$ for some $R>0$ and let $Q_1 \in C^1(\ov \Omega; \R)$ with $Q_1 \geq 0$, $Q_2 \in C^1(\ov \Omega \setminus \{0\}; \R)$. Suppose further that
\begin{enumerate}[\upshape i), wide]
\item \label{asm:g.sep} for some $\eps >0$, the condition \eqref{Q.asm.sep} is satisfied with $\Omega$ replaced by $\Omega \setminus \ov{B_\eps(0)}$ and $Q$ replaced separately by $Q_1$ and by $Q_2$; 	
\item $Q_2(0)=0$ and $|Q_2|$ attains the global minimum at $0$, \ie~for every $\delta>0$
\begin{equation}\label{Q2.inf}
	\inf_{x \in \Omega \setminus \ov{B_\delta(0)}} |Q_2(x)| >0;
\end{equation}
\item there exists $Q_\infty \in C(\Rd) \cap  C^1(\Rd \setminus \{0\};\R)$ with $\min_{|x|=1} |Q_{\infty}(x)|>0$ such that for some $\kappa >0$
\begin{equation}\label{Q.inf.g}
\begin{aligned}
Q_{\infty}(tx) &= t^\kappa Q_{\infty}(x),  && x \in \Rd, \ t>0, 
\\
Q_2(x) - Q_\infty(x) &= |x|^\kappa h_0(x), && h_0(x) = o(1), \quad |x| \to 0,
\\
|\nabla Q_2(x) - \nabla Q_\infty(x)| &= |x|^{\kappa-1} h_1(x), && h_1(x) = o(1), \quad |x| \to 0,
\end{aligned}
\end{equation}
and the discrete spectrum of $S_{\infty}:= -\Delta + \ii Q_{\infty}$ in $L^2(\Rd)$ is non-empty.
\hfill $\blacksquare$
\end{enumerate}
\end{asm-sec}

\begin{example}\label{ex:H_kappa}
Typical examples of $S_\infty$ in Assumption~\ref{asm:g} in one dimension are 
\begin{equation}\label{H.kappa.def}
\Dti + \ii x^n, \quad n \in \N \setminus\{1\}, \qquad	\Dti +\ii|x|^\kappa, \quad \kappa >0.
\end{equation}
The spectra of the former for $n={2k+1}$, $k \in \N$, are real, see \cite{Shin-2002-229}, and the spectra of the remaining operators with even potential can be obtained by complex scaling (after possibly reducing the problem to Dirichlet/Neumann operators in $L^2(\R_+)$). A typical case in more dimensions is an imaginary oscillator with potential $\ii \langle A x,x \rangle_{\Rd}$ and a positive definite matrix $A$.
\hfill $\blacksquare$
\end{example}

\begin{theorem}\label{thm:g}
Let Assumption~\ref{asm:g} be satisfied and let $T_g$, $g>0$, be as in \eqref{Tg.def}. Then the spectra of $T_g$ contain asymptotically the eigenvalues (with $k \in \N$ and $j \in \{1,\dots,m_a(\nu_k)\}$)

\begin{equation}\label{la.g}
	\la_{k,g}^{(j)}=g^\frac{2}{2+\kappa} \left(\nu_k+\rho_{k,g}^{(j)}\right),\quad g\to+\infty,
\end{equation}
where $\{\nu_k\} = \spd(S_\infty)$ and, as $g \to +\infty$,  $\rho_{k,g}^{(j)} = o_{j,k}(1)$ and for any $\beta \in (0,1)$,
\begin{equation}\label{rkn.est.g}
	\frac1{m_a(\nu_k)} \left| \sum_{j=1}^{m_a(\nu_k)} \rho_{k,g}^{(j)}\right| = \BigO_{k} \left(
	g^{-\frac{\min\{2,\kappa(1-\beta)\} }{2+\kappa}} + \sup_{|y|\leq g^{- \frac{\beta}{2+\kappa}}} |h_0(y)|
	\right).
\end{equation}

\end{theorem}
\begin{proof}
We select $\sigma = \sigma(g)$ to satisfy
\begin{equation}\label{sig.g.def}
	g \sigma^{2+\kappa} =1, \quad g>1
\end{equation}
so $\sigma \to 0$ as $g \to +\infty$. By scaling $x \mapsto \sigma x$, we obtain operators in $L^2(\sigma^{-1} \Omega)$
\begin{equation}
	\frac{1}{\sigma^2} \left( - \Delta + \sigma^2\Big(Q_1(\sigma x) + \ii g Q_2(\sigma x) \Big) \right ) =: \frac{1}{\sigma^2} S_g,
\end{equation}
which are unitarily equivalent to $T_g$. In the following we apply Corollary~\ref{cor:n} to the operators $S_g$ and $S_\infty = -\Delta + \ii Q_{\infty}$ in $L^2(\Rd)$.

We define 
\begin{equation}
	Q_\sigma(x) := \sigma^2\Big(Q_1(\sigma x) + \ii g Q_2(\sigma x) \Big), \quad  x \in \sigma^{-1} \Omega
\end{equation}
and verify conditions in Assumption~\ref{asm:n}.

\ref{asm:n.dom} From the scaling we have $ B_{\sigma^{-1}R} (0) \subset \sigma^{-1} \Omega $ and so $\sigma^{-1} \Omega$ exhaust $\Omega_\infty :=\Rd$. 

\ref{asm:n.Qn} We first split $Q_\sigma$ as $Q_\sigma = \eta Q_\sigma + (1-\eta) Q_\sigma$ where $\eta \in C_0^\infty(B_{2\alpha}(0))$ with $\eta = 1$ on $B_\alpha(0)$ and where sufficiently large $\alpha>0$, independent of $g$, will be fixed later. We show that $\eta Q_\sigma$ is uniformly bounded (and so can be treated as a perturbation, see remarks after Assumption~\ref{asm:Q}) and $(1-\eta) Q_\sigma$ satisfies \eqref{Qn.sep}. 

For $|x| \leq 2 \alpha$, using \eqref{Q.inf.g} and \eqref{sig.g.def} we have
\begin{equation}
	|Q_\sigma(x)| \ls \sigma^2 + \sigma^{2+\kappa} g \left( |Q_\infty(x)| +  |x|^\kappa h_0(\sigma x)\right) \ls 1.
\end{equation}
Since $\nabla (1-\eta) Q_\sigma = -(\nabla \eta) Q_\sigma + (1-\eta) \nabla Q_\sigma$
and $\supp \eta \subset B_{2\alpha}(0)$, it suffices to further analyze $|\nabla Q_\sigma(x)|$ for $|x|>\alpha$. Namely, we estimate
\begin{equation}
	\frac{|\nabla Q_\sigma(x)|}{|Q_\sigma(x)|^\frac32} =  \frac{|(\nabla Q_1)(\sigma x) + \ii g (\nabla Q_2)(\sigma x)|}{|Q_1(\sigma x) + \ii g Q_2(\sigma x)|^\frac32}.
\end{equation}

At first we focus on the region $\alpha < |x| \leq \delta \sigma^{-1}$ with a sufficiently small $\delta >0$ which will be selected later. Using \eqref{Q.inf.g} and homogeneity of $Q_\infty$, we obtain
\begin{equation}
|(\nabla Q_2)(\sigma x)| \leq |(\nabla Q_\infty)(\sigma x)| 
\left(
1 + \frac{|h_1(\sigma x)|}{\displaystyle \min_{|z|=1} |\nabla Q_\infty(z)|}
\right)
\ls |(\nabla Q_\infty)(\sigma x)|;
\end{equation}
in the last step we use Euler's homogeneous function theorem, the homegeneity of $Q_\infty$ and that the assumption $\min_{|z|=1} |Q_\infty(z)| >0$ implies that $\min_{|z|=1} |\nabla Q_\infty(z)| >0$ as well. Similarly, 
\begin{equation}\label{Q2.le}
	|Q_2(\sigma x)| \geq |Q_\infty(\sigma x)| 
	\left(
	1 - \frac{|h_0(\sigma x)|}{\displaystyle \min_{|z|=1} |Q_\infty(z)|}
	\right),
\end{equation}
thus, if $\delta>0$ is sufficiently small
\begin{equation}
	|Q_2(\sigma x)| \gs |Q_\infty(\sigma x)|, \quad \alpha < |x| < \delta \sigma^{-1}. 
\end{equation}
Hence, writing $x = t z$ with $|z|=1$, $t>\alpha$ and using \eqref{sig.g.def}, we arrive at
\begin{equation}
\frac{|\nabla Q_\sigma(x)|}{|Q_\sigma(x)|^\frac32} \ls \frac{1 + g (\sigma t)^{\kappa-1} |\nabla Q_\infty(z)|}{(g\sigma^{\kappa} t^{\kappa})^\frac32 |Q_\infty(z)|^\frac32}
\ls  
\frac{\sigma^3}{\alpha^\frac{3\kappa}{2}} + \frac{1}{\alpha^{\frac \kappa2 +1}}, \qquad \alpha < |x| < \delta \sigma^{-1}
\end{equation}
and so we can select $\alpha>0$ so that the right hand side is sufficiently small for all $x$ in the considered region.

Next let $\delta \sigma^{-1} \leq |x| \leq \eps \sigma^{-1}$ with $\eps>0$ from the assumption. Then, from \eqref{Q2.inf}, 
\begin{equation}
\frac{|\nabla Q_\sigma(x)|}{|Q_\sigma(x)|^\frac32} 
\ls \frac{1 + g}{g^\frac 32 |Q_2(\sigma x)|^\frac32}
\ls  \frac{1}{g^\frac 12}.
\end{equation}
%
%and so for all sufficiently large $g$, $\frac{|\nabla Q_\sigma(x)|}{|Q_\sigma(x)|^\frac32}$ is sufficiently small for all $x$ in the considered region.

Finally, let $\eps \sigma^{-1} \leq |x|$. Here we use \eqref{Q2.inf} and the $Q_1$ and $Q_2$ satisfy separately  \eqref{Q.asm.sep} outside $B_\eps(0)$. Thus writing $y = \sigma x$, we get
\begin{equation*}
%\begin{aligned}
\frac{|(\nabla Q_\sigma)( x)|}{|Q_\sigma( x)|^\frac32} 
 \leq \frac{|\nabla Q_1(y)|_+ g |\nabla Q_2(y)|}{|Q_1(y)+ \ii g Q_2(y)|^\frac32} 
 \leq \eps_\nabla  + \frac{M_\nabla}{g^\frac32 |Q_2(y)|^\frac32}+\frac{\eps_\nabla}{g^\frac12} + \frac{M_\nabla}{g^\frac12 |Q_2(y)|^\frac32}.
%\end{aligned}
\end{equation*}
%
%Hence for all sufficiently large $g$, $\frac{|\nabla Q_\sigma(x)|}{|Q_\sigma(x)|^\frac32}$ is sufficiently small also in this region.

\ref{asm:n.xin} We consider projection $P_\sigma:= \chi_{B_{\sigma^{-1}R}(0)}$ and cut-offs $\xi_\sigma \in C_0^\infty(B_{\sigma^{-1}R}(0))$ so that $\xi_\sigma(x) = 1$ for $|x| \leq \sigma^{-1}R -1$ and $\||\nabla \xi_\sigma|\|_{L^\infty}$, $\|\Delta \xi_\sigma\|_{L^\infty}$ are uniformly bounded. The conditions on the operator and form domains in \ref{asm:n.xin} can be verified easily since the support of $\xi_\sigma$ is bounded (see \eg~the proof of Theorem~\ref{thm:dom.tr}).

\ref{asm:n.Q.conv} We split the estimate of
\begin{equation}
\left\|	\frac{\xi_\sigma(Q_\sigma- \ii Q_\infty)}{(Q_\infty+1)(Q_\sigma+1)} \right\|_{L^\infty(\sigma^{-1}\Omega)}
\end{equation}
to three regions. First let $\sigma |x| \leq \sigma^{\beta}$ with $\beta \in (0,1)$. Then, using homogeneity of $Q_\infty$, $\min_{|z|=1} |Q_\infty(z)| >0$, \eqref{sig.g.def} and \eqref{Q.inf.g},
\begin{equation*}
\left|\frac{Q_\sigma(x)-\ii Q_\infty(x)}{(Q_\infty(x)+1)(Q_\sigma(x)+1)} \right|
\ls \sigma^2 + \frac{|\sigma^{2+\kappa} g Q_2(\sigma x) - Q_{\infty}(\sigma x) |}{|Q_{\infty}(\sigma x)|}
\ls \sigma^2 + |h_0(\sigma x)|.
\end{equation*}
Next, when $\sigma^\beta \leq |\sigma x| \leq \delta$ with $\delta >0$ fixed, but sufficiently small, we use inequality \eqref{Q2.le} and the properties of $Q_\infty$ similarly as above to arrive at
\begin{equation}
\left|\frac{Q_\sigma(x)- \ii Q_\infty(x)}{(Q_\infty(x)+1)(Q_\sigma(x)+1)} \right|
\ls \frac{1}{|Q_\infty(x)|} + \frac{1}{|Q_\sigma(x)|} \ls \sigma^{\kappa(1-\beta)}.
\end{equation}
Finally, for $\delta \leq |\sigma x| \leq R$, we use in addition \eqref{Q2.inf} and obtain
\begin{equation}
\left|\frac{Q_\sigma(x)- \ii Q_\infty(x)}{(Q_\infty(x)+1)(Q_\sigma(x)+1)} \right|
\ls \frac{1}{|Q_\infty(x)|} + \frac{1}{|Q_\sigma(x)|} 
\ls \sigma^\kappa + \frac{\sigma^\kappa}{|Q_2(\sigma x)|} 
\ls \sigma^\kappa.
\end{equation}

The estimate of the remaing terms in \eqref{tau.n.def} is similar. Namely, denoting $\zeta_\sigma$ the characteristic function of $\supp(1-\xi_\sigma)$, we obtain
\begin{equation}
\left\| 
\frac{\zeta_\sigma}{Q_\sigma}
\right\|_{L^\infty(\sigma^{-1} \Omega)}
+
\left\| 
\frac{\zeta_\sigma}{Q_\infty}
\right\|_{L^\infty(\Rd)}
\ls \sigma^\kappa.
\end{equation}

Since all assumptions of Corollary~\ref{cor:n} are satisfied, we obtain that $S_\sigma \to S_\infty$ as $g \to + \infty$ and the conclusion for eigenvalues \eqref{la.g} follows by spectral mapping. Since the eigenfunctions of $S_\infty$ decay exponentially, see Example~\ref{ex:EF.dec}, the second term in \eqref{kap.n.def}, entering the estimate of $\rho_{k,g}^{(j)}$, can be omitted.
\end{proof}

\subsection{Examples}
\label{ssec:ex.g}

\begin{example}[Enhanced dissipation]
	\label{ex:enh.dis}
For operators $T_g$, sufficient conditions for the divergence of the real parts of all eigenvalues of $T_g$ as $g \to +\infty$ were found \cf~\cite{Constantin-2008-168,Gallagher-2009-2009,Schenker-2011-18}. In \cite{Schenker-2011-18}, the specific operator
	\begin{equation}
		T_g=\Dti+x^2+\ii g (1+|x|^\kappa)^{-1},
	\end{equation}
in $L^2(\R)$ and with $\kappa>0$ was analyzed and an estimate on the divergence rate of the real part of eigenvalues $\Re \sigma (T_g)\gs g^{\frac{2}{\kappa+2}}$ was proved, \cf~\cite[Thm.~1.2]{Schenker-2011-18}. Similar problem and result was also established in \cite[Thm.~1.9]{Gallagher-2009-2009}.
\begin{figure}[htb!]
	\includegraphics[width=0.47\textwidth]{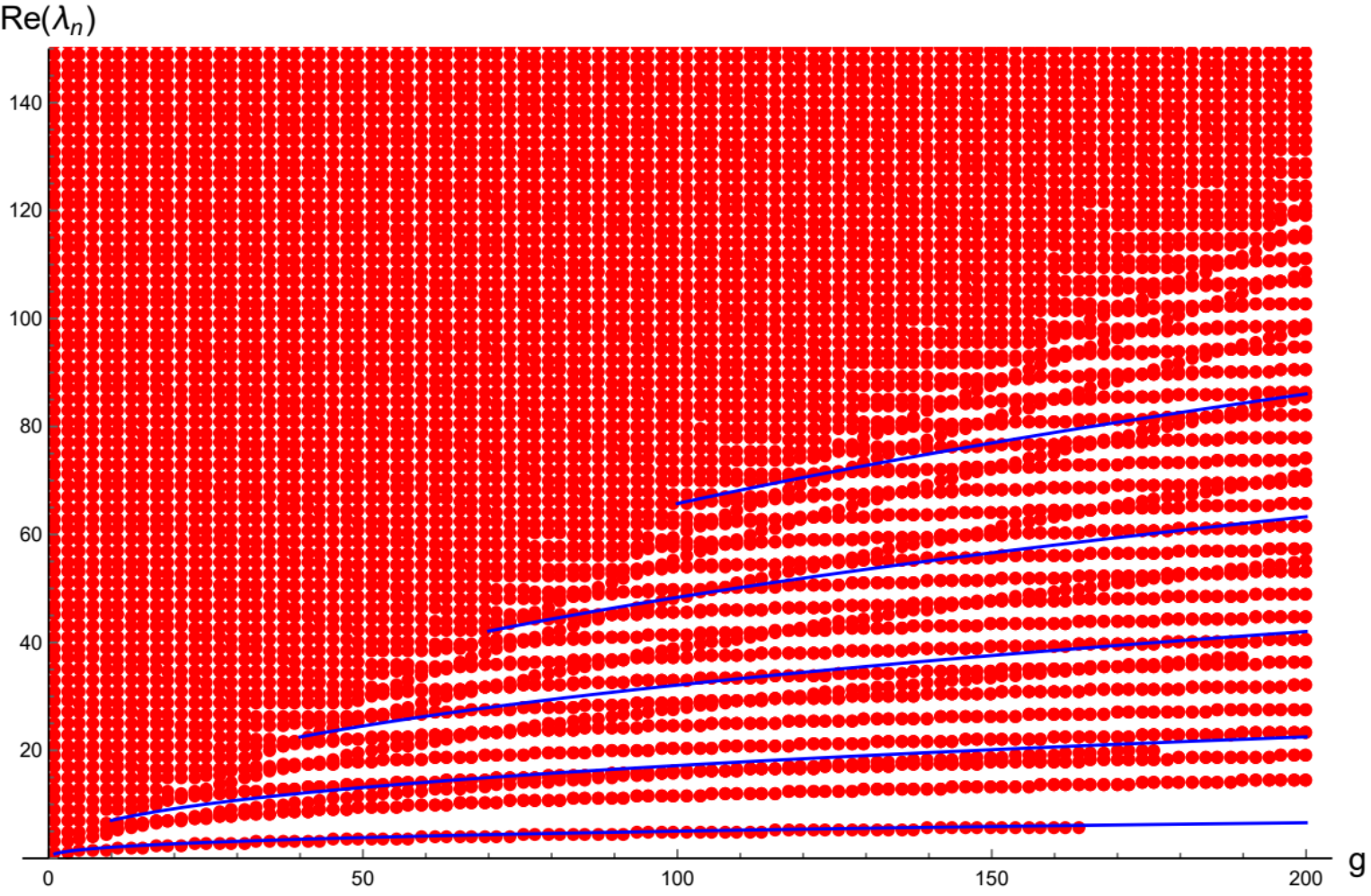}
	\hfill
	\includegraphics[width=0.47\textwidth]{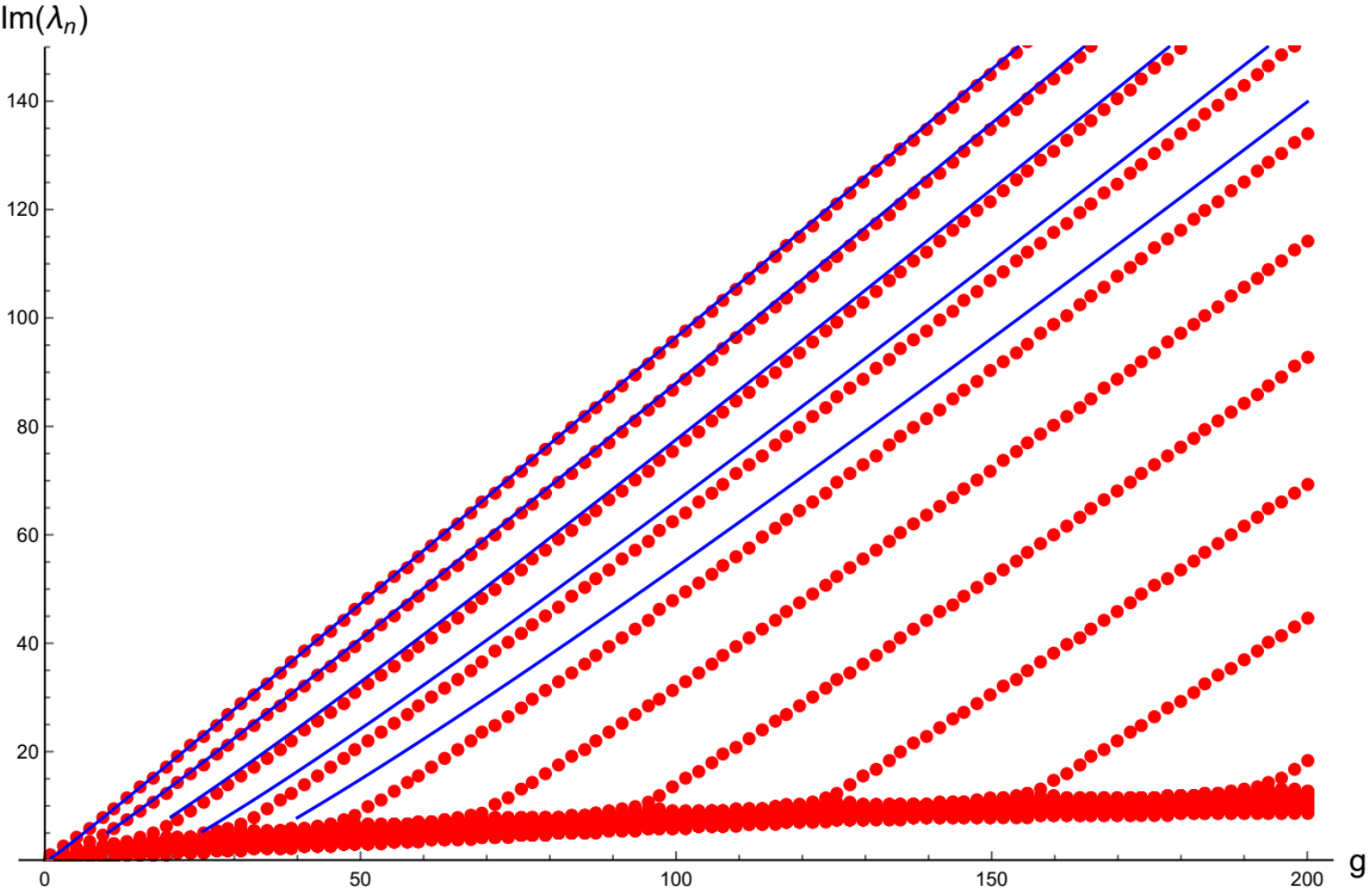}
	\caption{$Q_1(x)=x^2$, $Q_2(x)=(1+|x|^\kappa)^{-1}$: Real (left) and imaginary (right) part of the eigenvalues (red) of operators $T_g$ with $\kappa=3.15$ and $g=5,10,\dots,200$. Asymptotic curves (blue) for $\lambda_{k,g}$ for $k=1,2,\dots,5$.}
	\label{fig:Schenker}
\end{figure}

Note that the conjugated and shifted operator $T_g^*+\ii g$ satisfies Assumption \ref{asm:g} with $Q_2(x)=|x|^\kappa/(1+|x|^\kappa)$, $Q_\infty(x)=|x|^\kappa$ and $h_0(x)=-Q_2(x)$. Therefore by Theorem \ref{thm:g}, spectra of $T_g$ contain asymptotically the eigenvalues
\begin{equation}
	\lambda_{k,g}=g^{\frac{2}{\kappa+2}}(\overline{\nu_k+\rho_{k,g}})+\ii g, \quad g \to + \infty,
\end{equation}
where $\{\nu_k\}$ are the eigenvalues of operator in \eqref{H.kappa.def} with the potential $\ii |x|^\kappa$.
The remainder decays as $\rho_{k,g}=\mathcal{O}(g^{\frac{-\kappa}{2(2+\kappa)}})$ for $\kappa\in(0,4)$ and $\rho_{k,g}=\mathcal{O}(g^{\frac{-2}{2+\kappa}})$ for $\kappa\geq 4$.
This result shows that the estimate in \cite[Thm.~1.2]{Schenker-2011-18} is optimal (see Figure~ \ref{fig:Schenker}).
\hfill $\blacksquare$
\end{example}

\begin{example}[$\cP\cT$-symmetric phase transitions I] 
\label{ex:BM}

Let $\Omega = \R$, $Q_1$ be even, $Q_2$ odd and such that Assumption~\ref{asm:g} is satisfied. As in Example~\ref{ex:Q.odd}, the operators $T_g$ in \eqref{Tg.def} with such $Q_1$, $Q_2$ have the antilinear $\cP\cT$-symmetry and so the spectra of $T_g$ consists of complex conjugate pairs. The spectrum of $T_0$ is real due to the self-adjointness, however, as $g \to \infty$, a graduate appearance of complex conjugated (non-real) spectral points pairs, called $\cP\cT$-symmetric phase transitions, was observed in many examples, see~\eg~\cite{Znojil-2001-285} for one of the first works. 

For $Q_1(x) = x^2$, upper estimates on the number of non-real eigenvalues are given in \cite{Mityagin-2015-54} and precise spectral analysis of the double $\delta$ potential (with a fixed $b>0$)
\begin{equation}\label{Tg.BM}
 \Dti + x^2+\ii g (\delta(x-b)-\delta(x+b))
\end{equation}
is performed in \cite{Mityagin-2015-54a,Baker-2020-61}. In particular it is showed in \cite{Baker-2020-61} that the number of non-real eigenvalues of \eqref{Tg.BM} diverges as $g \to + \infty$.

We consider here 
\begin{equation}
	T_g=\Dti+x^2+\ii g x^3 \e^{-x^2},
\end{equation}
in $L^2(\R)$ which can be viewed as a "smooth version" of \eqref{Tg.BM}. In this case, we can apply Theorem~\ref{thm:g} in three stationary points of $Q_2(x) = x^3 \e^{-x^2}$, namely, $x_0=0$, $x_1=-\sqrt{3/2}$ and $x_2 = -x_1$. 

The operator $T_g$ satisfies Assumption \ref{asm:g} with $Q_1(x)=x^2$, $Q_2(x)=x^3\e^{-x^2}$, $Q_\infty(x)=x^3$, $\kappa=3$, $h_0(x)=\e^{-x^2}-1$. Therefore the eigenvalues
\begin{equation}
\lambda_{k,g}^{(x_0)}=g^{\frac25}(\nu_k+\BigO_k(g^{-\frac{6}{25}})), \quad g \to + \infty,
\end{equation}
where $\nu_k$ are (real) eigenvalues of the imaginary cubic oscillator (the potential $\ii x^3$), \cf~Example \ref{ex:H_kappa}, lie asymptotically in the spectra of $T_g$.

Further sets of eigenvalues can be obtained by applying the Theorem \ref{thm:g} to the operator $\widetilde T_g-\ii g (3/(2\e))^\frac32$, where $\widetilde T_g$ is the operator obtained from $T_g$ by the translation $x \mapsto x+x_1$. It satisfies the Assumption \ref{asm:g} with $\kappa=2$ and 
\begin{equation}
\begin{aligned}
Q_1(x)&=(x+x_1)^2, & Q_2(x)&=(x+x_1)^3\e^{-(x+x_1)^2}+(\tfrac{3}{2\e})^\frac32, 
\\ 
Q_\infty(x)&= (\tfrac{27}{2\e^3})^\frac12 x^2, & h_0(x)& =\frac{Q_2(x)}{x^2}-(\tfrac{27}{2\e^3})^\frac12. 
\end{aligned}
\end{equation}
Therefore the eigenvalues
\begin{equation}
	\lambda_{k,g}^{(x_1)}=g^{\frac12}(\nu_k+\BigO_k(g^{-\frac18}))-\ii g(2\e)^{-\frac12} + \tfrac32, \quad g \to + \infty,
\end{equation}
where $\nu_k=(\frac{27}{2\e^3})^{\frac14}\e^{\ii \frac{\pi}{4}}(2k+1)$, $k\in\N_0$, lie asymptotically in the spectra of $T_g$. Analogous steps can be implemented on the conjugate operator $T_g^*$ and we obtain the second set of eigenvalues $\lambda_{k,g}^{(x_2)}=\overline{\lambda_{k,g}^{(x_1)}}$, \cf~ Figure~\ref{fig:BM}. 
\begin{figure}[htb!]
\includegraphics[width=0.47\textwidth]{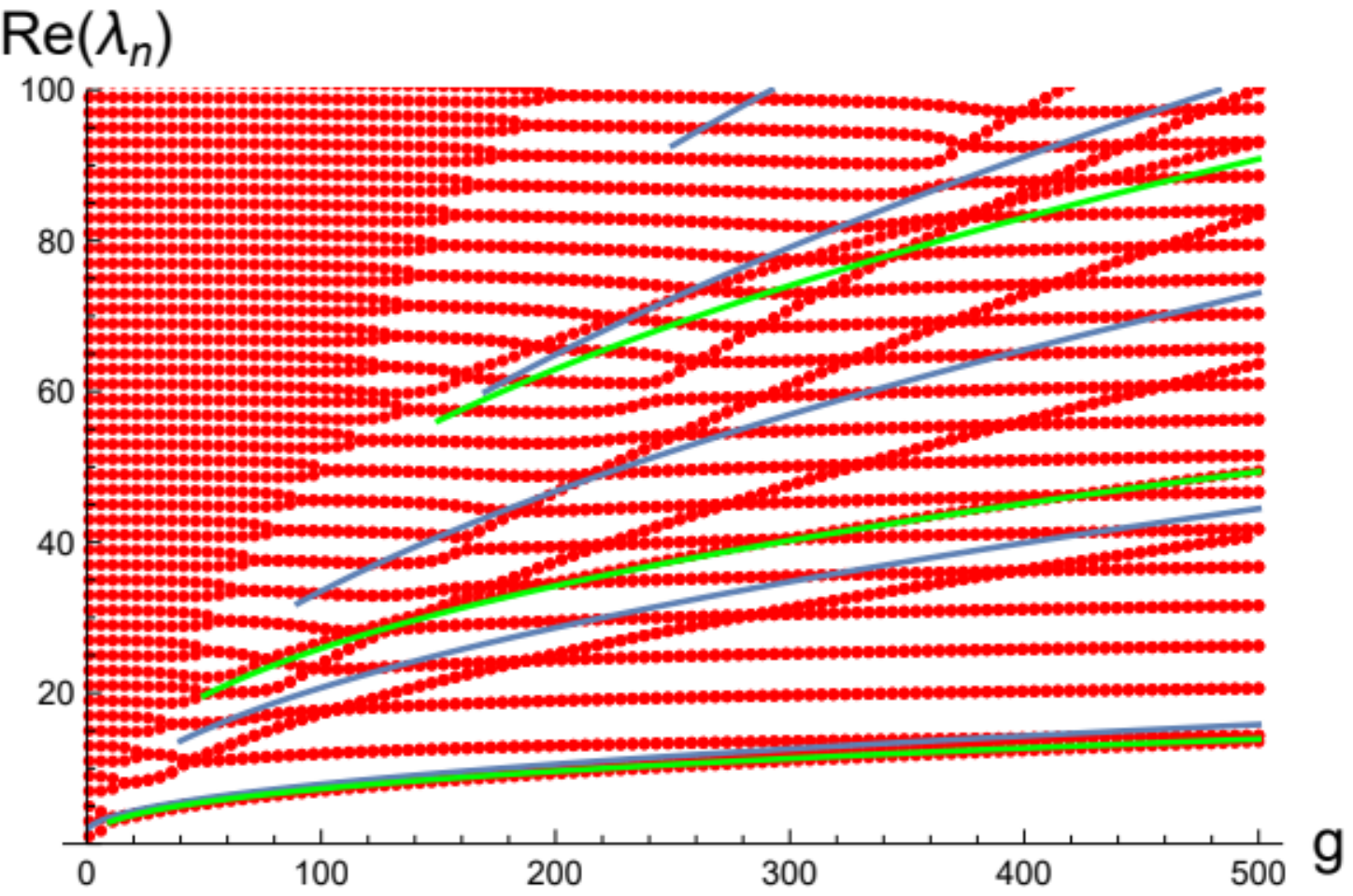}
\hfill
\includegraphics[width=0.47\textwidth]{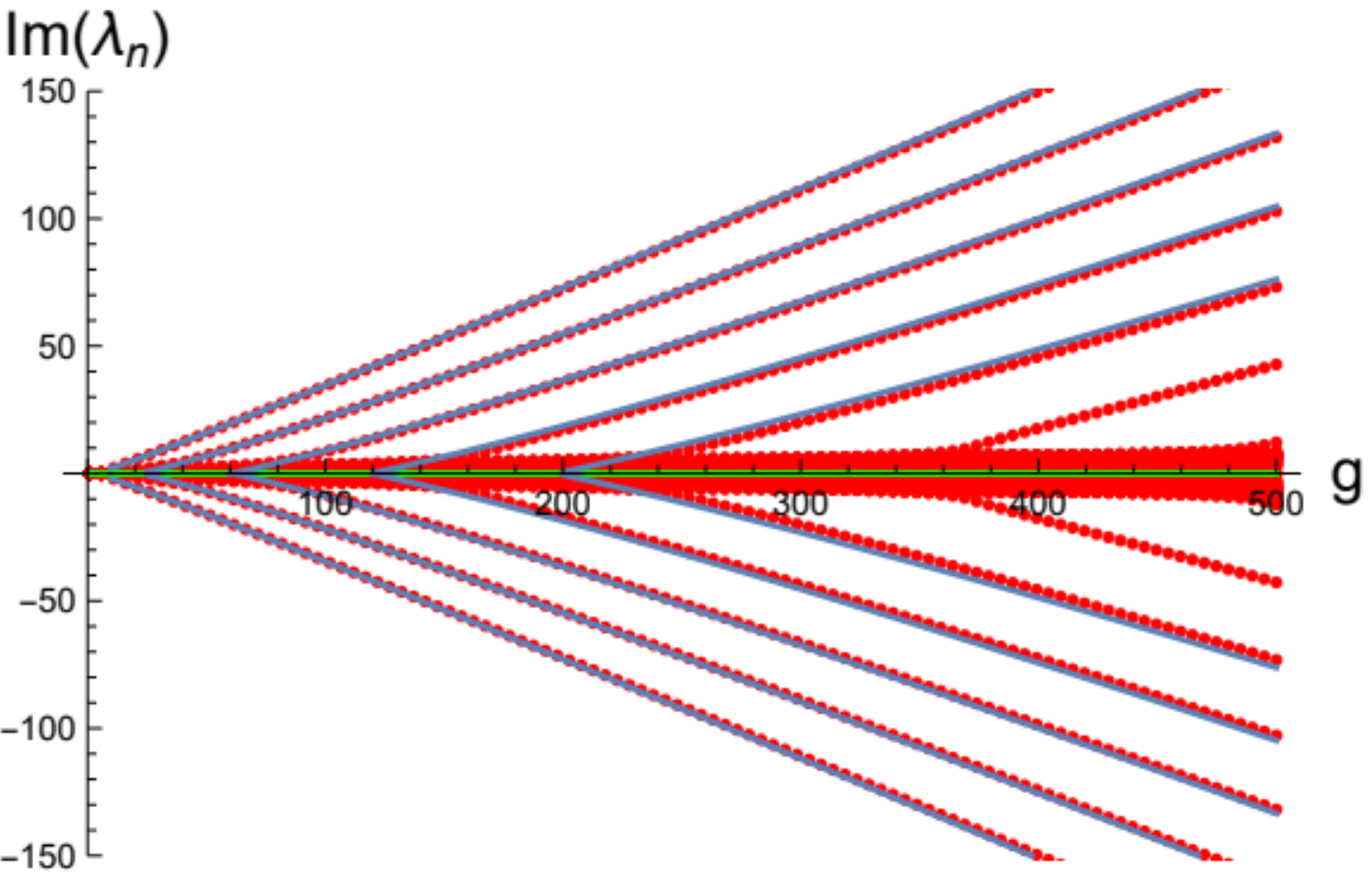}
\caption{$Q_1(x)=x^2$, $Q_2(x)=x^3\e^{-x^2}$: Real (left) and imaginary (right) part of the eigenvalues (red) of operators $T_g$ with $g=5,10,\dots,500$. Asymptotic curves (blue) $\lambda_{k,g}^{(x_1)},\lambda_{k,g}^{(x_2)}$ and (green) $\lambda_{k,g}^{(x_0)}$ for $k=1,2,\dots,5$.} 
\label{fig:BM}
\end{figure}
%

%Using Remark~\ref{rem:res.z}, we can estimate that, with some $\delta>0$, $\delta \log g$ eigenvalues of $T_g$ are in $1/2$-neighborhood of the transformed limit operator $H_2$ corresponding to the stationary point at $x_1$ (and by symmetry in $x_2$). Notice that in these cases $\kappa =2$, the optimal choice is $\beta = 2/3$ and so the convergence rate of the resolvents is $\BigO(g^{-1/6})$ as $g \to + \infty$. Using this estimate and the shift by $\ii g Q_2(x_1)$ it follows that the number of non-real eigenvalues of $T_g$ is at least $2 \delta \log g$, which is far from the upper estimate $\BigO(g^4 \log^8 g)$, $g \to + \infty$, from \cite[Thm.~4.1]{Mityagin-2015-54a}. The difference originates probably in the expected non-optimality of our estimate, but also because we did not include eigenvalues around imaginary axis with slowly diverging real parts (suggested by numerics, see Figure~\ref{fig:BM}). The latter seem to be related to the limiting eigenvalues found in \cite{Baker-2020-61} for \eqref{Tg.BM}, nevertheless, unlike for \eqref{Tg.BM}, the real parts of such eigenvalues need to diverge as a consequence of \cite[Thm.~1.1]{Schenker-2011-18}. Notice that Figure~\ref{fig:Schenker} hints to somewhat similar effect (a cloud of eigenvalues with imaginary parts around 10).
%
\hfill $\blacksquare$
\end{example}

\begin{example}[$\cP\cT$-symmetric phase transitions II] 
	\label{ex:CG}
	$\cP\cT$-symmetric phase transitions were studied in \cite{Caliceti-2014-19} for operators in $L^2(\R)$ with polynomial potentials 
	
	\begin{equation}\label{T.CG}
	\Dti + \frac{x^{2M}}{2M} + \ii g \frac{x^{M-1}}{M-1}, \quad M \in 2 \N,
	\end{equation}
	and the eventual transition of each eigenvalue was established, see \cite[Thm.~1.1]{Caliceti-2014-19} for precise claims. 
	
	For $M \geq 4$, Theorem~\ref{thm:g} used for the stationary point of $Q_2$ at $x_0=0$ yields that spectra of operators \eqref{T.CG} contain asymptotically the eigenvalues 
	\begin{equation}
			\la_{k,g,M}^{(x_0)} = g^{-\frac{2}{M+1}}(\nu_{k,M}+\mathcal{O}(g^{-\frac{2}{M+1}})), \quad g \to + \infty,
		\end{equation}
	where $\nu_{k,M}=(\tfrac{1}{M-1})^{\frac{2}{M+1}}\mu_{k,M}$, and $\{\mu_{k,M}\}_k$ are (positive) eigenvalues of $\Dti + \ii x^{M-1}$, see Example~\ref{ex:H_kappa}. Notice that the leading term of the asymptotic expansion of these eigenvalues is real and also that no such sequence is obtained for $M=2$ when $Q_2(x) = x$ since the spectrum of imaginary Airy operator is empty.
	Nonetheless, the (diverging) non-real eigenvalues found in \cite{Caliceti-2014-19} are clearly visible in Figure~\ref{fig:CG_x4} for $M=2$ and in similar plots for higher $M$. To obtain asymptotics of these we use other stationary points of the potential outside real axis.
	\begin{figure}[htb!]
		\includegraphics[width=0.47\textwidth]{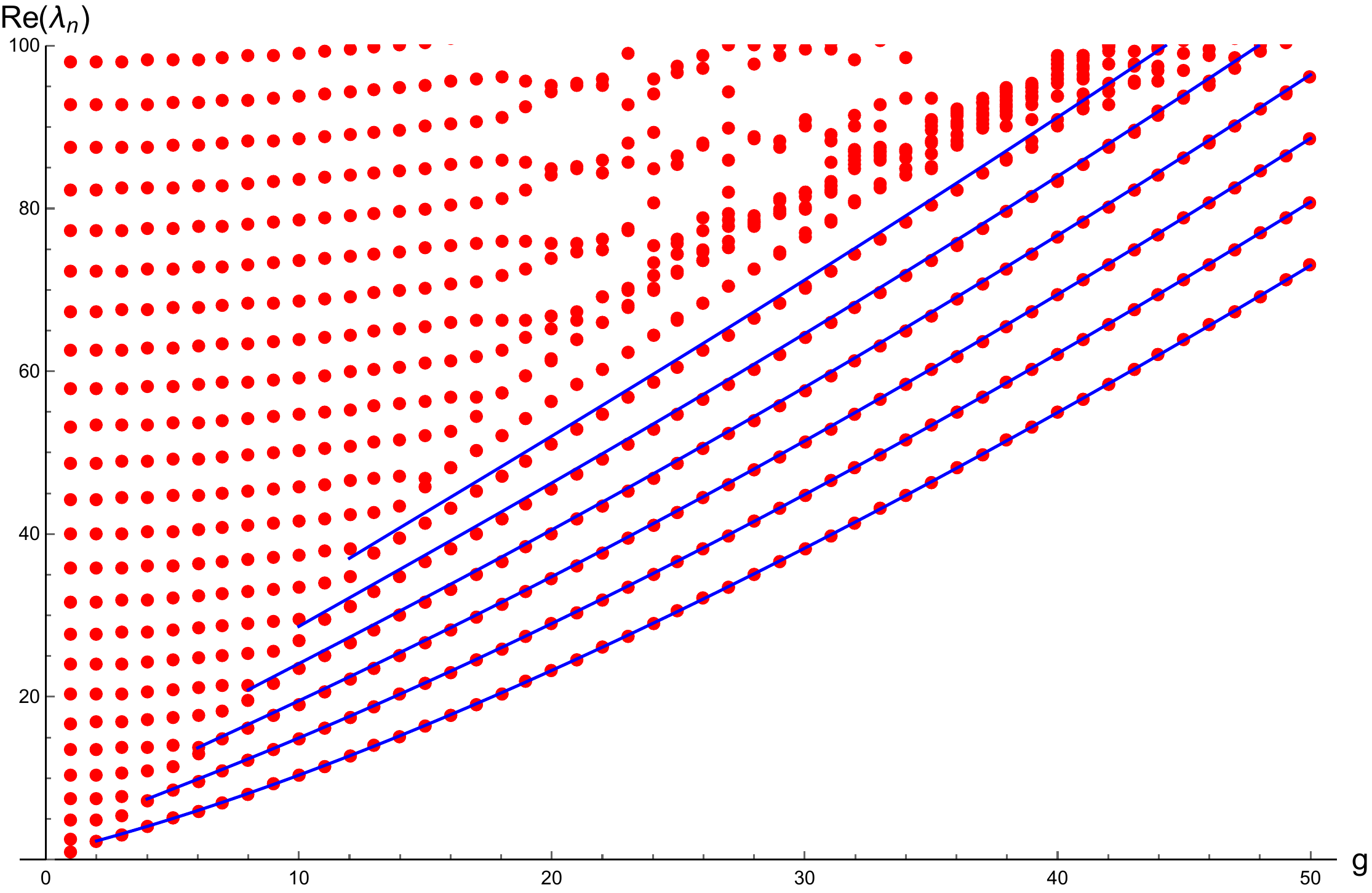}
		\hfill
		\includegraphics[width=0.47\textwidth]{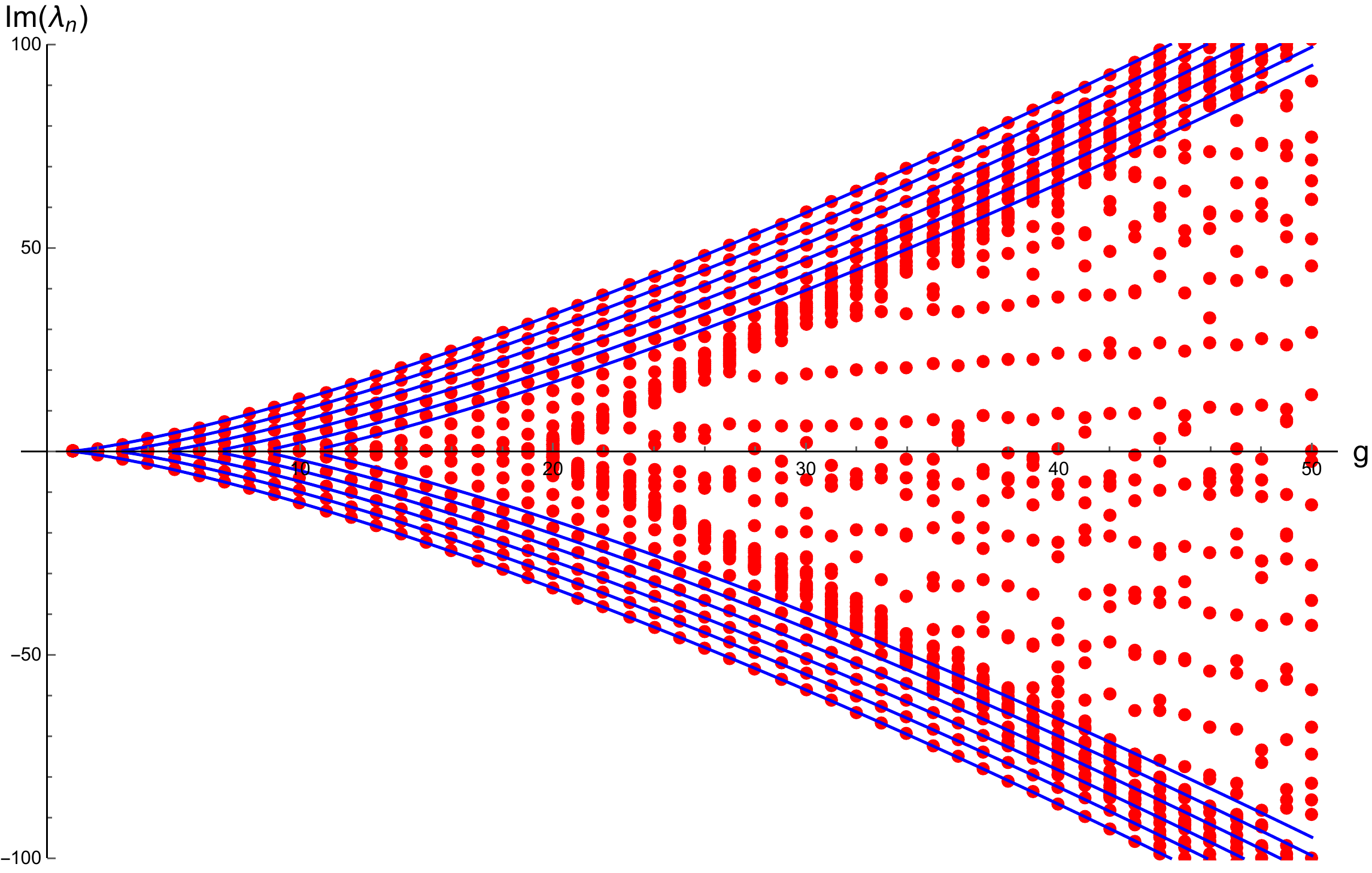}
		\caption{$Q_1(x)=x^4/4$, $Q_2(x)=x$: Real (left) and imaginary (right) part of the eigenvalues (red) of operators $T_g$ in \eqref{Tg.CG} with $g=5,10,\dots,500$. Asymptotic curves (blue) for $\lambda_{k,g}^{(x_2)}$, $\lambda_{k,g}^{(x_3)}$ with $k=1,2,\dots,5$.}
		\label{fig:CG_x4}
	\end{figure}
	
	Consider first a simpler shifted oscillator 
	$\Dti + x^2 + 2 \ii g x$ where Theorem~\ref{thm:g} is not applicable for the stationary point $x_0=0$ directly either. Nevertheless, writing $x^2 + 2 \ii g x =  (x + \ii g)^2 +g^2$ and the complex shift $x \mapsto x - \ii g$, \ie~to the complex stationary point $x_1 = - \ii g$, reveals the well-known diverging eigenvalues $\{2k+1 + g^2\}_{k \in \N_0}$. Notice that the complex shift leaves the spectrum invariant by an argument similar to complex scaling. Namely, the shift $x \mapsto x + \theta$ generates a holomorphic family (in $\theta$) of operators of type A since the operator domains are constant, moreover, for $\theta \in \R$, the spectra stay clearly invariant (such shifts induce a unitary transform). 
	
	For operators \eqref{T.CG}, we first rescale $ x \mapsto g^{2M/(M+1)} x$ to obtain
	\begin{equation}
		%	\label{Tg.CG.1}
		\frac{1}{g^{\frac{2}{M+1}}}\left[\Dti + g^2 \left( \frac{x^{2M}}{2M} + \ii \frac{x^{M-1}}{M-1} \right) \right]
	\end{equation}
	The stationary points of the potential read
	\begin{equation}
		x_0 = 0, \qquad x_k = \e^{\ii \frac{4k-1}{2(M+1)} \pi}, \quad k = 1, \dots, M+1.
	\end{equation}

	In particular for $M=2$, besides $x_0 =0$, which was already covered above, we have $x_1 = \ii$, $x_2= \e^{\ii \frac76 \pi}$ and $x_3= \e^{\ii \frac{11}6 \pi}$. The shift to $x_3$ leads to the operator 
	\begin{equation}\label{Tg.CG}
		T_g = \frac{1}{g^{\frac{2}{3}}} 
		\left(
		\Dti + \frac{g^2}4 \left[  x^2(x+\sqrt3 )^2 - \ii x^2 (2x + 3 \sqrt 3) \right] +\frac{3}4 g^2 \e^{\ii \frac \pi 3 }
		\right)	
	\end{equation}
	which is not directly covered by Theorem~\ref{thm:g} as $g^2$ multiplies the whole potential. Nonetheless, Theorem~\ref{thm:g} can be generalized in a straightforward way if the real part of the potential is non-negative and it yields that eigenvalues
	\begin{equation}
		\lambda_{k,g}^{(x_3)}=\sqrt{\frac 32} g^\frac13 (\nu_k+\BigO_k(g^{-\frac16}))+\frac 34 \e^{\ii\frac{5\pi}{3}}g^\frac43, \quad \lambda_{k,g}^{(x_2)} = \overline{\lambda_{k,g}^{(x_3)}},  \quad g \to + \infty,
	\end{equation}
	where $\nu_k=\e^{\ii \frac{\pi}{6}}(2k+1)$, $k\in\N_0$, lie asymptotically in the spectra of $T_g$, see Figure~\ref{fig:CG_x4} for illustration. The shift to $x_1$ gives the potential with the quadratic term $-3 x^2/2$ which does not correspond to a suitable limit operator.
	
	The situation is more complicated for $M >2$, there are more stationary points and in general the real part of the potential after the shift is not non-negative (although bounded from below). Moreover, numerics suggests that only two stationary points  lead to diverging eigenvalues. Namely the points for $k=\frac{M}{2}+1$ and $k=M+1$, \ie~ $\e^{\ii\frac{2M+3}{2M+2}}$ and $\e^{\ii\frac{4M+3}{2M+2}}$ (the points where the shifted potential has a global extreme of imaginary part).
	\hfill $\blacksquare$
\end{example}

\appendix

\section{appendix}

\begin{lemma}\label{lem:gn.est.app}
Let Assumption~\ref{asm:Q} be satisfied and let $T = -\Delta + Q$ be the Schr\"odinger operator defined as in Section~\ref{ssec:Schr}. Then for each $\eps_1>0$ there exists $C \geq 0$, depending only on $\eps_1$, $\eps_\nabla$ and $M_\nabla$ such that for all $f \in \Dom(T)$
	
\begin{equation}\label{T.norm.est.lem}
	\|T f\|^2 \geq \left( \frac{2-\eps_\nabla(2+\sqrt 2) - \eps_1}{2-\eps_\nabla} \right) (\|\Delta f\|^2 + \|Q f\|^2) - C_{\eps_1,\eps_\nabla,M_\nabla} \|f\|^2,
\end{equation}
\end{lemma}
\begin{proof}
By a standard approximation argument, see \eg~\cite{Krejcirik-2017-221}, it suffices to establish \eqref{T.norm.est.lem} for $f \in \Dom(-\Delta_{\rm D})$ with a bounded support. Integrating by parts we get
\begin{equation}\label{sep.in.1}
\|(-\Delta + Q)f\|^2 \geq \|\Delta f \|^2 + \|Qf\|^2 +   2 \langle \nabla f, (\Re Q) \nabla f \rangle - 2 \langle |\nabla f|, |\nabla Q| |f| \rangle. 
\end{equation}
Employing \eqref{asm:Q}, Cauchy-Schwartz and Young inequalities we get (with $\delta_1>0$)
\begin{equation}\label{nQ.1}
\begin{aligned}
|\langle |\nabla f|, |\nabla Q| |f| \rangle| &\leq \eps_\nabla \langle |Q|^\frac12 |\nabla f|, |f| \rangle  + M_\nabla \langle |\nabla f|, |f| \rangle
\\
&\leq  \eps_\nabla \delta_1 \||Q|^\frac12 \nabla f\|^2 + \frac{\eps_\nabla}{4 \delta_1}\|Qf\|^2 + M_\nabla\|\nabla f\| \|f\|.
\end{aligned}
\end{equation}
Next, integrating by parts and using that $|\nabla |Q|| \leq |\nabla Q|$,
\begin{equation}\label{nQ.2}
\||Q|^\frac12 \nabla f\|^2 = \langle |Q| \nabla f,  \nabla f \rangle \leq |\langle -\Delta f, |Q| f \rangle| + \langle |\nabla f|, |\nabla Q| |f|  \rangle. 
\end{equation}
Thus combining \eqref{nQ.1} and \eqref{nQ.2}, using Young inequality (with $\delta_2>0$) and 
\begin{equation}
\|\nabla f\|^2 = \langle -\Delta f, f \rangle \leq \delta_3 \|\Delta f\|^2 + \frac1{4 \delta_3} \|f\|^2,  	
\end{equation}
we obtain
\begin{equation*}
\begin{aligned}
|\langle |\nabla f|, |\nabla Q| |f| \rangle| &\leq \frac{1}{1 -\delta_1 \eps_\nabla} 
\left(
\delta_1 \eps_\nabla \|\Delta f\| \| Q f\| + \frac{\eps_\nabla}{4 \delta_1} \|Qf\|^2 + M_\nabla\|\nabla f\| \|f\|
\right)
\\ & \leq
\frac{\eps_\nabla \delta_1 \delta_2 + M_\nabla \delta_3 }{1 -\delta_1 \eps_\nabla} \|\Delta f\|^2 + 
\frac{\eps_\nabla}{1 -\delta_1 \eps_\nabla} \left(\frac{\delta_1}{4 \delta_2} + \frac1{4 \delta_1} \right) \|Qf\|^2
\\
& \quad + \frac{M_\nabla}{1 -\delta_1 \eps_\nabla}  \left( \frac{1}{ 4\delta_3} + \frac{1}{ 4} \right) \|f\|^2.
\end{aligned}
\end{equation*}
Inserting the last inequality in \eqref{sep.in.1} and since $\Re Q \geq 0$ by assumption, we get
\begin{equation}
\begin{aligned}		
\|Tf\|^2 &\geq \left( 1 - \frac{2\eps_\nabla \delta_1 \delta_2 }{1 -\delta_1 \eps_\nabla} - \frac{2 M_\nabla \delta_3}{1 -\delta_1 \eps_\nabla}\right)  \|\Delta f\|^2 
 \\ &\quad + 
\left( 1 - \frac{2\eps_\nabla}{1 -\delta_1 \eps_\nabla} \left(\frac{\delta_1}{4 \delta_2} + \frac1{4 \delta_1} \right) \right) \|Qf\|^2 - C \|f\|^2.
\end{aligned}
\end{equation}
Notice that $\delta_3>0$ can be chosen arbitrarily small and it is hidden in $\eps_1$ in \eqref{T.norm.est.lem}. Simple manipulations show that \eqref{T.norm.est.lem} holds if both 
\begin{equation}
\eps_\nabla < \frac{1}{\delta_1(1+2 \delta_2)}, \qquad \eps_\nabla <\frac{2 \delta_1 \delta_2}{\delta_1^2(1+2 \delta_2) + \delta_2}
\end{equation}
are satisfied (the first inequality implies that $1-\delta_1 \eps_\nabla>0$). Equating the right sides of these inequalities, we have $	\delta_1^2 (4 \delta_2^2-1) =\delta_2$ and our goal is to maximize $\delta_1(1+2\delta_2)$. By a simple calculus we obtain the values $\delta_1 =1/2$ and $2\delta_2 = 1 +\sqrt 2$, which yields the constants in \eqref{T.norm.est.lem}
\end{proof}

\begin{proof}[Proof of Theorem~\ref{thm:EF.dec}]
	
The crucial step in the proof is the following generalized weighted coercivity of the form $t$ proved in \cite[Thm.~3.3]{Krejcirik-2017-221}. Let $w \in W^{1,\infty}(\Omega;\R)$, then for every $\alpha \in (0,1]$ and for every $f \in \Dom(t)$ 
\begin{equation}\label{w.coer.aplha}
\begin{aligned}
& \Re t(f,e^{2w}f) + \Im t(f,\Phi e^{2w}f) \geq
\\
&  \qquad (1-\alpha) \|\nabla e^w f\|^2 
+ \int_{\Omega} \left(\widetilde Q_\alpha(x)-\left(1+\frac{2}{\alpha}\right) |\nabla w(x)|^2 \right)|e^{w(x)}f(x)|^2 \, \dd x,
\end{aligned}
\end{equation}
where  $\widetilde Q_\alpha = \Phi^2 + \Re Q - \frac1{2\alpha}|\nabla\Phi|^2$ and $\Phi$ is as in \eqref{Phi.def}. Hence taking $\alpha=1$,we get	
\begin{equation}\label{w.coer}
 \Re t(f,e^{2w}f) + \Im t(f,\Phi e^{2w}f) \geq
 \int_{\Omega} \left(\widetilde Q(x)-3 |\nabla w(x)|^2 \right)|e^{w(x)}f(x)|^2 \, \dd x.
\end{equation}
Inserting $f=\psi$ in \eqref{w.coer} leads to 
\begin{equation*}
(\Re \la + |\Im \la|) \|e^w \psi\|^2 + 2 \|e^w \psi\| \|e^w \psi_0\| 
\geq \int_{\Omega} (\widetilde Q(x)-3|\nabla w(x)|^2)|e^{w(x)} \psi (x)|^2 \, \dd x.
\end{equation*}

	In the next step, we approximate $W$ in a suitable way. To this end, we define
	\begin{equation}
	\eta_n(s) = \begin{cases}
	s, & 0 \leq s \leq n,
	\\
	2n-s, & n \leq s \leq 2n,
	\\
	0, & 2n \leq s;
	\end{cases}
	\end{equation}
	notice that $\|\eta_n'\|_{L^\infty(\R_+)}=1$ and $w_n(x):=\eta_n(W(x))$, $x \in \Omega$, 
	satisfy 
	\begin{equation}
	w_n \leq W, \quad |\nabla w_n| \leq |\nabla W|, \quad w_n \in W^{1,\infty}(\Omega).  
	\end{equation}
	Then using \eqref{Q.W.ineq} and Young's inequality with $\delta>0$ applied to $2 \|e^w \psi\| \|e^w \psi_0\|$, we arrive at (with $\Omega_2:=\Omega\setminus \Omega_1$)
\begin{equation*}\label{W.psi}
	\begin{aligned}
	& \int_{\Omega_1} |e^{w_n(x)} \psi (x)|^2 \, \dd x
	\\ & \ 
	\leq 
	\frac{
		|\Re \la| + |\Im \la| + \||\widetilde Q|+ 3 |\nabla W|^2\|_{L^\infty(\Omega_2)} + \delta
	}{ \delta} \|e^W\|_{L^\infty(\Omega_2)}^2 \|\psi\|^2
%	\\ &  \qquad 
	+ \frac{1}{\delta^2} \|e^W \psi_0\|^2.
	\end{aligned}
	\end{equation*}	
	The claim \eqref{eW.psi} then follows by Fatou's lemma and simple manipulations.
\end{proof}

%\newpage
{\footnotesize
	\bibliographystyle{acm}
	\bibliography{../references}

\begin{thebibliography}{10}

\bibitem{Almog-2008-40}
{\sc Almog, Y.}
\newblock {The Stability of the Normal State of Superconductors in the Presence
  of Electric Currents}.
\newblock {\em SIAM J. Math. Anal. 40\/} (2008), 824--850.

\bibitem{Almog-2018-59}
{\sc Almog, Y., Grebenkov, D.~S., and Helffer, B.}
\newblock {Spectral semi-classical analysis of a complex {S}chr\"{o}dinger
  operator in exterior domains}.
\newblock {\em J. Math. Phys. 59\/} (2018), 041501, 12.

\bibitem{Almog-2015-40}
{\sc Almog, Y., and Helffer, B.}
\newblock {On the spectrum of non-selfadjoint Schr{\"o}dinger operators with
  compact resolvent}.
\newblock {\em Comm. Partial Differential Equations 40\/} (2015), 1441--1466.

\bibitem{Almog-2016-48}
{\sc Almog, Y., and Henry, R.}
\newblock {Spectral analysis of a complex Schr\"{o}dinger operator in the
  semiclassical limit}.
\newblock {\em SIAM J. Math. Anal. 48\/} (2016), 2962--2993.

\bibitem{Baker-2020-61}
{\sc Baker, C., and Mityagin, B.}
\newblock {Non-real eigenvalues of the harmonic oscillator perturbed by an odd,
  two-point interaction}.
\newblock {\em J. Math. Phys. 61\/} (2020), 043505.

\bibitem{Beauchard-2015-21}
{\sc Beauchard, K., Helffer, B., Henry, R., and Robbiano, L.}
\newblock Degenerate parabolic operators of {K}olmogorov type with a geometric
  control condition.
\newblock {\em ESAIM Control Optim. Calc. Var. 21}, 2 (2015), 487--512.

\bibitem{BelHadjAli-2011}
{\sc BelHadjAli, H., Amor, A.~B., and Brasche, J.~F.}
\newblock Large coupling convergence: Overview and new results.
\newblock In {\em Partial Differential Equations and Spectral Theory}. Springer
  Basel, 2011, pp.~73--117.

\bibitem{Boegli-2018-8}
{\sc B\"ogli, S.}
\newblock {Local convergence of spectra and pseudospectra}.
\newblock {\em J. Spectr. Theory 8\/} (2018), 1051--1098.

\bibitem{Boegli-2019-40}
{\sc B\"ogli, S., and Marletta, M.}
\newblock {Essential numerical ranges for linear operator pencils}.
\newblock {\em IMA J. Numer. Anal. 40\/} (2019), 2256--2308.

\bibitem{Boegli-2020-279}
{\sc B\"ogli, S., Marletta, M., and Tretter, C.}
\newblock {The essential numerical range for unbounded linear operators}.
\newblock {\em J. Funct. Anal. 279\/} (2020), 108509.

\bibitem{Boegli-2017-42}
{\sc B{\"o}gli, S., Siegl, P., and Tretter, C.}
\newblock {Approximations of spectra of Schr\"odinger operators with complex
  potential on $\mathbb{R}^d$}.
\newblock {\em Comm. Partial Differential Equations 42\/} (2017), 1001--1041.

\bibitem{Brown-2004-24}
{\sc Brown, B.~M., and Marletta, M.}
\newblock {Spectral inclusion and spectral exactness for PDEs on exterior
  domains}.
\newblock {\em IMA J. Numer. Anal. 24\/} (2004), 21--43.

\bibitem{Caliceti-2014-19}
{\sc Caliceti, E., and Graffi, S.}
\newblock {An existence criterion for the {$\mathcal{PT}$}-symmetric phase
  transition}.
\newblock {\em Discrete Contin. Dyn. Syst. Ser. B 19\/} (2014), 1955--1967.

\bibitem{Constantin-2008-168}
{\sc Constantin, P., Kiselev, A., Ryzhik, L., and Zlato\v{s}, A.}
\newblock Diffusion and mixing in fluid flow.
\newblock {\em Ann. of Math. 168\/} (2008), 643--674.

\bibitem{Davies-1983-9}
{\sc Davies, E.~B.}
\newblock {Some Norm Bounds And Quadratic Form Inequalities For Schr\"odinger
  Operators}.
\newblock {\em J. Operator Theory 9\/} (1983), 147--162.

\bibitem{Davies-1984-12}
{\sc Davies, E.~B.}
\newblock {Some Norm Bounds And Quadratic Form Inequalities For Schr\"odinger
  Operators. II}.
\newblock {\em J. Operator Theory 12\/} (1984), 177--196.

\bibitem{Davies-1999-200}
{\sc Davies, E.~B.}
\newblock {Semi-Classical States for Non-Self-Adjoint Schr\"odinger Operators}.
\newblock {\em Comm. Math. Phys. 200\/} (1999), 35--41.

\bibitem{Dencker-2004-57}
{\sc Dencker, N., Sj{\"o}strand, J., and Zworski, M.}
\newblock {Pseudospectra of semiclassical (pseudo-) differential operators}.
\newblock {\em Commun. Pure Appl. Math. 57\/} (2004), 384--415.

\bibitem{DS2}
{\sc Dunford, N., and Schwartz, J.~T.}
\newblock {\em {Linear Operators, Part 2}}.
\newblock John Wiley \& Sons, Inc., New York, 1988.

\bibitem{EE}
{\sc Edmunds, D.~E., and Evans, W.~D.}
\newblock {\em {Spectral Theory and Differential Operators}}.
\newblock Oxford University Press, New York, 1987.

\bibitem{Evans-1978-80}
{\sc Evans, W.~D., and Zettl, A.}
\newblock {Dirichlet and separation results for {S}chr\"odinger-type
  operators}.
\newblock {\em Proc. Roy. Soc. Edinburgh Sect. A 80\/} (1978), 151--162.

\bibitem{Everitt-1978-79}
{\sc Everitt, W.~N., and Giertz, M.}
\newblock {Inequalities and separation for {S}chr\"odinger type operators in
  {$L_{2}({\bf R}^{n})$}}.
\newblock {\em Proc. Roy. Soc. Edinburgh Sect. A 79\/} (1978), 257--265.

\bibitem{Gallagher-2009-2009}
{\sc Gallagher, I., Gallay, T., and Nier, F.}
\newblock {Spectral Asymptotics for Large Skew-Symmetric Perturbations of the
  Harmonic Oscillator}.
\newblock {\em Int. Math. Res. Not. 2009\/} (2009), 2147--2199.

\bibitem{Guenther-2019-arxiv}
{\sc Guenther, U., and Stefani, F.}
\newblock {IR-truncated $\mathcal{PT}-$symmetric $ix^3$ model and its
  asymptotic spectral scaling graph}.
\newblock arXiv:1901.08526 [math-ph].

\bibitem{Henry-2014}
{\sc Henry, R.}
\newblock {On the semi-classical analysis of Schr\"odinger operators with
  purely imaginary electric potentials in a bounded domain}.
\newblock arXiv:1405.6183, 2014.

\bibitem{Henry-2014-4}
{\sc Henry, R.}
\newblock {Spectral instability for even non-selfadjoint anharmonic
  oscillators}.
\newblock {\em J. Spec. Theory 4\/} (2014), 349--364.

\bibitem{Henry-2014-15}
{\sc Henry, R.}
\newblock {Spectral Projections of the Complex Cubic Oscillator}.
\newblock {\em Ann. Henri Poincar{\'e} 15\/} (2014), 2025--2043.

\bibitem{Kato-1987-111}
{\sc Kato, T.}
\newblock {Variation of discrete spectra}.
\newblock {\em Comm. Math. Phys. 111\/} (1987), 501--504.

\bibitem{Kato-1966}
{\sc Kato, T.}
\newblock {\em {Perturbation theory for linear operators}}.
\newblock Springer-Verlag, Berlin, 1995.

\bibitem{Krejcirik-2017-221}
{\sc Krej{\v{c}}i{\v{r}}{\'i}k, D., Raymond, N., Royer, J., and Siegl, P.}
\newblock {N}on-accretive {S}chr{\"o}dinger operators and exponential decay of
  their eigenfunctions.
\newblock {\em Israel J. Math. 221\/} (2017), 779--802.

\bibitem{Krejcirik-2015-56}
{\sc Krej{\v{c}}i{\v{r}}{\'i}k, D., Siegl, P., Tater, M., and Viola, J.}
\newblock {Pseudospectra in non-Hermitian quantum mechanics}.
\newblock {\em J. Math. Phys. 56\/} (2015), 103513.

\bibitem{Krejcirik-2019-276}
{\sc Krej\v{c}i\v{r}{\'i}k, D., and Siegl, P.}
\newblock {Pseudomodes for Schr\"odinger operators with complex potentials}.
\newblock {\em J. Funct. Anal. 276\/} (2019), 2856--2900.

\bibitem{Mityagin-2015-54}
{\sc Mityagin, B.}
\newblock {The Spectrum of a Harmonic Oscillator Operator Perturbed by Point
  Interactions}.
\newblock {\em Int. J. Theor. Phys. 54\/} (2015), 4068--4085.

\bibitem{MiSiVi-2020}
{\sc Mityagin, B., Siegl, P., and Viola, J.}
\newblock {Concentration of eigenfunctions of Schr\"odinger operators}.
\newblock arXiv:1910.10048v2 [math.SP], 2020.

\bibitem{Mityagin-2015-54a}
{\sc Mityagin, B.~S.}
\newblock The spectrum of a harmonic oscillator operator perturbed by point
  interactions.
\newblock {\em Int. J. Theor. Phys. 54}, 11 (jan 2015), 4068--4085.

\bibitem{Osborn-1975-29}
{\sc Osborn, J.~E.}
\newblock {Spectral approximation for compact operators}.
\newblock {\em Math. Comput. 29\/} (1975), 712--725.

\bibitem{Reed4}
{\sc Reed, M., and Simon, B.}
\newblock {\em {Methods of Modern Mathematical Physics, Vol. 4: Analysis of
  Operators}}.
\newblock Academic Press, New York-London, 1978.

\bibitem{Schenker-2011-18}
{\sc Schenker, J.~H.}
\newblock Estimating complex eigenvalues of non-self adjoint {S}chr\"{o}dinger
  operators via complex dilations.
\newblock {\em Math. Res. Lett. 18\/} (2011), 755--765.

\bibitem{Shin-2002-229}
{\sc Shin, K.~C.}
\newblock {On the Reality of the Eigenvalues for a Class of
  $\mathcal{PT}$-Symmetric Oscillators}.
\newblock {\em Comm. Math. Phys. 229\/} (2002), 543--564.

\bibitem{Simon-2005}
{\sc Simon, B.}
\newblock {\em {Trace ideals and their applications}}, 2nd~ed., vol.~120.
\newblock AMS, Providence, RI, 2005.

\bibitem{Weidmann-2003}
{\sc Weidmann, J.}
\newblock {\em {Lineare Operatoren in Hilbertr\"aumen}}.
\newblock Vieweg+Teubner Verlag, 2003.

\bibitem{Znojil-2001-285}
{\sc Znojil, M.}
\newblock {$\mathcal{PT}$-symmetric square well}.
\newblock {\em Phys. Lett. A 285\/} (2001), 7--10.

\end{thebibliography}
}

\end{document}